\documentclass[11pt, oneside,reqno]{amsart}
\usepackage{amssymb}
\usepackage{amsmath}
\usepackage[normalem]{ulem}
\usepackage{multirow}
\usepackage{array}
\usepackage{graphicx}
\graphicspath{ {./images/} }
\usepackage{xcolor}
\usepackage{mathrsfs}
\usepackage[textwidth=14mm]{todonotes}
\usepackage{cancel}
\usepackage{enumitem}
\allowdisplaybreaks
\usepackage{geometry,mathtools}
\newtheorem{thm}{Theorem}[section]
\newtheorem{prop}[thm]{Proposition}
\newtheorem{cor}[thm]{Corollary}
\newtheorem{lem}[thm]{Lemma}
\theoremstyle{definition}

\newtheorem{rem}[thm]{Remark}
\newtheorem{dfn}[thm]{Definition}

\newtheorem{ex}[thm]{Example}

\reversemarginpar

\newcommand{\wt}{\mbox{\rm wt}\ }

\newcommand{\Der}{\mbox{\rm Der}}

\newcommand{\Ind}{\mbox{\rm Ind}}
\newcommand{\Ker}{\mbox{\rm Ker}}
\newcommand{\rad}{\mbox{\rm rad}}
\newcommand{\tr}{\mbox{\rm tr}}
\begin{document}
\title[Affine and Lattice Structures in Regular $\mathbb{N}$-Graded VOAs]{Affine and Lattice Structures in Regular $\mathbb{N}$-Graded Vertex Operator Algebras with Gorenstein $V_0$}

\author{Gaywalee Yamskulna}\address{Department of Mathematics, Illinois State University, Normal, IL 61790} \email{gyamsku@ilstu.edu } 
\thanks{The author is supported by the AMS-Simons Research Enhancement Grants for PUI Faculty (Award/Grant No. 142069), and AWM-NSF Travel Grant DMS-2015440.}



\begin{abstract} 
We study regular $\mathbb{N}$-graded vertex operator algebras $V=\bigoplus_{n=0}^{\infty}V_n$ whose weight-zero algebra $V_0$ is a nontrivial finite-dimensional local Gorenstein algebra. We aim to determine which structural features of strongly rational vertex operator algebras persist when the CFT-type condition $V_0=\mathbb{C}{\bf 1}$ is removed. 

First, we consider semisimple Lie subalgebras of the left Leibniz algebra $V_1$. Under the condition $\Ker~L(-1)|_{V_0}=\mathbb{C}{\bf 1}$, the product $u_1v$ induces an invariant symmetric bilinear form on every semisimple Lie subalgebra of $V_1$. Under $C_2$-cofiniteness and a nondegeneracy hypothesis, each simple component with nonzero form generates an affine vertex operator algebra at positive integral level and acts integrably on $V$. 

We then study the case in which $V_1$ is solvable. Using the Frobenius structure of the Gorenstein algebra $V_0$, we construct a distinguished nondegenerate subspace $\mathcal{M}$ contained in $V_1$. Under a quasi-primary condition, $\mathcal{M}$ is abelian and generates a Heisenberg vertex operator algebra. Assuming in addition semisimplicity of the relevant Heisenberg zero-mode action, a full-rank integrality condition, and lattice cocycle compatibility, we prove that $V$ contains a conformally embedded lattice vertex operator algebra $V_{K}$, where $K$ is a positive-definite even lattice of rank $\dim_{\mathbb{C}}\mathcal{M}$ and minimum norm at least 4. 

Finally, we study conformally shifted lattice vertex operator algebras as explicit models. These examples distinguish weight-one Lie structure, Lie structure generated by modes, and lattice structure detected by $\mathcal{M}$. They also show that regularity alone does not force semisimplicity of arbitrary Heisenberg zero-mode actions, demonstrating that the additional hypotheses in the lattice theorem reflect genuine structural obstructions.

\end{abstract}
\maketitle

\section{Introduction} 

Much of the classical structure theory of vertex operator algebras concerns
vertex operator algebras $V=\bigoplus_{n\geq 0}V_n$ of CFT-type, for which $V_0=\mathbb{C}{\bf 1}$. In the strongly rational
setting, Dong and Mason proved that the weight-one Lie algebra $V_1$ is
reductive \cite{DM1} and established an integrability theorem showing that
simple Lie subalgebras of $V_1$ generate affine vertex operator algebras at
positive integral levels \cite{DM4}. Mason subsequently developed the
lattice-subalgebra theory for strongly regular vertex operator algebras,
obtaining positive-definite even lattice subalgebras under suitable
hypotheses \cite{M}. These results form an important part of the rigid
low-degree structure theory in the CFT-type setting. The purpose of this
paper is to investigate how much of this structure persists when $V_0$ is
no longer one-dimensional.

The study of $\mathbb{N}$-graded vertex algebras with nontrivial degree-zero
algebra has several origins. The theory of chiral differential operators and
the chiral de Rham complex led to vertex algebroids and systematic
constructions of $\mathbb{N}$-graded vertex algebras generated in degrees zero and
one \cite{GMS, MSV}. In this setting $V_0$ is a commutative associative
algebra, whereas $V_1$ naturally carries a Leibniz-algebra structure. Li
and the author subsequently studied vertex algebras associated with vertex
algebroids and their graded modules; twisted modules were treated in
subsequent work \cite{LY1,LY2}.

A further indication of the breadth of vertex-algebra methods comes from the
representation theory of toroidal Lie algebras, the multi-loop analogs of
affine Kac--Moody algebras. Vertex-operator constructions for toroidal Lie
algebras led to large classes of representations with finite-dimensional
weight spaces \cite{BB}. Berman, Billig, and Szmigielski subsequently gave
explicit vertex operator algebra constructions of irreducible toroidal
modules \cite{BBS}, and Billig showed that representations of toroidal
extended affine Lie algebras are controlled by tensor products involving
lattice, affine, and Virasoro vertex operator algebras \cite{Bi}. A
systematic theory of toroidal vertex algebras and their modules was later
developed by Li, Tan, and Wang, who associated toroidal vertex algebras
directly to toroidal Lie algebras \cite{LTW}. These results demonstrate that
graded vertex-algebra structures arise naturally in the representation
theory of higher-dimensional analogs of affine Lie algebras and provide
additional motivation for understanding vertex algebras beyond the
traditional CFT-type framework.

A second line of development directly relevant to the present paper was
initiated by Dong and Mason. Their theory of local and semilocal vertex
operator algebras relates indecomposability and locality to the algebraic
structure of $V_0$ \cite{DM5}. Their shifted vertex operator algebras
provide an especially important source of regular $\mathbb{N}$-graded examples
beyond the usual CFT-type framework: a conformal shift changes the grading
while leaving the underlying vertex algebra unchanged, and regularity may
persist even when standard CFT-type and self-contragredience properties do
not \cite{DM2}.

These ideas motivated further study of the interaction between $V_0$ and
$V_1$. Mason and the author showed, under suitable hypotheses, that Levi
factors of the Leibniz algebra $V_1$ generate affine Kac--Moody vertex
operator subalgebras and that shifted self-contragredient theories can endow
$V_0$ with a de Rham-type duality structure \cite{MY}. Jitjankarn and the
author studied indecomposable non-simple $\mathbb{N}$-graded vertex algebras and
$C_2$-cofinite constructions associated with vertex algebroids
\cite{JY1,JY2}, while Bui and the author developed related constructions
from simple Leibniz algebras \cite{BY}. More recently, Keene, Soltermann,
and the author studied $\mathbb{N}$-graded vertex algebras with Gorenstein
$V_0$, emphasizing the roles of Frobenius duality, the socle, invariant
bilinear forms, and the Leibniz structure of $V_1$ \cite{KSY}.

There are complementary motivations for working beyond the traditional
CFT-type setting. Barron, Batistelli, Orosz Hunziker, Pedi\'c Tomi\'c, and
the author studied rationality for $\mathbb{C}$-graded vertex algebras and
applications to Weyl vertex algebras under conformal flow \cite{BBOHTY}.
McRae's recent work demonstrates strong rigidity consequences of
$C_2$-cofiniteness, self-contragredience, and tensor-category hypotheses
for rationality \cite{McR}. From the viewpoint of axiomatic conformal field
theory, Carpi, Raymond, Tanimoto, and Tener established an equivalence
between suitably generated M\"obius vertex algebras and non-unitary
M\"obius-covariant Wightman conformal field theories without assuming
simplicity or finite-dimensional conformal-weight spaces \cite{CRTT}.
Together, these developments reinforce the importance of understanding
graded vertex algebras outside the classical strongly rational framework.

In the present paper we study regular $\mathbb{N}$-graded vertex operator
algebras $V=\bigoplus_{n=0}^{\infty}V_n$ for which $V_0$ is a nontrivial finite-dimensional local Gorenstein
algebra. The passage from $V_0=\mathbb{C}{\bf 1}$ to a nontrivial Gorenstein
algebra changes the low-degree structure substantially: $V_1$ is
generally a left Leibniz algebra rather than a Lie algebra, and $u_1v$ is
$V_0$-valued rather than scalar-valued. Nevertheless, the Gorenstein
structure, together with regularity and suitable conditions on the
translation operator, provides enough rigidity to recover significant
affine, Heisenberg, and lattice structure. Our first results concern semisimple Lie subalgebras of $V_1$. Under the
condition $\Ker~ L(-1)|_{V_0}=\mathbb{C}{\bf 1},$ the product $u_1v$ induces an invariant symmetric scalar-valued bilinear
form on every semisimple Lie subalgebra of $V_1$. Under
$C_2$-cofiniteness and a nondegeneracy condition, each simple component
with nonzero form generates a rational affine vertex operator algebra at
positive integral level and acts integrably on $V$. Thus an important
part of the strongly rational affine-subalgebra theory persists beyond
CFT-type.

We next consider the case in which $V_1$ is solvable. Using the Frobenius
structure of the Gorenstein algebra $V_0$, we construct a distinguished
nondegenerate subspace $\mathcal{M}\subseteq V_1$. Under a quasi-primary
condition, $\mathcal{M}$ is abelian and generates a Heisenberg vertex operator
algebra. Assuming in addition the semisimplicity condition
$(\Omega\text{-SS})$, a full-rank integrality condition $(\mathrm{IL})$,
and lattice cocycle compatibility $(\mathrm{LC})$, we prove that $V$ contains a conformally embedded lattice vertex operator algebra $V_K$,
where $K$ is a positive-definite even lattice satisfying $\operatorname{rank}K=\dim\mathcal{M}$ and $\min\{\langle\alpha,\alpha\rangle:0\neq\alpha\in K\}\geq 4$. The minimum-norm bound follows from solvability of $V_1$: a norm-two
vector would produce an $\mathfrak{sl}_2$-subalgebra in $V_1$.

Finally, conformally shifted lattice vertex operator algebras provide
explicit models showing that weight-one Lie structure, mode-theoretic Lie
structure, and the lattice structure detected by $\mathcal{M}$ need not
coincide. They also show that regularity alone does not force semisimplicity
of arbitrary Heisenberg zero-mode actions. Thus the additional hypotheses
in the lattice theorem represent genuine structural restrictions rather
than consequences of regularity.

The paper is organized as follows. Section~2 recalls the required
background. Section~3 studies semisimple Lie subalgebras of $V_1$ and the
resulting affine vertex operator algebras. Section~4 develops the
Gorenstein, Heisenberg, and lattice theory in the solvable case. Section~5
studies conformal shifts and examples, including hidden
$\mathfrak{sl}_2$-actions and nonsemisimple Heisenberg zero-mode actions.

\noindent\textbf{Acknowledgments}

The author gratefully acknowledges research support from the AMS–Simons Research Enhancement Grants for Primarily Undergraduate Institution (PUI) Faculty and travel support from the AWM–NSF Travel Grant Program.

The author thanks the organizers of the workshop Recent Developments in Logarithmic Conformal Field Theory at the Banff International Research Station (BIRS), Banff, Canada; the conference CFT: Algebraic, Topological and Probabilistic Approaches in Conformal Field Theory at Institut Pascal, Orsay, France; the Special Session on Geometry, Representation Theory and Noncommutative Algebra at the AMS Sectional Meeting in St. Louis; the Algebra Seminar at the University of California, Santa Cruz; the Math CU Seminar, Chulalongkorn University, Thailand and the department of Mathematics at Srinakarinwirot University, Thailand, for their hospitality and for opportunities to present and discuss portions of this work.

The author is especially grateful to Chongying Dong and Geoffrey Mason for many fruitful conversations during a visit to the University of California, Santa Cruz. Their insightful questions and comments motivated the development of several examples in Section 5 and helped clarify the phenomena those examples illustrate.

\section{Preliminaries}

We use the standard definition and notation for vertex algebra $V$ as in \cite{B, FLM1, LLi}. Thus for $v\in V$, $Y(v,x)=\sum_{n\in\mathbb{Z}}v_nx^{-n-1}$ and the vacuum is denoted by ${\bf 1}$. We write $D(v)=v_{-2}{\bf 1}$; then $Y(v,x){\bf 1}=e^{xD}v$ and $[D,v_n]=-nv_{n-1}$. Also, we use the standard notion of an ideal of a vertex algebra; see \cite{LLi}. 

Similarly, we use the standard definition of a vertex operator algebra $(V,Y, {\bf 1},\omega)$ as in \cite{FLM1, LLi} with $Y(\omega, x)=\sum_{n\in\mathbb{Z}}L(n)x^{-n-2}$ and $L(-1)=D$.

\begin{dfn}\cite{Z} Let $(V,Y(~,~),{\bf 1},\omega)$ be a vertex operator algebra. We define a second vertex operator algebra $(V, Y[~,~], {\bf 1}, \widetilde{\omega})$ associated with $V$ by defining the Virasoro vector $$\widetilde{\omega}=\omega-\frac{c}{24},\text{ and }$$ the vertex operators $Y[v,x]$ are defined for homogeneous $v\in (V,Y(~,~),{\bf 1},\omega)$ via the equality 
$$Y[v,x]=Y(v,e^x-1)e^{x\wt(v)}=\sum_{n\in\mathbb{Z}}v[n]x^{-n-1}.$$ In addition,  we write $$Y[\widetilde{\omega},x]=\sum_{n\in\mathbb{Z}}L[n]x^{-n-2},$$ and $V=\bigoplus_{n\in\mathbb{Z}}V_{[n]}\text{ where }V_{[n]}=\{v\in V~|~L[0]v=nv\}.$ For $v\in V_{[n]}$, we denote its conformal weight by $\wt[v]$.
\end{dfn}
\begin{rem}\cite{Z}
We will only use the following square-bracket Virasoro identities and the induced filtration comparison.
\begin{enumerate} 
\item Using the fact that $Y[\widetilde{\omega},x]=\sum_{n\in\mathbb{Z}}L[n]x^{-n-2}$, we have  
    \begin{eqnarray*}
        L[-2]&=&\omega[-1]-\frac{c}{24},\\
        L[-1]&=&L(-1)+L(0)\\
        L[0]&=&L(0)+\sum_{n=1}^{\infty}\frac{(-1)^{n-1}}{n(n+1)}L(n).
    \end{eqnarray*}
    \item  It is worth mentioning that in general, $V_n$ and $V_{[n]}$ are distinct, and for each $N$, we have $\bigoplus_{n\leq N}V_n=\bigoplus_{n\leq N}V_{[n]}.$
 \end{enumerate}   
\end{rem}

We use the standard notions of weak, admissible, and ordinary $V$-modules from \cite{DLM0}. An ordinary $V$-module is $L(0)$-graded with finite-dimensional homogeneous subspaces and bounded below. 

A vertex operator algebra is \emph{rational} if every admissible $V$-module is completely reducible, and it is \emph{regular} if every weak $V$-module is a direct sum of simple ordinary $V$-modules. \cite{DLM}

\begin{prop}\cite{DM} If $M$ is a simple ordinary $V$-module, then $M=\bigoplus_{n\in\mathbb{N}}M_{n+\lambda}$ for some $\lambda\in\mathbb{C}$ such that $M_{\lambda}\neq 0$. This number $\lambda$ is called the \emph{conformal weight} of $M$.
\end{prop}

\begin{prop}\cite{DLM} Suppose that $V$ is a rational vertex operator algebra. Then the following statements hold:
\begin{enumerate}
    \item Every simple admissible $V$-module is an ordinary $V$-module.
    \item $V$ has only finitely many isomorphism classes of simple admissible $V$-modules.
\end{enumerate}
\end{prop}

\begin{dfn}\cite{FHL, Z} Let $V$ be a vertex operator algebra. Then 
\begin{enumerate}
    \item $V$ is {\em finitely generated} if there exists $n\in\mathbb{Z}$ so that $U=\oplus_{m\leq n}V_m$ generates $V$ in the sense that $$ V=span\{u^1_{i_1} \cdots u^s_{i_s} u~| ~u^j,u\in U,s\in\{0,1,2,...\},i_j\in\mathbb{Z}\}.$$
\item For a $V$-module $M$, we define $C_2(M)=Span\{v_{-2}m~|~v\in V, m\in M\}.$ We say that $M$ satisfies condition $C_2$ in case $C_2(M)$ has finite co-dimension in $M$.
\end{enumerate}
\end{dfn}

\begin{prop}\label{Info about regularity} \cite{DLM1, Li1} \ \ 

\begin{enumerate} 
\item If $V$ is regular then $V$ satisfies the $C_2$-cofiniteness condition and it is finitely generated. In addition, $V$ is rational.
\item Suppose that $V$ is a rational vertex operator algebra which satisfies Condition $C_2$. Then the central charge of $V$ and the conformal weight of each simple $V$-module are rational numbers.
\end{enumerate}
\end{prop}
\begin{dfn}\cite{Z} For any $V$-module, we define \emph{the 1-point function $Z_{M}(u,\tau)$ associated to $M$} as follows: for $u\in V_k$,
\begin{equation}\label{correl} Z_M(u,\tau)=tr_Mo(u)q^{L(0)-c/24}.\end{equation} 
When $u$ is the vacuum vector ${\bf 1}$, we have $Z_M(\tau)=Z_{M}({\bf 1},\tau)$. Here, $o(u)=u_{k-1}$, $\tau$ denotes an element in the complex upper half-plane $\mathbb{H}$ and $q=e^{2\pi i\tau}$. 
\end{dfn} 
\begin{prop}\cite{DLM1,Z} Let $V$ be a rational vertex operator algebra that satisfies the $C_2$-cofiniteness condition. Assume that $\{M^1,M^2,...,M^r\}$ is the set of all inequivalent simple $V$-modules. Then the functions (\ref{correl}) are holomorphic in $\mathbb{H}$ and the following holds: for any homogeneous vector $u\in V_{[k]}$ and for any matrix $\gamma=\left[\begin{array}{cc}a&b\\c&d\end{array}\right]$ in the modular group $SL(2,\mathbb{Z})$, there are scalars $\rho_{i,j}(\gamma)$, $1\leq i,j\leq r$ independent of $u$ and $\tau$, and an equality 
    \begin{equation}\label{gamma matrix} Z_{M^i}(u,\frac{a\tau+b}{c\tau+d})=(c\tau+d)^k\sum_{j=1}^r\rho_{i,j}(\gamma)Z_{M^j}(u,\tau).\end{equation}
    In addition, \begin{equation}\label{ouov}tr_{M^i}o(u)o(v)q^{L(0)-c/24}=Z_{M^i}(u[-1]v,\tau)-\sum_{k\geq 1}E_{2k}(\tau)Z_{M^i}(u[2k-1]v,\tau).\end{equation} Here, $E_{2k}=-\frac{B_{2k}}{2k!}+\frac{2}{(2k-1)!}\sum_{n=1}^{\infty}\sigma_{2k-1}(n)q^n$ where $\sigma_k(n)$ is the sum of the $k$-powers of the divisors of $n$ and $B_{2k}$ a Bernoulli number. Note that $E_{2k}(\tau)$ is a holomorphic modular form on $SL(2,\mathbb{Z})$ if $k>1$. However, $E_2(\tau)$ is not modular, and its transformation with respect to the $S$ matrix is as follows: \begin{equation}E_2(-1/\tau)=\tau^2 E_2(\tau)-\frac{\tau}{2\pi i}.\end{equation}
\end{prop}
\subsection{Examples of vertex operator algebras and their modules}
\begin{ex} \textbf{Vertex operator algebras and their modules associated to Heisenberg algebras}

Let $\mathfrak{h}$ be a $d$-dimensional vector space equipped with a nondegenerate symmetric bilinear form $\langle\cdot,\cdot\rangle$. We consider $\mathfrak{h}$ as an abelian Lie algebra with $\langle\cdot,\cdot\rangle$ as an invariant symmetric bilinear form. We have the affine Lie algebra $\hat{\mathfrak{h}}=\mathfrak{h}\otimes \mathbb{C}[t,t^{-1}]\oplus \mathbb{C}{\bf k},$ with the Lie bracket relations $$[{\bf k},\hat{\mathfrak{h}}]=0,~{[a\otimes t^m,b\otimes t^n]}=\langle a,b\rangle m\delta_{m+n,0}{\bf k}$$
for $a,b\in\mathfrak{h}$, and $m,n\in\mathbb{Z}$. The affine Lie algebra $\hat{\mathfrak{h}}$ is a $\mathbb{Z}$-graded Lie algebra with $\hat{\mathfrak{h}}=\coprod_{n\in\mathbb{Z}}\hat{\mathfrak{h}}_{(n)}$, where $\hat{\mathfrak{h}}_{(0)}=\mathfrak{h}\oplus\mathbb{C}{\bf k}$ and $\hat{\mathfrak{h}}_{(n)}=\mathfrak{h}\otimes t^{-n}$ for $n\neq 0$. For convenience, we set $h(m)$ for $h\otimes t^m$.

We set $\hat{\mathfrak{h}}^+=\mathfrak{h}\otimes t\mathbb{C}[t]$ and $\hat{\mathfrak{h}}^-=\mathfrak{h}\otimes t^{-1}\mathbb{C}[t^{-1}]$. The subalgebra $\hat{\mathfrak{h}}_{\mathbb{Z}}=\hat{\mathfrak{h}}^+\oplus\hat{\mathfrak{h}}^-\oplus\mathbb{C}{\bf k}$ of $\hat{\mathfrak{h}}$ is a Heisenberg algebra. 

Let $l$ be a nonzero complex number. Now, we consider the induced irreducible $\hat{\mathfrak{h}}$-module, irreducible even under $\hat{\mathfrak{h}}_{\mathbb{Z}}$, $$M(l)=U(\hat{\mathfrak{h}})\otimes_{U(\mathfrak{h}\otimes \mathbb{C}[t]\oplus\mathbb{C}{\bf k}}\mathbb{C}\cong S(\hat{\mathfrak{h}}^-)~\text{(linearly), }$$ 
$\mathfrak{h}\otimes \mathbb{C}[t]$ acting trivially on $\mathbb{C}$ and ${\bf k}$ acting as $l$. An $\hat{\mathfrak{h}}$-module on which ${\bf k}$ acts as $l$ is called a {\em level} $l$ module. Then $M(l)$ is of level $l$. 

It is worth mentioning that a level $l$ module for the affine Lie algebra $\hat{\mathfrak{h}}$ with respect to $\langle\cdot,\cdot\rangle$ amounts to a level 1 module for the affine Lie algebra $\hat{\mathfrak{h}}$ with respect to $l\langle\cdot,\cdot\rangle$.

\begin{prop}\cite{FLM1, LLi} Let $l$ be any nonzero complex number. Then $M(l)=\bigoplus_{n=0}^{\infty} M(1)_n$ is a vertex operator algebra of central charge $d=\dim\mathfrak{h}$ with the Virasoro vector given by $$\omega_{\mathfrak{h}}=\frac{1}{2l}\sum_{i=1}^d\beta_i(-1)\beta_i(-1){\bf 1}\in M(l)_2.$$ Here, let $\{\beta_1,...,\beta_d\}$ be an orthonormal basis of $\mathfrak{h}$ with respect to the bilinear form $\langle\cdot,\cdot\rangle$. The weight grading on $M(l)$ is given by $L_{\mathfrak{h}}(0)$-eigenvalues, $\mathfrak{h}=M(l)_{1}$ which generates $M(l)$ as a vertex algebra. Here $Y(\omega_{\mathfrak{h}},x)=\sum_{n\in\mathbb{Z}}L_{\mathfrak{h}}(n)x^{-n-2}$. 
\end{prop}

Let $l\in\mathbb{C}^{\times}$ and let $W$ be a weak $M(1)$-module or equivalently, a level $l$ affine $\hat{\mathfrak{h}}$-module such that for every $w\in W$, $(\mathfrak{h}\otimes t^n)w=0$ for $n$ sufficiently large. Let $a\in W$ be a highest weight vector of $W$ of weight $\alpha\in\mathfrak{h}$, i.e., $h(n)a=\langle h,\alpha\rangle\delta_{n,0}a$ for $h\in\mathfrak{h}$, $n\geq 0$. Then $L_{\mathfrak{h}}(0)a=\frac{\langle\alpha,\alpha\rangle}{2l}a$. $L_{\mathfrak{h}}(-1)a=\frac{1}{l}\alpha(-1)a.$ 

Now, using the fact that the form $\langle\cdot,\cdot\rangle$ is nondegenerate, we may identify $\mathfrak{h}$ with its dual space $\mathfrak{h}^*$. Let $\alpha\in\mathfrak{h}^*$. We let $\mathbb{C}_{\alpha}$ be the one-dimensional $\mathfrak{h}$-module with $h\in\mathfrak{h}$ acting as the scalar $\langle h,\alpha\rangle$. We set $M(l,\alpha)=\Ind_{\mathfrak{h}}^{\hat{\mathfrak{h}}}\mathbb{C}_{\alpha}=U(\hat{\mathfrak{h}}_+)\otimes_{\mathbb{C}}\mathbb{C}_{\alpha}=S(\hat{\mathfrak{h}}_+)\otimes_{\mathbb{C}}\mathbb{C}_{\alpha},$ an $\hat{\mathfrak{h}}$-module of level $l$. 
\begin{prop}\cite{FLM1, LLi} 
For any $\alpha\in\mathfrak{h}(=\mathfrak{h}^*)$, $M(l,\alpha)=\bigoplus_{n=0}^{\infty}M(l,\alpha)_{n+\frac{1}{2l}\langle\alpha,\alpha\rangle}$, with $M(l,\alpha)_{\frac{1}{2l}\langle\alpha,\alpha\rangle}=\mathbb{C}\alpha$ is an irreducible $M(l,0)$-module with lowest weight $\frac{1}{2l}\langle\alpha,\alpha\rangle$. Moreover, the modules $M(l,\alpha)$ for $\alpha\in\mathfrak{h}$ exhaust the irreducible $M(l,0)$-modules up to equivalence.  
    
\end{prop}
\end{ex} 
\begin{ex}\textbf{Vertex operator algebras and their modules associated to even lattices} 

Fix a finitely generated free abelian group $L$. We are interested in groups $\hat{L}$ which are central extensions of $L$ by $\mathbb{Z}_2$. So, there is a short exact sequence of group $1\rightarrow\{\pm 1\}\rightarrow\hat{L}\stackrel{-}{\rightarrow}L\rightarrow 0,$ and $\hat{L}$ can be identified with $L\times\{\pm 1\}$ as a set, with multiplication $(\alpha, e)(\beta,f)=(\alpha+\beta,\varepsilon(\alpha,\beta)ef)$ $(\alpha,\beta\in L, e,f\in\{\pm 1\})$, where $\varepsilon:L\times L\rightarrow\{\pm 1\}.$ We may and shall take $\varepsilon$ to be bimultiplicative, i.e., \begin{eqnarray*}\varepsilon(\alpha+\beta,\gamma)&=&\varepsilon(\alpha,\gamma)\varepsilon(\beta,\gamma),\\ \varepsilon(\alpha,\beta+\gamma)&=&\varepsilon(\alpha,\beta)\varepsilon(\alpha,\gamma).\end{eqnarray*} This ensures that $\varepsilon\in Z^2(L,\{\pm 1\})$ is a 2-cocycle and that multiplication in $L$ is associative. In particular we note that $\varepsilon(\alpha,0)=\varepsilon(0,\alpha)=1$ and $\varepsilon(\alpha,\beta)=\varepsilon(\alpha,-\beta)=\varepsilon(-\alpha,\beta)$. If $L$ is a positive definite integral lattice with bilinear form $(~,~)$ then we may further choose $\varepsilon$ so that it satisfies \begin{eqnarray*}
\varepsilon(\alpha,\beta)\varepsilon(\beta,\alpha)&=&(-1)^{(\alpha,\beta)+(\alpha,\alpha)(\beta,\beta)},\\ \varepsilon(\alpha,\alpha)&=&(-1)^{((\alpha,\alpha)+(\alpha,\alpha)^2)/2}. 
\end{eqnarray*}

Now assume $L$ is a positive-definite even lattice equipped with a bilinear form $(~,~)$, with $\varepsilon$ as defined above. The twisted group algebra $\mathbb{C}^{\varepsilon}[L]$ has a basis $\{e^{\alpha}~|~\alpha\in L\}$ and multiplication $e^{\alpha}e^{\beta}=\varepsilon(\alpha,\beta)e^{\alpha+\beta}$ $(\alpha,\beta\in L)$. It is $\mathbb{Z}$-graded by $\wt(e^{\alpha})=\frac{1}{2}(\alpha,\alpha)$. 

There are Lie algebras $\mathfrak{h}:=\mathbb{C}\otimes_{\mathbb{Z}}L$, $\hat{\mathfrak{h}}:=\mathfrak{h}\otimes \mathbb{C}[t,t^{-1}]\oplus\mathbb{C}c$, $\hat{\mathfrak{h}}^+:=\mathfrak{h}\otimes t\mathbb{C}[t]$, $\hat{\mathfrak{h}}^-=\mathfrak{h}\otimes t^{-1}\mathbb{C}[t^{-1}]$, with brackets $[x\otimes t^m,y\otimes t^n]=(x,y)m\delta_{m+n,0}c$, $[c,\hat{\mathfrak{h}}]=0$, and an induced $\hat{\mathfrak{h}}$-module $M(1)=U(\hat{\mathfrak{h}})\otimes_{U(\mathfrak{h}\otimes\mathbb{C}[t]\oplus\mathbb{C}c)}\mathbb{C}\cong S(\hat{\mathfrak{h}}^-)$ (linearly), $\mathfrak{h}\otimes\mathbb{C}[t]$ acting trivially on $\mathbb{C}$ and $c$ acting as 1. Fock space for the lattice theory is $$V_L=M(1)\otimes \mathbb{C}^{\varepsilon}[L] \text{ (linearly)}$$ with the usual tensor product grading. The Virasoro vector is $\omega_L=\frac{1}{2}\sum_i (h_i(-1))^2{\bf 1},$ the sum ranging over any orthonormal basis $\{h_i~|~i\in I\}$ of $\mathfrak{h}$. 

For $\alpha\in L$, we write $\alpha(n):=\alpha\otimes t^n$, $\alpha(z):=\sum_{n\in\mathbb{Z}}\alpha(n)z^{-n-1}$, $z^{\alpha}:e^{\beta}\mapsto z^{(\alpha,\beta)}e^{\beta}$, and set 
$$Y(e^{\alpha},z):=\exp\left(\sum_{m=1}^{\infty}\alpha(-m)\frac{z^m}{m}\right)\exp\left(-\sum_{m=1}^{\infty}\alpha(m)\frac{z^{-m}}{m}\right)e^{\alpha}z^{\alpha},$$ and for $v=\alpha_1(-n_1)\cdots \alpha_k(-n_k)\otimes e^{\alpha}\in V_L$ $(n_i\geq 1)$ set 
$$Y(v,z):=\left(\frac{1}{(n_1-1)!}\left(\frac{d}{dz}\right)^{n_1-1}\alpha_1(z)\right)....\left(\frac{1}{(n_k-1)!}\left(\frac{d}{dz}\right)^{n_k-1}\alpha_k(z)\right)Y(e^{\alpha},z):,$$ with the usual normal ordering conventions. 

For $\gamma,\rho\in L$, we have 
$$e^{\gamma}(n)e^{\rho}=\begin{cases}0&\text{if }(\gamma,\rho)\geq -n\\
\varepsilon(\gamma,\rho)e^{\gamma+\rho}&\text{if }(\gamma,\rho)=-n-1\\
\varepsilon(\gamma,\rho)\gamma(-1)e^{\gamma+\rho}&\text{if }(\gamma,\rho)=-n-2\\\frac{1}{2}\varepsilon(\gamma,\rho)(\gamma(-2)e^{\gamma+\rho}+\gamma(-1)^2e^{\gamma+\rho})&\text{if }(\gamma,\rho)=-n-3.\end{cases}$$
\end{ex}

The dual lattice $L^{\circ}$ of $L$ is defined to be $L^{\circ}=\{\beta\in\mathfrak{h}~|~\langle \beta, L\rangle\subseteq \mathbb{Z}\}.$ Then $L^{\circ}$ is a rational lattice whose rank is equal to the rank of $L$. In addition, $L^{\circ}=\cup_{i\in L^{\circ}/L}(L+\lambda_i)$ is a coset decomposition such that $\lambda_0=0$.

\begin{prop}\label{VVLiso}\cite{B, D, DL, DLM0, DM1, FLM1} 

\begin{enumerate}
\item Let $L$ be a positive-definite even lattice. Then $V_L$ is a vertex operator algebra with a Virasoro vector $\omega_L=\frac{1}{2}\sum_{i=1}^d\beta_i(-1)\beta_i(-1){\bf 1}\in M(1).$ Here, let $\{\beta_1,...,\beta_d\}$ be an orthonormal basis of $\mathfrak{h}$ with respect to the bilinear form $\langle\cdot,\cdot\rangle$. 
\item Any weak $V_L$-module is completely reducible and any simple weak $V_L$-module is isomorphic to $V_{L+\beta}$ for some $\beta$ in the dual lattice of $L$. In particular, $V_L$ is regular. 
\item Let $V$ be a simple vertex operator algebra that has a subalgebra isomorphic to $M(1)$. Assume that $V$ is isomorphic to $V_L$ as $M(1)$-modules for some even positive definite lattice $L$. Then $V$ and $V_L$ are isomorphic vertex operator algebras.
    \end{enumerate}
\end{prop}

\begin{prop}\cite{DM2} Let $h\in L^{\circ}$. We set $\omega_h=\omega_L+h(-2){\bf 1}$. Then 
\begin{enumerate} 
    \item $V_{L,h}=(V_L,Y,{\bf 1},\omega_h)$ is a vertex operator algebra.
    \item The categories of weak, admissible, and ordinary $V_L$-modules are equivalent to the categories of weak, admissible, and ordinary $V_{L,h}$-modules, respectively. In particular, $V_{L,h}$ is a regular vertex operator algebra. 
    \item The shifted vertex operator algebra $V_{L,h}$ is self-dual if and only if $2h\in L$.
\end{enumerate}
    
\end{prop}

\begin{ex}\textbf{Affine vertex operator algebras and their modules}

Let $\mathfrak{g}$ be a finite-dimensional simple Lie algebra, $\mathfrak{h}\subseteq\mathfrak{g}$ a Cartan subalgebra, $\Phi$ the associated root system with simple roots $\Delta$ and $\langle\cdot,\cdot\rangle$ the nondegenerate symmetric invariant bilinear form on $\mathfrak{g}$ normalized so that the longest positive root $\theta\in\Phi$ satisfies $\langle\theta,\theta\rangle=2$. The corresponding affine Kac-Moody algebra $\hat{\mathfrak{g}}$ is defined to be $$\hat{\mathfrak{g}}=\mathfrak{g}\otimes\mathbb{C}[t,t^{-1}]\oplus \mathbb{C}K$$ where $K$ is central and $$[a(m),b(n)]=[a,b](m+n)+m\delta_{m+n,0}\langle a,b\rangle K,~(a,b\in\mathfrak{g}, m,n\in\mathbb{Z}).$$ Here $a(m)=a\otimes t^m$.

Highest weight irreducible $\hat{\mathfrak{g}}$-modules are parametrized by linear forms on $\mathfrak{h}\oplus\mathbb{C}K$. These in turn correspond to pairs $(\lambda,k)$ where $\lambda\in\mathfrak{h}^*$ is a weight and $K\mapsto k$. The corresponding $\hat{\mathfrak{g}}$-module is denoted by $L(k,\lambda)$. Note that $K$ acts on $L(k,\lambda)$ as multiplication by $k$. One knows that $L(k,\lambda)$ is integrable if and only if $k$ is a nonnegative integer and $\lambda$ is a dominant integral weight.
\begin{prop}\label{rationalaffine}\cite{DL, DLM0, FZ, LLi, Li0}
 \begin{enumerate}\item For any scalar $k$ that does not equal minus the dual Coxeter number, $L(k,0)$ carries the structure of a vertex operator algebra of strong CFT-type. It is generated by the weight 1 subspace $L(k,0)_1$, which is isomorphic (as a Lie algebra) to $\mathfrak{g}$. 
 \item If $k$ is a positive integer, $L(k,0)$ is a rational vertex operator algebra whose irreducible modules are precisely the integrable highest weight modules $L(k,\lambda)$ satisfying $\langle \lambda,\theta\rangle\leq k$, where $\theta$ is the longest positive root in $\Delta$. In addition, $L(k,0)$ is regular.   
 \end{enumerate}
\end{prop}

\end{ex}

\subsection{On Centralizers and the Center of $\mathbb{N}$-graded vertex operator algebras}

Now, we recall the notions of the centralizer and the center of vertex algebras from \cite{LLi}. Let $V$ be a vertex algebra. Given a subset $S$ of $V$, we define $$C_V(S)=\{v\in V~|~[Y(v,x_1),Y(s,x_2)]=0\text{ for all }s\in S\}.$$ We call $C_V(S)$ the centralizer or commutant of $S$ in $V$. In the case $S=V$, we call $C_V(V)$, the center of $V$, and write $C(V)=C_V(V)$. Notice that we have 
\begin{eqnarray*}
C_V(S)&=&\{v\in V~|~v_ns=0\text{ for all }s\in S,~n\geq 0\}\\
&=&\{v\in V~|~s_nv=0\text{ for all }s\in S,~n\geq 0\}.
\end{eqnarray*}
$C_V(S)$ is a vertex subalgebra. As in the classical theories, we have $C_V(S)=C_V(\langle S\rangle)$
\begin{dfn}\cite{LLi} Let $(V,Y,{\bf 1},\omega)$ be a vertex operator algebra. A vertex operator subalgebra with a possibly different conformal vector is a vertex subalgebra $(U,Y,{\bf 1})$ of $V$ together with an element $\omega'$ of $U$ such that $(U,Y,{\bf 1},\omega')$ is a vertex operator algebra (with a grading possibly different from that of $V$). By abuse of terminology, we shall sometimes call $(U,Y,{\bf 1},\omega')$ a vertex operator subalgebra of $V$ if the context makes the situation clear.
\end{dfn}
\begin{rem}\cite{LLi} Let $(U,Y,{\bf 1},\omega')$ be a vertex operator subalgebra of a vertex operator algebra $(V,Y,{\bf 1},\omega)$, with a possibly different conformal vector. We shall write $Y(\omega',x)=\sum_{n\in\mathbb{Z}}L'(n)x^{-n-2}$, where we view the operators $L'(n)$ as acting on $V$, not just on $U$. Since $(U,Y,{\bf 1},\omega')$ is a vertex operator algebra, the operators $L'(n)$ acting on $U$ give rise to a representation of the Virasoro algebra on $U$ with the central element acting as the scalar $c_U$ on $U$, and the $L'(-1)$-derivative and $L'(-1)$-bracket formulas hold on $U$. \end{rem}
\begin{prop}\cite{LLi} Let $(U,Y, {\bf 1},\omega')$ be a vertex operator subalgebra of a vertex operator algebra $(V,Y,{\bf 1},\omega)$. Then for $u\in U$, $[L'(-1),Y(u,x)]=Y(L'(-1)u,x)=\frac{d}{dx}Y(u,x)$ acting on $V$ where $Y(\omega',x)=\sum_{n\in\mathbb{Z}}L'(n)x^{-n-2}$. In addition, $$[L'(m),L'(n)]=(m-n)L'(m+n)+\frac{1}{12}(m^3-m)\delta_{m+n,0}c_U$$ acting on $V$, and $C_V(U)=Ker_VL'(-1)$, the kernel of $L'(-1)$ acting on $V$.\end{prop} 

We now prove the following theorems. 
\begin{thm}Let $(V,Y,{\bf 1},\omega)$ be a nonzero vertex operator algebra such that $V_n=0$ for $n<0$. Let $(U,Y,{\bf 1},\omega')$ be a vertex operator subalgebra of $V$ and assume that $\omega'\in U\cap V_{(2)}$ and that $L(1)\omega'=0$, $L(2)\omega'\in\mathbb{C}{\bf 1}$. We set $Y(\omega',x)=\sum_{n\in\mathbb{Z}}L'(n)x^{-n-2}$ acting on $V$. Then the gradings of $V$ and $U$ are compatible (i.e., $L(0)=L'(0)$ on $U$), and more generally, $L(n)=L'(n)$ on $U$ for all $n\geq -1$. Now, we set $\omega''=\omega-\omega'$. Then $\omega''\in C_V(U)$, and $(C_V(U),Y,{\bf 1},\omega'')$ is a vertex operator subalgebra of $V$ of central charge equal to $c_V-c_U$, and we have $\omega''\in C_V(U)\cap V_{(2)}$ and $L(1)\omega''=0$.\end{thm}\begin{proof} We set $Y(\omega'',x)=\sum_{n\in\mathbb{Z}}L''(n)x^{-n-2}$ acting on $V$ so that $L''(n)=L(n)-L'(n)$ for $n\in\mathbb{Z}$. For simplicity we set $L(2)\omega'=\beta {\bf 1}$ where $\beta\in\mathbb{C}$.
Observe that for $m,n\in\mathbb{Z}$,
\begin{eqnarray*}
{[L(m),L'(n)]}&=&[\omega_{m+1},\omega'_{n+1}]\\
&=&\sum_{i=0}^{\infty}\binom{m+1}{i}(\omega_i\omega')_{m+n+2-i}\\
&=&(L(-1)\omega')_{m+n+2}+ (m+1)(L(0)\omega')_{m+n+1}\\
&&+\sum_{i=3}^{\infty}\binom{m+1}{i}(L(i-1)\omega')_{m+n+2-i}\\
&=&-(m+n+2)\omega'_{m+n+1}+2(m+1)\omega'_{m+n+1}\\
&&+\delta_{m+n,0}\binom{m+1}{3}\beta {\bf 1}_{m+n-1}\\
&=&(m-n)L'(m+n)+\delta_{m+n,0}\frac{m^3-m}{6}\beta {\bf 1}_{-1}\\
&=&(m-n)L'(m+n)+\delta_{m+n,0}\frac{m^3-m}{12}2\beta {\bf 1}_{-1}.
\end{eqnarray*}
This implies that 
\begin{eqnarray*}
    {[L^{''}(m),L'(n)]}&=&[L(m)-L'(m),L'(n)]\\
    &=&[L(m),L'(n)]-[L'(m),L'(n)]\\
    &=&(m-n)L'(m+n)+\delta_{m+n,0}\frac{m^3-m}{12}2\beta {\bf 1}_{-1}\\
    &&-((m-n)L'(m+n)+\delta_{m+n,0}\frac{m^3-m}{12}c_U)\\
    &=&\delta_{m+n,0}\frac{m^3-m}{12}(2\beta {\bf 1}_{-1}-c_U).\\
\end{eqnarray*}
Therefore, $[L''(m),L'(n)]=0$ when $m\neq -n$ and 
$L'(-1)\omega''=L'(-1)L''(-2){\bf 1}=L''(-2)L'(-1){\bf 1}=0$. We have that $\omega''\in Ker_VL'(-1)=C_V(U)$. This implies that $L''(n)u=\omega_{n+1}u=0$ for all $u\in U$, $n\geq -1$. Consequently, for $n\geq 1$, we have that $L(n)=L'(n)$ on $U$.
In addition, we have $L'(n)\omega''=0$ and $L''(n)\omega'=0$ for $n\geq -1$. Moreover, we have 
$$0=L''(2)\omega'=L''(2)L'(-2){\bf 1}=\frac{1}{2}(2\beta{\bf 1}_{-1}-c_U).$$ Hence $\beta=\frac{1}{2}c_U$.

Since $[L'(m),L'(n)]=(m-n)L'(m+n)+\frac{1}{12}(m^3-m)\delta_{m+n,0}c_U$, we have
\begin{eqnarray*}
&&[L''(m),L''(n)]\\
&=&[L(m)-L'(m),L(n)-L'(n)]\\
&=&(m-n)L(m+n)+\frac{1}{12}(m^3-m)\delta_{m+n,0}c_V+(m-n)L'(m+n)\\
&&+\frac{1}{12}(m^3-m)\delta_{m+n,0}c_U-[L(m),L'(n)]-[L'(m),L(n)]\\
&=&(m-n)(L(m+n)-L'(m+n))+\frac{1}{12}(m^3-m)\delta_{m+n,0}(c_V+c_U)\\
&&-\delta_{m+n,0}\frac{m^3-m}{12}2\beta {\bf 1}_{-1}+\delta_{m+n,0}\frac{n^3-n}{12}2\beta{\bf 1}_{-1}\\
&=&(m-n)L''(m+n)+\frac{1}{12}(m^3-m)\delta_{m+n,0}(c_V-c_U).
\end{eqnarray*}

For $v\in C_V(U)$, and $n\geq -1$, $L''(n)v=L(n)v-L'(n)v=L(n)v$. In particular, we have $L''(-1)v=L(-1)v$ and $L''(0)v=L(0)v$. Consequently, $$Y(L''(-1)v,x)=\frac{d}{dx}Y(v,x)$$ for $v\in C_V(U)$ and $C_V(U)$ is $L(0)$-stable and hence graded (by the grading of $V$), with the grading given by the $L''(0)$-eigenvalues. We can conclude that $(C_V(U),Y,{\bf 1},\omega'')$ is a vertex operator algebra of central charge equal to $c_V-c_U$. 
\end{proof}
\begin{thm}Let $(V,Y,{\bf 1},\omega)$ be a nonzero vertex operator algebra such that $V_n=0$ for $n<0$. Let $(U,Y,{\bf 1},\omega')$ be a vertex operator subalgebra of $V$ and assume that $\omega'\in U\cap V_{(2)}$ and that $L(1)\omega'=0$, $L(2)\omega'\in\mathbb{C}{\bf 1}$. We set $Y(\omega',x)=\sum_{n\in\mathbb{Z}}L'(n)x^{-n-2}$ acting on all of $V$. Then $C_V(C_V(U))=U$ if and only if $(U,Y,{\bf 1},\omega')$ is a maximal in the set that $(T,Y,{\bf 1},\omega')$ is a vertex operator subalgebra of $V$, then $T\subset U$. \end{thm}
\begin{proof}
We know that $(C_V(C_V(U)),Y,{\bf 1},\omega')$ is a vertex operator subalgebra of $V$ and 
\begin{eqnarray*}
C_V(C_V(U))&=&Ker_VL''(-1)=Ker_V(L(-1)-L'(-1))\\
&=&\{v\in V~|~L'(-1)v=L(-1)v=v_{-2}{\bf 1}\}.
\end{eqnarray*} 
Observe that 
\begin{eqnarray*}
C_V(U)&=&\{w\in V~|~[Y(w,x_1),Y(u,x_2)]=0\text{ for all }u\in U\},\\
C_V(C_V(U))&=&\{v\in V~|~[Y(v,x_1), Y(s,x_2)]=0\text{ for all }s\in C_V(U)\}.
\end{eqnarray*}
Hence, $U\subset C_V(C_V(U))$.
Notice that if $U$ is maximal then $C_V(C_V(U))=U$. Now, let us assume the converse, that is $C_V(C_V(U))=U$. Now let $(T,Y,{\bf 1},\omega')$ be a vertex operator subalgebra of $V$. So, when $v\in T$, we have that $L'(-1)v=v_{-2}{\bf 1}$. Hence $T\subseteq C_V(C_V(U))=U$.
\end{proof}
\subsection{Abelian coset vertex operator algebra $\Omega_V$}

Let $(V=\bigoplus_{n\in\mathbb{Z}}V_n,Y(~,~),{\bf 1},\omega)$ be a vertex operator algebra with a vertex operator subalgebra $(M(l),Y(~,~),{\bf 1},\omega_{\mathfrak{h}})$ associated with an affine Lie algebra $\hat{\mathfrak{h}}$ of level $l$, where $l\in\mathbb{C}^{\times}$ and $\mathfrak{h}\subseteq V_1$ is a $d$-dimensional vector space equipped with a nondegenerate symmetric bilinear form $\langle \cdot,\cdot\rangle$. Note that $\langle u,v\rangle=u_1v$ for $u,v\in\mathfrak{h}$.

We assume that $V$ is a semisimple $\mathfrak{h}$-module with $h\in\mathfrak{h}$ being represented by $h_0$ and $L(n)\mathfrak{h}=0$ for $n\geq 1$. We define $\mathcal{C}$ to be the category of weak $V$-modules $W$ on which $\mathfrak{h}$ acts semisimply and $\hat{\mathfrak{h}}^+$ acts locally nilpotently. Then for each $W\in\mathcal{C}$, we have $W=U(\hat{\mathfrak{h}}^-)\otimes \Omega_W$ where $\Omega_W=\{w\in W~|~h_nw=0\text{ for }h\in\mathfrak{h},~n\geq 1\}.$ For $\alpha\in\mathfrak{h}$, we set 
\begin{eqnarray*}
    W^{\alpha}&=&\{w\in W~|~h_0w=\langle \alpha,h\rangle w,~h\in\mathfrak{h}\},\\
    \Omega_W^{\alpha}&=&W^{\alpha}\cap \Omega_W.
\end{eqnarray*} 
We define $P_V$ to be a group generated by $\alpha\in\mathfrak{h}$ such that $V^{\alpha}\neq 0$. $P_V$ is, in fact, a subgroup of the additive group $\mathfrak{h}$. Then $V=\bigoplus_{\alpha\in P_V}V^{\alpha}$. \begin{prop}\label{PV}\cite{LX}
    If $V$ is a simple vertex operator algebra such that $V$ is a semisimple $\mathfrak{h}$-module and $L(n)\mathfrak{h}=0$ for $n\geq 1$, then $P_V=\{\alpha\in \mathfrak{h}~|~V^{\alpha}\neq 0\}$.
\end{prop}
Now, we set $P_W=\cup_{\lambda\in\mathfrak{h}, W^{\lambda}\neq 0}P_V+\lambda$. Then we have the following: $$W^{\alpha}=M(1)\otimes \Omega_W^{\alpha},~
    \Omega_W=\oplus_{\alpha\in P_W}\Omega_W^{\alpha},~
    W=\oplus_{\alpha\in P_W}W^{\alpha}.$$
\begin{prop}\label{PW}\cite{DM, DM1, Li3} Let $V$ be a simple vertex operator algebra such that $V$ is a semisimple $\mathfrak{h}$-module and $L(n)\mathfrak{h}=0$ for $n\geq 1$. Let $W\in \mathcal{C}$ such that $W$ is a simple $V$-module. Then we have the following:
\begin{enumerate}
    \item $P_W=P_V+\lambda$ for any $\lambda\in\mathfrak{h}$ with $W^{\lambda}\neq 0$.
    \item $\Omega_V^0$ is a simple vertex operator subalgebra of $V$ with a Virasoro element $\omega_{\Omega}=\omega-\omega_{\mathfrak{h}}$. In addition, each $\Omega_V^{\alpha}$ is a simple $\Omega_V^0$-module. 
    \item $V^0=M(1)\otimes \Omega_V^0$ is a simple, conformal vertex operator subalgebra of $V$. Moreover, as $V^0$-module, $V^{\alpha}=M(1,\alpha)\otimes \Omega_V^{\alpha}$ is simple. 
    \item If $\alpha\in P_W$, then $W^{\alpha}=M(1)\otimes \Omega_W^{\alpha}$ is a simple $V^0$-module. Moreover, $W^{\alpha}\not\cong W^{\beta}$ if $\alpha\neq \beta$.
\end{enumerate}
\end{prop}

We continue with an assumption that $V$ is a vertex operator algebra such that $V$ is a semisimple $\mathfrak{h}$-module and $L(n)\mathfrak{h}=0$ for $n\geq 1$. We define $$P_V^{\circ}=\{\beta\in\mathfrak{h}~|~\langle \alpha,\beta\rangle\in\mathbb{Z}\text{ for }\alpha\in P_V\}.$$ Now, we let $W$ be a weak $V$-module and $\alpha\in P_V^{\circ}$. We define a new action of $V$ on $W$ by $$Y_W^{\Delta_{\alpha}(x)}(v,x)=Y_W(\Delta(\alpha,x)v,x)$$ where $\Delta(\alpha,x)=x^{\alpha_0}\exp\left(\sum_{k=1}^{\infty}\frac{\alpha_k}{-k}(-x)^{-k}\right)$. \begin{prop}\label{Keven}\cite{Li3} Assume that $V$ is a vertex operator algebra such that $V$ is a semisimple $\mathfrak{h}$-module and $L(n)\mathfrak{h}=0$ for $n\geq 1$. 
\begin{enumerate} \item For a weak $V$-module $W$, $W^{(\alpha)}=(W,Y_W^{\Delta_{\alpha}(x)}(~,~))$ is a weak $V$-module. In addition, $W$ is irreducible if and only if $W^{(\alpha)}$ is irreducible and $(W^{(\alpha)})^{(\beta)}\cong W^{(\alpha+\beta)}$ for $\alpha,\beta\in P_V^{\circ}$.
\item Assume that $V$ is simple and $P_V$ equipped with $\langle~,~\rangle$ is a nondegenerate rational lattice of finite rank and $l$ is a nonzero rational number. Then the group $K=\{\alpha\in P_V^{\circ}~|~V^{(\alpha)}\cong V\text{ as a $V$-module}\}$ equipped with the $\mathbb{Z}$-bilinear form $l\langle~,~\rangle$ is an even lattice. In addition, $lK\subseteq P_V$ and for $W\in\mathcal{C}$, and $\alpha\in K$, we have $\langle\alpha,s\rangle\in\mathbb{Z}$ for $s\in P_W$.
    \end{enumerate}
\end{prop}

Under the assumption of Proposition \ref{Keven} (2), there exists a $\langle\pm 1\rangle$-valued function $\varepsilon_K(\cdot,\cdot)$ on $K\times K$ such that 
\begin{eqnarray*}
    &&\varepsilon_K(\alpha+\beta,\gamma)=\varepsilon_K(\alpha, \gamma)\varepsilon_K(\beta,\gamma),~\varepsilon_K(\gamma,\alpha+\beta)=\varepsilon_K(\gamma,\alpha)\varepsilon_K(\gamma,\beta),\\
    &&\varepsilon_K(\alpha,\beta)\varepsilon_K(\beta,\alpha)^{-1}=(-1)^{l\langle\alpha,\beta\rangle}\text{ for }\alpha,\beta,\gamma\in K.
\end{eqnarray*} 
For $\alpha\in K$, we let $\psi_{\alpha}:V^{(\alpha)}\rightarrow V$ be a $V$-isomorphism from $V^{(\alpha)}$ to $V$. Now, we assume that 
\begin{equation}\label{psiepsilon} \psi_{\alpha}\psi_{\beta}=\varepsilon_K(\alpha,\beta)\psi_{\alpha+\beta}\text{ for }\alpha,\beta\in K.\end{equation}
\begin{prop}\label{lattice construction}\cite{Li3}
 Assume that $V$ is simple, $P_V$ equipped with $\langle~,~\rangle$ is a nondegenerate rational lattice of finite rank, $l$ is a nonzero rational number, and $K$ satisfies the equation (\ref{psiepsilon}). Then $V_{lK}=C_V(\Omega_V^0),\text{ and }C_V(V_{lK})=C_V(C_V(\Omega_V))=\Omega_V.$
\end{prop}


\section{Semisimple Lie subalgebras of $V_1$ and affine vertex operator algebras}

We first consider the case when the Leibniz algebra $V_1$ contains a semisimple Lie algebra. Under the assumption $\ker D=\mathbb {C}{\mathbf 1}$, the vertex-algebra pairing $u_1v$ induces an invariant symmetric bilinear form on a Levi factor $S$ of $V_1$. We show that, under $C_2$-cofiniteness and an appropriate nondegeneracy condition, each simple component generates a positive-integral-level affine vertex algebra and acts integrably on $V$.

We let $V=\bigoplus_{n=0}^{\infty}V_n$ be an $\mathbb{N}$-graded \textbf{vertex algebra} such that $1<\dim V_0<\infty$, and $\dim V_i<\infty$ for all $i\geq 1$. It is known that $V_1$ is a left Leibniz algebra with a Leibniz bracket defined by $[b,b']=b_0b'$ for $b,b'\in V_1$. By the Levi decomposition for Leibniz algebras, there exists a Levi semisimple subalgebra $S$ of $V_1$ such that $V_1=S\dot{+}\mathrm{Rad}(V_1)$. Here, $\mathrm{Rad}(V_1)$ is the solvable radical of the Leibniz algebra $V_1$. 

For the rest of this section, we assume that ${\Ker}(D|_{V_0})=\mathbb{C}{\bf 1}$. Here $D$ is the canonical translation operator of the vertex algebra $V$. By skew-symmetry of the vertex algebra,
$u_0v=-v_0u+D(u_1v)$ for $u,v\in S$. On the other hand, since $S$ is a Lie algebra under the bracket $[u,v]=u_0v$, $u_0v=-v_0u$. Hence $D(u_1v)=0$. Because $u_1v\in V_0$ and ${\Ker}(D|_{V_0})=\mathbb{C}{\bf 1}$, we have that  $u_1v\in\mathbb{C}{\bf 1}$. 

We define a bilinear map $(~,~):S\times S\rightarrow\mathbb{C}$ in the following way: for $u,v\in S$, $(u,v){\bf 1}=u_1v$. Because $u_1v=v_1u$ for all $u,v\in V_1$, we can conclude that $(~,~)$ is a symmetric bilinear map on $S$. Since $u_0v=-v_0u$, $u_1v\in\mathbb{C}{\bf 1}$ and $u_1(v_0w)=v_0u_1v-(v_0u)_1w$ for all $u,v,w\in S$, it implies that $u_1(v_0w)=(u_0v)_1w$ for all $u,v,w\in S$. Equivalently, $([u,v],w)=(u,[v,w])$ for all $u,v,w\in S$ and $(~,~)$ is an invariant symmetric bilinear form on $S$. In summary, we have the following Proposition.
\begin{prop}
    Let $V=\bigoplus_{n=0}^{\infty}V_n$ be an $\mathbb{N}$-graded vertex algebra such that $1<\dim V_0<\infty$, $\dim V_i<\infty$ for all $i\geq 1$ and ${\Ker}(D|_{V_0})=\mathbb{C}{\bf 1}$. Also, let $S$ be a Levi semisimple Lie subalgebra of $V_1$. We define a bilinear map $(~,~):S\times S\rightarrow\mathbb{C}$ in the following way: for $u,v\in S$, $(u,v){\bf 1}=u_1v$. Then $(~,~)$ is an invariant symmetric bilinear form on $S$.
\end{prop}
The same argument applies to any semisimple Lie subalgebra of \(V_1\), yielding the following corollary.
\begin{cor}
    Let $V=\bigoplus_{n=0}^{\infty}V_n$ be an $\mathbb{N}$-graded vertex algebra such that $1<\dim V_0<\infty$, $\dim V_i<\infty$ for all $i\geq 1$ and $\mathrm{Ker}(D|_{V_0})=\mathbb{C}{\bf 1}$. Also, let $\mathfrak{g}$ be a semisimple Lie subalgebra of $V_1$. We define a bilinear map $(~,~):\mathfrak{g}\times \mathfrak{g}\rightarrow\mathbb{C}$ in the following way: for $u,v\in \mathfrak{g}$, $(u,v){\bf 1}=u_1v$. Then $(~,~)$ is an invariant symmetric bilinear form on $\mathfrak{g}$.
\end{cor}

\begin{lem}\label{Chevalley} Let $V=\bigoplus_{n=0}^{\infty}V_n$ be an $\mathbb{N}$-graded vertex algebra such that $1<\dim V_0<\infty$, $\dim V_i<\infty$ for all $i\geq 1$ and $\mathrm{Ker}(D|_{V_0})=\mathbb{C}{\bf 1}$. Let $\mathfrak{g}$ be a simple Lie subalgebra of $V_1$ such that $(~,~)|_{\mathfrak{g}}\neq 0$. Then $(~,~)$ is nondegenerate on $\mathfrak{g}$ and there is $k\in\mathbb{C}^{\times}$ such that $(~,~)=k \langle \cdot,\cdot\rangle$ on $\mathfrak{g}$. Here  $\langle \cdot,\cdot\rangle$ is the normalized invariant bilinear form of $\mathfrak{g}$ such that $\langle\alpha,\alpha\rangle=2$ for long roots $\alpha$.
\end{lem}
\begin{proof} We define $\mathrm{Rad}((~,~)|_{\mathfrak{g}})=\{u\in\mathfrak{g}~|~(u,v)=0\text{ for all }v\in\mathfrak{g}\}$. Because $(~,~)$ is $\mathfrak{g}$-invariant, and $\mathrm{Rad}((~,~)|_{\mathfrak{g}})$ is an ideal of the simple Lie algebra $\mathfrak{g}$, we can conclude that $\mathrm{Rad}((~,~)|_{\mathfrak{g}})$ is either $\{0\}$ or $\mathfrak{g}$. However, because $(~,~)|_{\mathfrak{g}}\neq 0$, we then have that $\mathrm{Rad}((~,~)|_{\mathfrak{g}})=\{0\}$. Therefore, $(~,~)$ is nondegenerate on $\mathfrak{g}$. 

Since $\mathfrak{g}$ is a finite-dimensional complex simple Lie algebra, every invariant symmetric bilinear form on $\mathfrak{g}$ is a scalar multiple of the normalized invariant bilinear form. Thus there exists $k\in\mathbb{C}$ such that $(u,v)=k\langle u,v\rangle$ for all $u,v\in\mathfrak g.$ Since $(~,~)|_{\mathfrak{g}}\neq0$, we have $k\neq 0$ and hence $k\in\mathbb{C}^{\times}$. 
\end{proof}

\begin{rem}In particular, the hypothesis of Lemma \ref{Chevalley} holds if $(e^i,f^i)\neq 0$ for some pair of Chevalley generators $e^i,f^i$.
\end{rem}

\begin{prop}\label{semisimple-orthogonal-decomposition}
Let $V=\bigoplus_{n=0}^{\infty}V_n$ be an $\mathbb{N}$-graded vertex algebra such that $1<\dim V_0<\infty$, $\dim V_i<\infty$ for all $i\geq 1$ and $\mathrm{Ker}(D|_{V_0})=\mathbb{C}{\bf 1}$. Let $\mathfrak{g}$ be a semisimple Lie subalgebra of $V_1$ and suppose that $\mathfrak{g}
=
\mathfrak{g}_1\oplus\cdots\oplus\mathfrak{g}_r$, where each $\mathfrak{g}_i$ is a simple Lie algebra. Then
\begin{enumerate}
\item $(\mathfrak{g}_i,\mathfrak{g}_j)=0$ for all $i\neq j$. Consequently,
$(~,~)|_{\mathfrak{g}}
=\perp_{i=1}^{r}(~,~)|_{\mathfrak{g}_i}$. 
\item Moreover, for each $i$, either $(~,~)|_{\mathfrak{g}_i}=0$, or there exists $k_i\in\mathbb{C}^{\times}$ such that $(~,~)|_{\mathfrak{g}_i}
= k_i\langle\cdot,\cdot\rangle_i$, where $\langle\cdot,\cdot\rangle_i$ is the normalized invariant
bilinear form on $\mathfrak{g}_i$ such that $\langle\alpha,\alpha\rangle_i=2$ for every long root $\alpha$ of $\mathfrak{g}_i$. In the latter case,
$(~,~)|_{\mathfrak{g}_i}$ is nondegenerate.
\item In particular, $(~,~)|_{\mathfrak{g}}$ is nondegenerate if and only if $(~,~)|_{\mathfrak{g}_i}\neq 0$ for every $1\leq i\leq r$. 
\end{enumerate}
\end{prop}

\begin{proof}
Since $\mathfrak{g}
=
\mathfrak{g}_1\oplus\cdots\oplus\mathfrak{g}_r$ is a direct sum of simple ideals, we have $[\mathfrak{g}_i,\mathfrak{g}_j]=0$ for $i\neq j$. Moreover, each $\mathfrak{g}_i$ is perfect: $\mathfrak{g}_i=[\mathfrak{g}_i,\mathfrak{g}_i]$. Now, we let $i\neq j$, $u\in\mathfrak{g}_i$, and
$v\in\mathfrak{g}_j$. Since $\mathfrak{g}_i$ is perfect, $u$ is a
finite sum of elements of the form $[a,b]$ with
$a,b\in\mathfrak{g}_i$. By the invariance of $(~,~)$,
$([a,b],v)=(a,[b,v])$. Since $[\mathfrak{g}_i,\mathfrak{g}_j]=0$, we have $[b,v]=0$. Hence $([a,b],v)=0$. It follows that $(\mathfrak{g}_i,\mathfrak{g}_j)=0$ for $i\neq j$. Therefore, $(~,~)|_{\mathfrak{g}}
=
\perp_{i=1}^{r}(~,~)|_{\mathfrak{g}_i}$. This proves (1).

For each $i$, the restriction $(~,~)|_{\mathfrak{g}_i}$ is an
invariant symmetric bilinear form on the simple Lie algebra
$\mathfrak{g}_i$. If $(~,~)|_{\mathfrak{g}_i}=0$, there is nothing to prove. Suppose instead that $(~,~)|_{\mathfrak{g}_i}\neq 0$. By Lemma \ref{Chevalley}, the restriction $(~,~)|_{\mathfrak{g}_i}$
is nondegenerate, and there exists $k_i\in\mathbb{C}^{\times}$ such
that $(~,~)|_{\mathfrak{g}_i}
=
k_i\langle\cdot,\cdot\rangle_i$. This proves (2).

Finally, since the simple ideals $\mathfrak{g}_i$ are mutually
orthogonal, the radical of $(~,~)|_{\mathfrak{g}}$ is
\[
\operatorname{Rad}((~,~)|_{\mathfrak{g}})
=
\bigoplus_{i=1}^{r}
\operatorname{Rad}((~,~)|_{\mathfrak{g}_i}).
\]
Hence $(~,~)|_{\mathfrak{g}}$ is nondegenerate if and only if its
restriction to every simple ideal $\mathfrak{g}_i$ is nondegenerate.
By Lemma \ref{Chevalley}, this is equivalent to $(~,~)|_{\mathfrak{g}_i}\neq0$ for every $i$. This completes the proof for (3).
\end{proof}
Theorems \ref{affine-integrability} and \ref{main for section 3} state that, assuming Lemma \ref{Chevalley} and that $V$ is $C_2$-cofinite, $V$ contains a vertex subalgebra isomorphic to a rational affine vertex algebra. However, the proof of Theorem \ref{affine-integrability} relies on the following extension of Lemma 2.4 of \cite{Mi}, stated here as Proposition \ref{Miyamoto-spanning}. Although Lemma 2.4 of \cite{Mi} is formulated for $\mathbb{N}$-graded vertex operator algebras, its proof extends to the more general setting of $\mathbb{N}$-graded vertex algebras with finite-dimensional homogeneous subspaces and $\dim V_0>1$.

\begin{prop}\label{Miyamoto-spanning}
Let $V=\bigoplus_{n=0}^{\infty}V_n$ be a $C_2$-cofinite $\mathbb{N}$-graded vertex algebra, and let
$A=\{v^1,\ldots,v^d\}$ be a finite set of homogeneous representatives
whose images span $V/C_2(V)$. Then $V$ is spanned by ${\bf 1}$ together
with vectors of the form
$$v^{i_1}_{-n_1}v^{i_2}_{-n_2}\cdots v^{i_s}_{-n_s}{\bf 1},$$
where $n_1>n_2>\cdots>n_s>0$ and $v^{i_j}\in A$ for all $j$.

Consequently, if $v^{i_1}_{-n_1}\cdots v^{i_s}_{-n_s}{\bf 1}\in V_N$, then $N\geq \frac{s(s-1)}{2}$. In particular, for every integer $t\geq0$, $\bigoplus_{N\leq \frac{t(t+1)}2}V_N$ is spanned by vectors of the above form with $s\leq t+1$.
\end{prop}

\begin{proof}
The first assertion is the specialization of Lemma 2.4 of \cite{Mi}
to the weak module $W=V$ generated by the vacuum vector ${\bf 1}$.
Indeed, Lemma 2.4 of \cite{Mi} states that $V$ is spanned by vectors
$$v^{i_1}_{m_1}\cdots v^{i_s}_{m_s}{\bf 1},$$
where $m_1<\cdots<m_s$, with $v^{i_j}\in A$. Since $v_n{\bf 1}=0$ for $n\geq 0$, a nonzero vector of this form must satisfy $m_s<0$, and hence $m_1<\cdots<m_s<0$. One may write $m_j=-n_j$, and obtain $n_1>n_2>\cdots>n_s>0$.

Now we let $u=v^{i_1}_{-n_1}\cdots v^{i_s}_{-n_s}{\bf 1}$ be homogeneous of degree $N$. Since every $v^{i_j}$ has
nonnegative weight, $\wt\big(v^{i_j}_{-n_j}v\big)= \wt(v^{i_j})+n_j-1+\wt(v)\geq n_j-1+\wt(v)$. Therefore
$$N
=\sum_{j=1}^s
\left(\wt(v^{i_j})+n_j-1\right)
\geq
\sum_{j=1}^s(n_j-1).$$
Since $n_1>n_2>\cdots>n_s>0$, we have $n_j\geq s-j+1$, and hence $$\sum_{j=1}^s(n_j-1)
\geq
(s-1)+(s-2)+\cdots+1+0
=
\frac{s(s-1)}2.$$
Thus $N\geq \frac{s(s-1)}2$. If $N\leq\frac{t(t+1)}2$, then $\frac{s(s-1)}2
\leq
\frac{t(t+1)}2$, which implies
$s\leq t+1$. This proves the proposition.
\end{proof}
We can now use Proposition \ref{Miyamoto-spanning} to prove the
integrability theorem.
\begin{thm}\label{affine-integrability}
Let $V=\bigoplus_{n=0}^{\infty}V_n$ be an $\mathbb{N}$-graded vertex algebra such that $1<\dim V_0<\infty$, $\dim V_n<\infty$ for all $n\geq 1$, ${\Ker}(D|_{V_0})=\mathbb{C}{\bf 1}$, and suppose that $V$ is $C_2$-cofinite. Let $\mathfrak{g}$ be a simple
Lie subalgebra of $V_1$ such that $(~,~)|_{\mathfrak{g}}\neq0$. Let $U$ be the vertex subalgebra of $V$ generated by $\mathfrak{g}$.
Then there exists a positive integer $k$ such that $(~,~)|_{\mathfrak{g}}
=
k\langle\cdot,\cdot\rangle$, 
and $U\cong L_{\widehat{\mathfrak{g}}}(k,0)$. Moreover, $V$ is an integrable
$\widehat{\mathfrak{g}}$-module.
\end{thm}

\begin{proof}
By Lemma \ref{Chevalley}, the restriction of $(~,~)$ to
$\mathfrak{g}$ is nondegenerate, and there exists
$k\in\mathbb{C}^{\times}$ such that $(u,v)=k\langle u,v\rangle$ for all $u,v\in\mathfrak{g}$, where $\langle\cdot,\cdot\rangle$ is the normalized invariant
bilinear form on $\mathfrak{g}$.

For $u,v\in\mathfrak{g}$ and $m,n\in\mathbb{Z}$, the commutator
formula gives
\[
[u_m,v_n]
=
(u_0v)_{m+n}
+
m(u_1v)\delta_{m+n,0}.
\]
Since $u_0v=[u,v]$ and $u_1v=(u,v){\bf 1}
=
k\langle u,v\rangle{\bf 1}$, we obtain
$$[u_m,v_n]
=
[u,v]_{m+n}
+
mk\langle u,v\rangle
\delta_{m+n,0}.
$$
Hence $u(m)\mapsto u_m,~
K\mapsto k$ defines a representation of the affine Kac--Moody algebra
$\widehat{\mathfrak{g}}$ on $V$ of level $k$.

We first consider the case $\mathfrak{g}\cong\mathfrak{sl}_2$. Let
$\{\alpha,e^\alpha,e^{-\alpha}\}$
be the standard basis, normalized so that
$\langle\alpha,\alpha\rangle=2$, $[\alpha,e^{\alpha}]=2e^{\alpha}$, and $[\alpha,e^{-\alpha}]=-2e^{-\alpha}$. Since each $V_n$ is finite-dimensional and invariant under the
zero-mode action of $\mathfrak{g}$, every $V_n$ is a finite-dimensional
completely reducible $\mathfrak{sl}_2$-module. Consequently, $V$ is a
direct sum of finite-dimensional $\mathfrak{sl}_2$-modules.

We call a nonzero vector $v\in V$ an $\mathfrak{sl}_2$-weight vector
of weight $\lambda$ if $\alpha_0v=\langle\alpha,\lambda\rangle v$. The possible weights belong to
$\frac{1}{2}\mathbb{Z}\alpha$. Since $C_2(V)$ is invariant under the zero-mode action of
$\mathfrak{g}$, the quotient $V/C_2(V)$ is a finite-dimensional $\mathfrak{g}$-module. We may therefore choose the finite homogeneous set $A=\{v^1,\ldots,v^d\}$ in Proposition \ref{Miyamoto-spanning} so that every $v^i$ is an
$\mathfrak{sl}_2$-weight vector. Because $A$ is finite, there exists
$
\lambda_0=m\alpha\in\frac12\mathbb{Z}_{\geq0}\alpha$ such that the $\mathfrak{sl}_2$-weight of every element of $A$ is less
than or equal to $\lambda_0$.

Let $u=
v^{i_1}_{-n_1}\cdots v^{i_s}_{-n_s}{\bf 1}$ be one of the spanning vectors in Proposition
\ref{Miyamoto-spanning}. If $v^{i_j}$ has
$\mathfrak{sl}_2$-weight $\lambda_j$, then
$[\alpha_0,(v^{i_j})_{-n_j}]
=
(\alpha_0v^{i_j})_{-n_j}=\lambda_j v^{i_j}_{-n_j}$. It follows inductively that $u$ is an $\mathfrak{sl}_2$-weight vector
of weight $\lambda_1+\cdots+\lambda_s$. Therefore, the $\mathfrak{sl}_2$-weight of $u$ is less than or equal to $s\lambda_0$.

By Proposition \ref{Miyamoto-spanning}, if $N\leq\frac{t(t+1)}{2}$, then every vector of $V_N$ is a linear combination of spanning
vectors containing at most $t+1$ factors. Hence every
$\mathfrak{sl}_2$-weight occurring in $V_N$ is bounded above by
$(t+1)\lambda_0=(t+1)m\alpha.$ Thus every $\mathfrak{sl}_2$-weight in $V_N$ is at most $
(t+1)m\alpha$.

Choose an integer $t$ sufficiently large so that $\frac{t}{2}>m$. Now, set $N=\frac{t(t+1)}{2}$
and consider $w=(e^\alpha)_{-1}^{N}{\bf 1}$.
Since $e^\alpha\in V_1$, we have $w\in V_N$. Moreover,
$[\alpha_0,(e^\alpha)_{-1}]
=
2(e^\alpha)_{-1}$. Thus, if $w\neq0$, it is an $\mathfrak{sl}_2$-weight vector of
weight $N\alpha
=
\frac{t(t+1)}2\alpha$. On the other hand, every
$\mathfrak{sl}_2$-weight occurring in $V_N$ is at most $(t+1)m\alpha$. However, $\frac{t(t+1)}{2}>(t+1)m$ because $t/2>m$. This is a contradiction. Hence $(e^{\alpha})_{-1}^{N}{\bf 1}=0$. Therefore the highest-weight
$\widehat{\mathfrak{sl}}_2$-module $U$ generated by ${\bf 1}$ is
integrable. By Theorem 10.7 of \cite{K}, it follows that $U\cong L_{\widehat{\mathfrak{sl}}_2}(k,0)$ for some nonnegative integer $k$. Since $(~,~)|_{\mathfrak{g}}\neq0$, we have $k\neq0$. Hence $k$ is a positive integer. Since $V$ is a weak
$U$-module and $U$ is regular, by Proposition \ref{rationalaffine}, $V$ is therefore integrable as a
$\widehat{\mathfrak{sl}}_2$-module.

We now consider an arbitrary finite-dimensional complex simple Lie
algebra $\mathfrak{g}$. Let $\Phi$ be its root system. Let $\theta$ be the longest positive root in $\Phi$. We choose root vectors $e^{\theta}\in \mathfrak{g}_{\theta}$, $e^{-\theta}\in\mathfrak{g}_{-\theta}$ and let $h_{\theta}$ be the corresponding coroot. Then $\mathfrak{s}_{\theta}
=
\mathbb{C}e^{\theta}
\oplus
\mathbb{C}h_{\theta}
\oplus
\mathbb{C}e^{-\theta}$ is a Lie subalgebra isomorphic to $\mathfrak{sl}_2$. Applying the preceding argument to $\mathfrak{s}_{\theta}$, we obtain $(e^{\theta})_{-1}^{N_{\theta}}{\bf 1}=0$ for some positive integer $N_{\theta}$. Hence the highest-weight
$\widehat{\mathfrak{g}}$-module generated by ${\bf 1}$ is integrable.
In addition, $U\cong L_{\widehat{\mathfrak{g}}}(k,0)$ for some nonnegative integer $k$. Again, $k\neq0$ because $(~,~)|_{\mathfrak{g}}\neq 0$. Consequently, $k\in\mathbb{Z}_{>0}$.

Finally, $V$ is a weak module for $U\cong L_{\widehat{\mathfrak{g}}}(k,0)$. Since $k$ is a positive integer, every weak module for this affine
vertex operator algebra is a direct sum of integrable
$\widehat{\mathfrak{g}}$-modules. Hence $V$ is an integrable
$\widehat{\mathfrak{g}}$-module. This completes the proof.
\end{proof}

\begin{thm}\label{main for section 3} Let $V=\bigoplus_{n=0}^{\infty}V_n$ be an $\mathbb{N}$-graded vertex algebra such that $1<\dim V_0<\infty$, and $\dim V_i<\infty$ for all $i\geq 1$, ${\Ker}(D|_{V_0})=\mathbb{C}{\bf 1}$ and $V$ is $C_2$-cofinite. Let $S$ be a Levi-subalgebra of the Leibniz algebra $V_1$ such that $$S=\mathfrak{g}_1\oplus\dots\oplus \mathfrak{g}_n$$ where each $\mathfrak{g}_j$ is a simple Lie algebra. Assume that for each $j$, $\mathfrak{g}_j$ satisfies the assumption in Lemma \ref{Chevalley}. Then 
\begin{enumerate}\item $(~,~)$ is nondegenerate on $S$.
\item The vertex subalgebra of $V$ that is generated by $S$ is isomorphic to
$$L_{\hat{\mathfrak{g}}_1}(k_1,0)\otimes \dots\otimes L_{\hat{\mathfrak{g}}_n}(k_n,0)$$ for some positive integers $k_1,...,k_n$.
\end{enumerate}
\end{thm}
\begin{proof} This follows immediately from Proposition \ref{semisimple-orthogonal-decomposition}, and Theorem \ref{affine-integrability}.\end{proof}
\section{Regular $\mathbb{N}$-graded vertex operator algebras with solvable $V_1$} 

We now turn to the complementary case in which $V_1$ is solvable. When $V_0$ is a local graded Gorenstein algebra, the interaction between the Frobenius structure of $V_0$, the invariant bilinear form, and the weight-one space $V_1$ produces a different source of rigidity. 

The first goal of this section is to identify a distinguished nondegenerate subspace $\mathcal{M}\subseteq V_1$ associated with the Gorenstein structure of $V_0$. Under the quasi-primary condition $L(1)\mathcal{M}=0$, we show that $\mathcal{M}$ is abelian and the vertex operator subalgebra generated by $\mathcal{M}$ is a Heisenberg vertex operator algebra. We then study the action of the corresponding Heisenberg zero modes on $V$ and the decomposition of $V$ into Heisenberg generalized eigenspaces. 

Passing from Heisenberg structure to lattice structure requires additional information. We isolate three hypotheses that control the relevant obstructions: the condition ($\Omega$\text{-SS}), which rules out the Jordan obstruction for the zero modes of $\mathcal{M}$ on the vacuum space with respect to the Heisenberg algebra that is generated by $\mathcal{M}$; the full-rank integrality condition ($\mathrm{IL}$), which produces the required integral lattice; and the lattice cocycle compatibility condition ($\mathrm{LC}$), which permits reconstruction of a lattice vertex operator algebra from the corresponding simple-current translations. These assumptions play distinct roles and will be introduced separately below. 

The main result of the section shows that under these hypotheses, $V$ contains a conformal vertex operator subalgebra isomorphic to $V_K$, where $K$ is a positive-definite even lattice of rank $\dim\mathcal{M}$. Moreover, $\min\{\langle\alpha,\alpha\rangle_{\mathcal{M}}~|~0\neq \alpha\in K\}\geq 4$. 

The minimum-norm bound reflects the solvability of $V_1$: a norm-two vector in $K$ would give rise to a copy of $\mathfrak{sl}_2$ in the weight-one space of the conformally embedded lattice vertex operator algebra, and hence in $V_1$, contradicting solvability.

Thus, in the solvable case, we replace affine vertex operator algebras in Section~3 with Heisenberg and lattice vertex operator algebras. This provides a mechanism through which familiar structures from strongly rational vertex operator algebras persist when $V_0$ is allowed to be a nontrivial Gorenstein algebra.

Let $(V=\bigoplus_{n=0}^{\infty}V_n,Y(~,~),{\bf 1},\omega)$ be an $\mathbb{N}$-graded vertex operator algebra such that $V_0=\bigoplus_{d=0}^sV_0^d$ is a graded-Gorenstein algebra with the unique maximal ideal $\mathfrak{m}$.

As $V_0$ is a graded Gorenstein algebra, the socle $soc(V_0)$ is a simple $V_0$-module and $soc(V_0)\cong V_0/\mathfrak{m}$. In particular, one has $$V_0^s=soc(V_0)=\mathbb{C}t=Ann_{V_0}(\mathfrak{m})$$ for some $t\in\mathfrak{m}$. Note that $V_0^0=\mathbb{C}{\bf 1}$ and $\mathbb{C}t$ is the unique minimal ideal of $V_0$. 

\begin{prop}\label{V_0M}\cite{KSY}
    Let $V=\bigoplus_{n=0}^{\infty}V_n$ be an $\mathbb{N}$-graded vertex operator algebra such that $V_0=\bigoplus_{d=0}^sV_0^d$ is a graded Gorenstein algebra with the unique maximal ideal $\mathfrak{m}$ and $1<\dim_{\mathbb{C}}V_0$. Now we set $\mathfrak{a}=Span\{v_0a~|~v\in V_1,~a\in\mathfrak{m}\}$. Then the following statements hold
    \begin{enumerate}
        \item $\mathfrak{a}$ is an ideal of $V_0$.
        \item For $a\in\mathfrak{m}$, $v\in V_1$, $v_0a\in\mathfrak{m}$ if and only if $L(1)(a_{-1}v)\in\mathfrak{m}$. In addition, $v_0t\in\mathbb{C}t$ if and only if $L(1)(t_{-1}v)\in\mathbb{C}t$.
        \item If $\mathfrak{a}\neq V_0$ then $v_0t\in \mathbb{C}t$ for all $v\in V_1$.
    \end{enumerate}
\end{prop}
\begin{rem}
    The original statement in Proposition \ref{V_0M} assumes an $\mathbb{N}$-graded quasi vertex operator algebra $\bigoplus_{n=0}^{\infty}V_n$ such that $V_0$ is a graded Gorenstein algebra with a unique maximal ideal. 
\end{rem}
We will use $B(~,~)$ to denote the nondegenerate symmetric invariant bilinear form associated with the Poincaré duality of $V_0$. For simplicity, we may and shall set $B({\bf 1},t)=1$.

Following \cite{MY, KSY}, we define a bilinear map $((~,~)):V_1\times V_1\rightarrow\mathbb{C}$ by $((u,v))=B(u_1v,t)$ for $u,v\in V_1$. Also, we define $$\rad((~,~))=\{u\in V_1~|~((u,v))=0\text{ for all }v\in V_1\}.$$ 

We next study relations between $L(-1)V_0$, $\mathfrak{m}$ and $\rad((~,~))$.
\begin{lem}\label{degenerate of (( , ))} Assume that $(V=\bigoplus_{n=0}^{\infty}V_n,Y(~,~),{\bf 1},\omega)$ is an $\mathbb{N}$-graded vertex operator algebra such that $\dim V_0>1$, $V_0=\bigoplus_{d=0}^sV_0^d$ is a graded-Gorenstein algebra with the unique maximal ideal $\mathfrak{m}$. We define a bilinear map $((~,~)):V_1\times V_1\rightarrow\mathbb{C}$ by $((u,v))=B(u_1v,t)$ for $u,v\in V_1$. Here, $B(~,~)$ is the nondegenerate symmetric invariant bilinear form associated with the Poincaré duality of $V_0$. Also, we define $$\rad((~,~))=\{u\in V_1~|~((u,v))=0\text{ for all }v\in V_1\}.$$  The following properties hold
\begin{enumerate} \item $L(-1)V_0\subseteq \rad((~,~))$ if and only if $\mathfrak{m}$ is a $V_1$-module. Moreover, if $V_0\neq \mathfrak{a}$ then $L(-1)V_0\subseteq \rad((~,~))$
\item If $\Ker L(-1)|_{V_0}=\mathbb{C}{\bf 1}$, and $\mathfrak{m}$ is a $V_1$-module then $((~,~))$ is degenerate on $V_1$.
    \end{enumerate}
\end{lem}
\begin{proof}
    \begin{enumerate}
        \item We assume that $L(-1)V_0\subseteq \rad ((~,~))$. We will show that $\mathfrak{m}$ is $V_1$-module. If $\mathfrak{m}$ is not a $V_1$-module, then there exist $u\in V_1$, $a\in\mathfrak{m}$ such that $u_0a=\lambda {\bf 1}+\beta$ where $\lambda\neq 0$ and $\beta\in \mathfrak{m}$. This implies that 
        \begin{eqnarray*}
            ((u,L(-1)a))&=&B((L(-1)a)_1u,t)\\
            &=&B(u_0a,t)\\
            &=&B(\lambda {\bf 1}+\beta,t)\\
            &=&\lambda B({\bf 1},t)\neq 0.
        \end{eqnarray*} This is impossible. So, $\mathfrak{m}$ is a $V_1$-module. 

       Conversely, assume that $\mathfrak{m}$ is a $V_1$-module. Since for $u\in V_1$, $a\in V_0$, we have $((u,L(-1)a))=B(u_0a,t)=0$. Next, we assume that $V_0\neq \mathfrak{a}$. By Proposition \ref{V_0M}, we then have that $\mathfrak{a}$ is a proper ideal of $V_0$. In fact, $\mathfrak{a}\subseteq\mathfrak{m}$ because $\mathfrak{m}$ is a unique maximal ideal of $V_0$. This implies that for $u\in V_1$, $a\in V_0$, we have $((u,L(-1)a))=B(u_0a,t)=0$ and $L(-1)V_0\subseteq \rad ((~,~))$.

        \item Since $Ker ~L(-1)|_{V_0}=\mathbb{C}{\bf 1}$, $\dim V_0>1$, we then have that $L(-1)V_0\neq 0$. Assume that $\mathfrak{m}$ is a $V_1$-module. By $(1)$, we then have that $0\neq L(-1)V_0\subseteq \rad((~,~))$. Therefore, $((~,~))$ is degenerate.
    \end{enumerate} 
\end{proof}

For now we assume that $V_0\neq \mathfrak{a}$. By Proposition \ref{V_0M}, $\mathbb{C}t$ is a $V_1$-module, we define a linear map $\Psi:V_1\rightarrow\mathbb{C}t;v\mapsto v_0t$, and $$M=\{u\in V_1~|~u_0t=0\}.$$  Notice that for $b\in V_1$, $u\in M$ we have $(b_0u)_0t=b_0u_0t-u_0b_0t=0$. Hence, $b_0u\in M$ and $M$ is an ideal of $V_1$. In summary, we have the following statement.
\begin{prop}\label{solvable condition}
    Let $V=\bigoplus_{n=0}^{\infty}V_n$ be an $\mathbb{N}$-graded vertex operator algebra such that $V_0=\bigoplus_{d=0}^sV_0^d$ is a graded Gorenstein algebra with the unique maximal ideal $\mathfrak{m}$, $Soc(V_0)=\mathbb{C}t$ and $V_1$ is a solvable Leibniz algebra. We set $\mathfrak{a}=Span\{v_0a~|~v\in V_1,~a\in\mathfrak{m}\}$ and we define $M=\{u\in V_1~|~u_0t=0\}.$ If $\mathfrak{a}\neq V_0$, then $M$ is a solvable ideal of $V_1$. 
\end{prop}
\begin{rem} In Proposition \ref{solvable condition}, we assume that $V_1$ is solvable. However, we note that this assumption is natural. For instance, it was shown in Theorem 5.3 of \cite{KSY} that for an $\mathbb{N}$-graded vertex algebra $V=\bigoplus_{n=0}^{\infty}V_n$ such that $V_0$ is a Gorenstein algebra with the unique maximal ideal $\mathfrak{m}$ and $\Ker(D|_{V_0} ) = \mathbb{C}{\bf 1}$, if every subalgebra of the left-Leibniz algebra $V_1$ is closed under $\mathfrak{m}$, and $Leib(V_1) = D(V_0)$, then $V_1$ is a solvable Leibniz algebra.
\end{rem}

For the rest of this section, we assume that
\[
 (V=\bigoplus_{n=0}^{\infty}V_n,Y(\,\cdot\,,\,\cdot\,),{\bf1},\omega)
\]
is a regular $\mathbb N$-graded vertex operator algebra such that
\begin{enumerate}
\item $V_1$ is a solvable Leibniz algebra;
\item $V_0=\bigoplus_{d=0}^sV_0^d$ is a graded-Gorenstein local algebra
with unique maximal ideal $\mathfrak m$;
\item $\Ker ~ L(-1)|_{V_0}=\mathbb C{\bf1}$; and
\item $V_0\neq\mathfrak a$, where
$\mathfrak a=\operatorname{Span}\{v_0a\mid v\in V_1,\ a\in\mathfrak m\}$.
\end{enumerate}
By Proposition \ref{degenerate of (( , ))} and the preceding discussion,
the symmetric form $((\, ,\,))$ on $V_1$ is degenerate. We set
$$R=\rad((\, ,\,)).$$ Then there is a vector-space complement $\mathbb U$ such that $$V_1=R\oplus\mathbb U
 \text{ and }
 ((\, ,\,))|_{\mathbb U}\text{ is nondegenerate}.$$

We shall not require $\mathbb U$ to be anisotropic.  
Let $\mathfrak M$
denote the collection of Leibniz subalgebras $\mathcal W\subseteq\mathbb U$
for which $((\, ,\,))|_{\mathcal W}$ is nondegenerate.  All results below
are stated for a chosen nonzero $\mathcal M\in\mathfrak M$.  This
formulation avoids any use of an anisotropic complement over
$\mathbb C$.

\begin{lem}\label{Solvable Lie} Let $V=\bigoplus_{n=0}^{\infty}V_n$ be an $\mathbb{N}$-graded vertex operator algebra such that $V_1$ is solvable and $\Ker~L(-1)|_{V_0}=\mathbb{C}{\bf 1}$ and $V_0\neq \mathfrak{a}$. Let $\mathcal{M}$ be a Leibniz subalgebra of $V_1$ that is contained in $\mathfrak{M}$. Then $\mathcal{M}$ is actually a solvable Lie algebra. In addition, $u_1v\in\mathbb{C}{\bf 1}$ for all $u,v\in\mathcal{M}$.  
\end{lem}
\begin{proof}
    Let $u,v\in\mathcal{M}$. Since $L(-1)(u_1v)=u_0v+v_0u$ and $u_0v,v_0u\in\mathcal{M}$, we can conclude that $L(-1)u_1v\in rad((~,~))\cap\mathcal{M}\subseteq rad((~,~))\cap \mathbb U=\{0\}$. Hence, $L(-1)u_1v=0$. Using the fact that $\Ker ~L(-1)|_{V_0}=\mathbb{C}{\bf 1}$, we can conclude further that $u_1v\in\mathbb{C}{\bf 1}$, and $\mathcal{M}$ is a Lie algebra. Since $V_1$ is solvable and $\mathcal{M}$ is a Leibniz-subalgebra of $V_1$, we then have that $\mathcal{M}$ is solvable as well.
\end{proof}
\begin{rem} For this Lemma, the statements are true without the assumption that $V$ is regular. In fact, the statement still holds when we replace an $\mathbb{N}$-graded vertex operator algebra by an $\mathbb{N}$-graded vertex algebra and replace $L(-1)$ by $D$.
\end{rem}

\begin{lem}\label{regular-local-simple}
If $V=\bigoplus_{n=0}^{\infty}V_n$ is a regular $\mathbb{N}$-graded vertex operator algebra such that $V_0$ is local, then $V$ is simple.
\end{lem}
\begin{proof}
Because $V$ is regular, the adjoint weak $V$-module is a direct sum of
simple ordinary $V$-modules.  A $V$-submodule of the adjoint module is an
ideal of $V$.  Hence, if $V$ were not simple, we could write $ V=I_1\oplus\cdots\oplus I_s$, $s>1$, as a direct sum of nonzero ideals.  Writing
${\bf1}=e_1+\cdots+e_s$, $e_i\in I_i\cap V_0$, the ideal decomposition implies
$(e_i)_{-1}e_j=0\quad(i\neq j)$, $(e_i)_{-1}e_i=e_i$. Thus $V_0$ would contain a nontrivial idempotent.  This is impossible
because $V_0$ is local.  Therefore $s=1$ and $V$ is simple.
\end{proof}

Since $V$ is regular, it is rational and $C_2$-cofinite and has only
finitely many inequivalent simple ordinary modules (cf. Proposition \ref{Info about regularity}).  Let
$\{M^j\mid1\leq j\leq r\}$ be representatives of their isomorphism classes, with $M^1=V$.  We write
\[
 Y^j(u,x)=\sum_{n\in\mathbb Z}u^j_nx^{-n-1},
 \qquad
 Y^j(\omega,x)=\sum_{n\in\mathbb Z}L^j(n)x^{-n-2}.
\]
Each $M^j$ has a decomposition into finite-dimensional $L^j(0)$-eigenspaces $M^j=\bigoplus_{n=0}^{\infty}M^j_{n+\lambda_j}$, where $\lambda_j\in\mathbb Q$ is the conformal weight of $M^j$.

\begin{thm}\label{M-reductive}
Let $V=\bigoplus_{n=0}^{\infty}V_n$ be a regular $\mathbb{N}$-graded vertex operator algebra such that
$V_1$ is a solvable Leibniz algebra and $V_0$ is a graded-Gorenstein
local algebra. Assume $\Ker L(-1)|_{V_0}=\mathbb C\bf 1$ and $V_0\neq \mathfrak{a}$.
Let $\mathcal{M}\in\mathfrak{M}$, and assume $L(1)\mathcal{M}=0$. Then $\mathcal{M}$ is reductive. Since $\mathcal{M}$ is solvable, it is abelian.
Moreover, the vertex subalgebra $\langle \mathcal{M}\rangle$ generated by $\mathcal{M}$
is isomorphic to a Heisenberg vertex operator algebra of rank
$\dim \mathcal{M}$.
\end{thm}

\begin{proof}
By Lemma \ref{Solvable Lie}, $\mathcal{M}$ is a solvable Lie algebra and $u_1v=((u,v)){\bf 1}$ for $u,v\in \mathcal{M}$.

Suppose that $\mathcal{M}$ is not reductive. Since $\mathcal{M}$ is solvable, this implies that $\mathcal{M}$ is not abelian. Hence, there exists a nonzero vector $u\in Span\{b_0v~|~b,v\in \mathcal{M}\}\neq \{0\}$.

Let $W$ be any finite-dimensional $\mathcal{M}$-module. Since $\mathcal{M}$ is solvable,
Lie's theorem implies that there is a basis of $W$ with respect to which every element of $\mathcal{M}$ acts by an upper triangular matrix.
Consequently, every element of $Span\{b_0v~|~b,v\in \mathcal{M}\}\neq \{0\}$ acts by a strictly upper
triangular matrix. Hence
\begin{equation}\label{abeliansolvable}\tr_W\bigl(\rho(u)\rho(v)\bigr)=0
\qquad\text{for all }v\in\mathcal M.\end{equation}

Let $M^1,\ldots,M^r$ be the inequivalent simple $V$-modules.
Each homogeneous subspace of $M^j$ is finite-dimensional and is
stable under the zero-mode action of $\mathcal M$. Therefore the equation (\ref{abeliansolvable})
implies
\begin{equation}\label{abeliansolvable2}\tr_{M^j}o(u)o(v)q^{L(0)-c/24}=0
\qquad
(v\in\mathcal M,\ 1\leq j\leq r).
\end{equation}
Since $((\, ,\,))|_{\mathcal{M}}$ is nondegenerate and $u\neq0$, there exists
$v\in \mathcal{M}$ such that $((u,v))=1$. Thus $u_1v=\bf 1$. Since $L(1)u=L(1)v=0$, both $u$ and $v$ have square-bracket weight
one, and $u[1]v=\mathbf1$, $u[2k-1]v=0$ for $k\geq2$.

Zhu's two-point trace identity (cf. (\ref{ouov})) therefore gives $$\tr_{M^j}o(u)o(v)q^{L(0)-c/24}
=
Z_{M^j}(u[-1]v,\tau)-E_2(\tau)Z_{M^j}(\tau).$$
By (\ref{abeliansolvable2}), we have
\begin{equation}\label{abeliansolvable3}Z_{M^j}(u[-1]v,\tau)
=
E_2(\tau)Z_{M^j}(\tau).\end{equation}
Moreover, $L[0](u[-1]v)=2u[-1]v$.

Applying modular invariance under $S:\tau\longmapsto-\frac1\tau$ to (\ref{abeliansolvable3}), we obtain
$$
\tau^2\sum_{i=1}^r\rho_{ji}
 Z_{M^i}(u[-1]v,\tau)
=
E_2(-1/\tau)
\sum_{i=1}^r\rho_{ji}Z_{M^i}(\tau).
$$
Using (\ref{abeliansolvable3}) again and
$E_2(-1/\tau)
=
\tau^2E_2(\tau)-\frac{\tau}{2\pi i},
$
we obtain
$$
\frac{\tau}{2\pi i}
\sum_{i=1}^r\rho_{ji}Z_{M^i}(\tau)=0.
$$
But modular invariance of the characters gives
$
\sum_{i=1}^r\rho_{ji}Z_{M^i}(\tau)
=
Z_{M^j}(-1/\tau),
$
and the latter is nonzero. This is a contradiction. Therefore $\mathcal{M}$ is reductive. Since $\mathcal{M}$ is also solvable, $\mathcal{M}$ is
abelian.

Finally, let $\{u^1,\ldots,u^\ell\}$ be an orthonormal basis of $\mathcal M$
with respect to $((\, ,\,))$. Then
$$[u^i_m,u^j_n]
=\sum_{s=0}^{\infty}\binom{m}{s}(u^i_su^j)_{m+n-s}=m(u^i_1u^j)_{m+n-1}=
m\delta_{ij}\delta_{m+n,0},
$$
and the vacuum is annihilated by all nonnegative modes. Hence there
is a surjective homomorphism from the rank-$\ell$ Heisenberg vertex
operator algebra $M(1)$ onto $\langle \mathcal{M}\rangle$. Since $M(1)$ is
simple, this homomorphism is injective. Thus $\langle \mathcal{M}\rangle\cong M(1)$. \end{proof}


\subsection{Semisimplicity of the Heisenberg zero modes}

By Theorem \ref{M-reductive}, the Lie algebra $\mathcal{M}$ is abelian and
the vertex operator subalgebra generated by $\mathcal{M}$ is the rank
$\dim\mathcal{M}$ Heisenberg vertex operator algebra $M(1)$.  The existence
of this Heisenberg subalgebra does not by itself imply that the commuting
zero modes $u_0$, $u\in\mathcal{M}$, act semisimply.  We now isolate the
possible nonsemisimple part intrinsically, without assuming that
$o(\mathcal{M})=\{u_0\mid u\in\mathcal{M}\}$ is an algebraic Lie algebra.

Fix $u\in\mathcal{M}$.  Since $[L(0),u_0]=0$, the operator $u_0$ preserves
each finite-dimensional homogeneous subspace $V_n$.  Hence
\[
 V=\bigoplus_{\lambda\in\mathbb{C}}V_u[\lambda],
\]
where
\[
 V_u[\lambda]
 =
 \{a\in V\mid (u_0-\lambda {\bf 1}_{-1})^N a=0
       \text{ for some }N\geq1\}.
\]
On $V_u[\lambda]$ define
\[
 S_u=\lambda {\bf 1}_{-1},\qquad N_u=u_0-\lambda {\bf 1}_{-1}.
\]
Thus $u_0=S_u+N_u$ is the locally finite Jordan decomposition of $u_0$ on
$V$.

\begin{prop}\label{heisenberg-jordan-derivations}
Let $V=\bigoplus_{n=0}^{\infty}V_n$ be a regular $\mathbb{N}$-graded vertex operator algebra such that
$V_1$ is a solvable Leibniz algebra and $V_0$ is a graded-Gorenstein
local algebra. Assume $\Ker L(-1)|_{V_0}=\mathbb C\bf 1$ and $V_0\neq \mathfrak{a}$.
Let $\mathcal{M}\in\mathfrak{M}$ such that $L(1)\mathcal{M}=0$. Then for every
$u\in\mathcal{M}$, the operators $S_u$ and $N_u$ defined above are
grading-preserving derivations of $V$, and $N_u$ is locally nilpotent.
Moreover, for $v\in\mathcal{M},\ n\in\mathbb{Z}$, 
 $[S_u,v_n]=[N_u,v_n]=0$. In particular, $N_u|_{M(1)}=0$.
\end{prop}

\begin{proof}
Let $a\in V_u[\lambda]$ and $b\in V_u[\mu]$.  Since $u_0$ is a derivation, we have $u_0(a_nb)=a_nu_0b+(u_0a)_nb$, and $a_nb\in V_u[\lambda+\mu].$ It follows that $N_u(a_nb)=u_0(a_nb)-(\lambda+\mu)a_nb=(N_ua)_nb+a_n(N_ub)$, and similarly
$S_u(a_nb)=(S_ua)_nb+a_n(S_ub)$.
Hence $S_u,N_u\in{\Der}(V)$. Both preserve the $L(0)$-grading, and $N_u$ is
locally nilpotent because it is nilpotent on every finite-dimensional
generalized eigenspace contained in a homogeneous subspace.

Since $\mathcal{M}$ is abelian, $u_0v=0$ for every $v\in\mathcal{M}$.
Consequently, $[u_0,v_n]=(u_0v)_n=0$. For $w\in V_u[\lambda]$, there exists a positive number $r$ such that $(u_0-\lambda {\bf 1}_{-1})^rw=0$. Because $[u_0,v_n]=0$, we have $(u_0-\lambda {\bf 1}_{-1})^rv_nw=v_n(u_0-\lambda {\bf 1}_{-1})^rw=0$. Consequently, $v_nw\in V_u[\lambda]$. Thus every $v_n$ preserves the generalized eigenspaces of $u_0$ and commutes
with both Jordan parts. Hence $[S_u,v_n]=[N_u,v_n]=0$. Finally, $N_u\mathbf1=0$ and $N_u$ commutes with all negative Heisenberg
modes.  Since $M(1)$ is generated from $\mathbf1$ by these modes,
$N_u|_{M(1)}=0$.
\end{proof}

We define the Heisenberg vacuum space
$$\Omega_V
 =
 \{w\in V\mid v_nw=0
       \text{ for every }v\in\mathcal{M}\text{ and }n>0\}.$$ 
       Proposition \ref{heisenberg-jordan-derivations} implies $N_u(\Omega_V)\subseteq\Omega_V$.

The semisimplicity of Heisenberg zero modes is well understood in the classical strongly regular setting. In particular, the work of Dong--Mason (\cite{DM4}) and Mason (\cite{M}) shows that, under strong CFT-type hypotheses, the action of the reductive Lie algebra $V_1$ is completely reducible, and this semisimplicity is a basic input in the resulting Heisenberg weight decomposition and lattice-subalgebra theory. In \cite{Li3}, Li's Heisenberg vacuum-space and abelian-coset theory likewise provides the standard Fock-space framework once the relevant zero-mode action is semisimple. The situation considered here is different: $V$ is not assumed to be of CFT-type or self-contragredient, and $V_0$ may be a nontrivial graded-Gorenstein local algebra. Consequently, semisimplicity of the zero modes $u_0$ ($u\in\mathcal{M}$) cannot be taken for granted. Writing $u_0=S_u+N_u$ for the Jordan decomposition, the preceding results show that the nilpotent part $N_u$ commutes with the Heisenberg oscillator modes. Thus the possible obstruction to semisimplicity is carried entirely by the Heisenberg vacuum multiplicity space $\Omega_V$. This motivates the following condition, which isolates precisely the additional hypothesis needed to recover the semisimple charge theory in the present non-CFT-type setting.

\begin{dfn}
\label{omega-ss-definition}
We say that $V$ satisfies the \emph{Heisenberg vacuum semisimplicity
condition}, denoted by $(\Omega\text{-SS})$, if $u_0|_{\Omega_V}$ is semisimple for every $u\in\mathcal{M}$. Equivalently, $N_u|_{\Omega_V}=0$ for every $u\in\mathcal{M}$.\end{dfn}

\begin{prop}(Heisenberg vacuum-space criterion for semisimplicity)
\label{omega-semisimplicity-criterion}
Let $V=\bigoplus_{n=0}^{\infty}V_n$ be a regular $\mathbb{N}$-graded vertex operator algebra such that
$V_1$ is a solvable Leibniz algebra and $V_0$ is a graded-Gorenstein
local algebra. Assume $\Ker ~L(-1)|_{V_0}=\mathbb C\bf 1$ and $V_0\neq \mathfrak{a}$.
Let $\mathcal{M}\in\mathfrak{M}$ such that $L(1)\mathcal{M}=0$. The following are
equivalent:
\begin{enumerate}
\item $\mathcal{M}$ acts semisimply on $V$;
\item $V$ satisfies $(\Omega\text{-SS})$.
\end{enumerate}
\end{prop}

\begin{proof}
We only need to prove that the statement $(2)$ implies the statement $(1)$.  Fix $u\in\mathcal{M}$ and a
$u_0$-generalized $\lambda$-eigenspace  $V_u[\lambda]=\left(U(\widehat{\mathcal{M}}^-)\Omega_V\right)\cap V_u[\lambda]=U(\widehat{\mathcal{M}}^-)\Omega_{V,u}[\lambda]$ where $\Omega_{V,u}[\lambda]=\Omega_V\cap V_u[\lambda]$. Because $N_u$ commutes
with every negative Heisenberg mode and vanishes on $\Omega_V$,
under this identification $N_u=1\otimes N_u|_{\Omega_V}$ and Condition $(\Omega\text{-SS})$ says that $N_u|_{\Omega_V}=0$, we can conclude that $N_u=0$ on
$V_u[\lambda]$.  Since this holds for every generalized $\lambda$-eigenspace, $u_0$ is semisimple on $V$.  As $u\in\mathcal{M}$ was arbitrary,
$\mathcal{M}$ acts semisimply on $V$.
\end{proof}

\begin{thm}\label{nonselfdual-semisimplicity}
Let $V=\bigoplus_{n=0}^{\infty}V_n$ be a regular $\mathbb{N}$-graded vertex operator algebra such that
$V_1$ is a solvable Leibniz algebra and $V_0$ is a graded-Gorenstein
local algebra. Assume $\Ker ~L(-1)|_{V_0}=\mathbb C\bf 1$ and $V_0\neq \mathfrak{a}$.
Let $\mathcal{M}\in\mathfrak{M}$ such that $L(1)\mathcal{M}=0$. Assume that $V$ satisfies
$(\Omega\text{-SS})$.  Then $\mathcal{M}$ acts semisimply on $V$ and on
every simple ordinary $V$-module.  Since $V$ is regular, every weak
$V$-module is a direct sum of simple ordinary modules, and hence the
$\mathcal{M}$-action is semisimple on every weak $V$-module in this sense.
\end{thm}

\begin{proof}
By Proposition \ref{omega-semisimplicity-criterion}, $\mathcal{M}$ acts
semisimply on $V$. Now, we will show that $\mathcal{M}$ acts semisimply on every simple ordinary $V$-module. Since the zero modes commute and the form
$((\, ,\,))|_{\mathcal{M}}$ is nondegenerate, we write $V$ as a direct sum of $\mathcal{M}$-eigenspaces: $$V=\bigoplus_{\beta}V(\beta),$$ where $V(\beta)=
 \{a\in V\mid
   u_0a=((u,\beta))a\text{ for every }u\in\mathcal{M}\}$. Let $W$ be a simple ordinary $V$-module. For each $\alpha\in \mathcal{M}$, we define $$W[\alpha]=\{w\in W~|~\text{for }u\in\mathcal{M},~(u_0^w-((u,\alpha)){\bf 1}_{-1})^Nw=0\text{ for some }N\geq 1\}$$ and set $E_W(\alpha)=
 \{w\in W[\alpha]\mid
   u_0^Ww=((u,\alpha))w
   \text{ for every }u\in\mathcal{M}\}$. 

Let $\{u^1,\dots,u^{\ell}\}$ be a basis of $\mathcal{M}$. For $\alpha\in \mathcal{M}$, for $i\in\{1,\dots,\ell\}$, we set $T_i=u^i_0-((u^i,\alpha)){\bf 1}_{-1}$. Clearly, $[T_i,T_j]=0$. Let $0\neq w\in W[\alpha]$. There exists $r_1\geq 1$ such that $T_1^{r_1}w=0$ and $T_1^{r_1-1}w\neq 0$. Now, we set $w_1=T_1^{r_1-1}w$. There exists $r_2\geq 1$ such that $T_2^{r_2}w_1=0$ and $w_2=T_2^{r_2-1}w_1\neq 0$. Since $[T_1,T_2]=0$ and $T_1w_1=0$, we have $T_1w_2=T_1(T_2^{r_2-1}w_1)=T_2^{r_2-1}T_1(w_1)=0$. Continuing inductively, we obtain $0\neq w_{\ell}\in \bigcap_{i=1}^{\ell}\Ker~ T_i$. Since $\{u^1,...,u^{\ell}\}$ is a basis of $\mathcal{M}$, we can conclude that $T_uw_{\ell}=0$ for every $u\in\mathcal{M}$ and $w_{\ell}\in E_W(\alpha)$. Here, $T_u=u_0-((u,\alpha)){\bf 1}_{-1}$. Hence, whenever $W[\alpha]\neq0$,  $E_W(\alpha)\neq0$ as well. 

Now we set $E_W:=\bigoplus_{\alpha}E_W(\alpha)\neq0.$ For $a\in V(\beta)$, $w\in E_W(\alpha)$, and $u\in\mathcal{M}$, the
zero-mode commutator formula gives
$
 [u_0^W,a_n]=(u_0a)_n=((u,\beta))a_n.
$
Therefore $u_0^W(a_nw)
 =
 ((u,\alpha+\beta))a_nw$ and $a_nw\in E_W(\alpha+\beta)$. Thus $E_W$ is a nonzero $V$-submodule of $W$. However, since $W$ is simple, we can conclude that 
$E_W=W$. Moreover, all zero modes from $\mathcal{M}$ act semisimply on $W$.
The final assertion follows from regularity.
\end{proof}


\begin{rem}
\label{omega-obstruction}
Theorem \ref{nonselfdual-semisimplicity} identifies the obstruction to
semisimplicity intrinsically.  The possible nilpotent part of a Heisenberg
zero mode is invisible on $M(1)$ and is supported
entirely on the Heisenberg vacuum multiplicity spaces.  Thus
$(\Omega\text{-SS})$ is precisely the condition that removes this
obstruction.  
\end{rem}
\begin{dfn}\cite{FHL, Li-1} Let $(V,Y,{\bf 1})$ be a vertex operator algebra. A bilinear form $(~,~)$ on $V$ is said to be invariant if it satisfies the following conditions
$$(Y(u,z)a,b)=(a,Y(e^{zL(1)}(-z^{-2})^{L(0)}u,z^{-1})b)\text{ for }u,a,b\in V.$$

\end{dfn} 

Let us investigate properties of regular $\mathbb{N}$-graded vertex operator algebras that have a nondegenerate invariant bilinear form. 
\begin{lem}\label{nondegenerateVacuumspace}
Let $V=\bigoplus_{n=0}^{\infty}V_n$ be a regular $\mathbb{N}$-graded vertex operator algebra such that
$V_1$ is a solvable Leibniz algebra and $V_0$ is a graded-Gorenstein
local algebra. Assume $\Ker ~L(-1)|_{V_0}=\mathbb C\bf 1$ and $V_0\neq \mathfrak{a}$.
Let $\mathcal{M}\in\mathfrak{M}$ such that $L(1)\mathcal{M}=0$. Assume that \(V\) admits a nondegenerate invariant bilinear form $\langle~,~\rangle$.
Then the restriction of $\langle~,~\rangle$ to the Heisenberg vacuum
space $
\Omega_V=\{w\in V\mid u_nw=0
\text{ for all }u\in\mathcal M,\ n>0\}
$ is nondegenerate.
\end{lem}

\begin{proof}
For $u\in\mathcal{M}$, by invariance of $\langle~,~\rangle$, we have, for $n, \mathbb{Z}$, $
\langle u_na,b\rangle=-\langle a,u_{-n}b\rangle$.

Since $\langle\mathcal M\rangle\cong M(1)$, the usual Heisenberg
oscillator decomposition gives

$$
V=U(\widehat{\mathcal M}^{-})\Omega_V.
$$ Here $\widehat{\mathcal{M}}^{-}=\mathcal{M}\otimes t^{-1}\mathbb{C}[t^{-1}]$. Thus every \(v\in V\) may be written as

$$
v=v_0+\sum_j
u^{j_1}_{-n_1}\cdots
u^{j_{r_j}}_{-n_{r_j}}v_j,
$$ where $v_0,v_j\in\Omega_V$, $n_i>0$, and $r_j\ge1$.

Let $0\neq w\in\Omega_V$. Since the invariant bilinear form on $V$
is nondegenerate, there exists $v\in V$ such that
$\langle w,v\rangle\neq 0$.
For every $u\in\mathcal M$, $n>0$, and $a\in V$, $
\langle w,u_{-n}a\rangle
=-\langle u_nw,a\rangle=0,
$ because $w\in\Omega_V$. Repeatedly applying this identity shows that $w$ is orthogonal to every term containing at least one negative
Heisenberg mode. Hence $
0\neq\langle w,v\rangle=\langle w,v_0\rangle
$ for some \(v_0\in\Omega_V\).
Thus no nonzero vector of \(\Omega_V\) lies in the radical of the
restricted form. Since the form is symmetric, its restriction to
\(\Omega_V\) is nondegenerate.
\end{proof}

We continue to investigate properties of regular $\mathbb{N}$-graded vertex operator algebras that have a nondegenerate invariant bilinear form. As a reminder we let $V=\bigoplus_{n=0}^{\infty}V_n$ be a regular $\mathbb{N}$-graded vertex operator algebra such that
$V_1$ is a solvable Leibniz algebra and $V_0$ is a graded-Gorenstein
local algebra. In addition, we assume $\Ker ~L(-1)|_{V_0}=\mathbb C\bf 1$ and $V_0\neq \mathfrak{a}$.
Also, we let $\mathcal{M}\in\mathfrak{M}$ such that $L(1)\mathcal{M}=0$, and let $\langle~,~\rangle$ be a nondegenerate invariant bilinear form on $V$. Then, by invariance of $\langle~,~\rangle$, we have that for $u\in\mathcal{M}$, $\langle u_0a,b\rangle=-\langle a,u_0b\rangle$. Consequently, for every polynomial $p(x)$ we have  
$$\langle p(u_0)a,b\rangle=\langle a,p(-u_0)b\rangle.$$ In particular, we have that for $i\geq 1$, $\langle (u_0)^ia,b\rangle=(-1)^i\langle a,(u_0)^ib\rangle$. 

Let us assume that $a\in V_u[\lambda]$ and $b\in V_u[\mu]$. Then there are positive integers $r,s$ such that $(u_0-\lambda {\bf 1}_{-1})^ra=0$, and $(u_0-\mu {\bf 1}_{-1})^sb=0$. Now, let us take $p(x)=(x-\lambda)^r$. Hence, we have 
$$\langle (u_0-\lambda {\bf 1}_{-1})^ra,b\rangle=(-1)^r\langle a,(u_0+\lambda {\bf 1}_{-1})^rb\rangle.$$ This implies that $\langle a,(u_0+\lambda {\bf 1}_{-1})^rb\rangle=0$. Now, let us observe that $$u_0+\lambda {\bf 1}_{-1}=(\lambda+\mu){\bf 1}_{-1}+(u_0-\mu {\bf 1}_{-1}),$$ and $u_0-\mu {\bf 1}_{-1}$ is nilpotent. If $\lambda+\mu\neq 0$, then $u_0+\lambda {\bf 1}_{-1}$ is invertible on $V_u[\mu]$ and $$(u_0+\lambda {\bf 1}_{-1})^{-1}=\frac{1}{\lambda +\mu}\sum_{j=0}^{s-1}\left(-\frac{u_0-\mu {\bf 1}_{-1}}{\lambda +\mu}\right)^j.$$ $(u_0+\lambda {\bf 1}_{-1})^r$ is invertible on $V_u[\mu]$ as well. This implies that there exists $v\in V_u[\mu]$ such that $b =(u_0+\lambda)^rv$. Consequently, we have 
$$\langle a,b\rangle=\langle a, (u_0+\lambda{\bf 1}_{-1})^rv\rangle=(-1)^r\langle (u_0-\lambda {\bf 1}_{-1})^ra,v\rangle=0.$$ Hence, if $\lambda+\mu\neq 0$ then $\langle V_u[\lambda],V_u[\mu]\rangle=0$.

By Jordan decomposition, we observe that for each $n\geq 0$, 
$$u_0|_{V_n}=S+N,
~
S=S_u|_{V_n},
~
N=N_u|_{V_n},$$
where $S$ is semisimple, $N$ is nilpotent, and
$[S,N]=0$. Because the invariant bilinear form is nondegenerate and the
$L(0)$-grading is orthogonal, its restriction to each $V_n$ is
nondegenerate. Hence, adjoints are well-defined on $V_n$.  Taking
adjoints in
$u_0|_{V_n}=S+N$ gives $\left(u_0|_{V_n}\right)^*=S^*+N^*$. The operator $S^*$ is semisimple, since $S$ is semisimple, and $N^*$ is nilpotent, since $N$ is nilpotent.  Furthermore,
$$[S^*,N^*]
=
-[S,N]^*
=
0.$$
Consequently,
$\left(u_0|_{V_n}\right)^*=S^*+N^*$
is the additive Jordan decomposition of $\left(u_0|_{V_n}\right)^*$.

On the other hand, since $\left(u_0|_{V_n}\right)^*=-u_0|_{V_n}=(-S)+(-N)$.
Here $-S$ is semisimple, $-N$ is nilpotent, and
$[-S,-N]=[S,N]=0$. Therefore, $\left(u_0|_{V_n}\right)^*=(-S)+(-N)$ is also an additive Jordan decomposition of $\left(u_0|_{V_n}\right)^*$.

By uniqueness of the additive Jordan decomposition on the
finite-dimensional vector space $V_n$, we obtain
$S^*=-S$, and $N^*=-N$. Thus, for all $a,b\in V_n$, $\langle S_u a,b\rangle=-\langle a,S_u b\rangle$, and $\langle N_u a,b\rangle=-\langle a,N_u b\rangle$. Since $n$ was arbitrary and $V=\bigoplus_{n\geq0}V_n$, these identities hold on every element of $V$. 

Now, for each $\lambda\in \mathbb{C}$, we set $$\Omega_{V,u}[\lambda]=\Omega_V\cap V_{u}[\lambda].$$ By Lemma \ref{nondegenerateVacuumspace}, the preceding results restrict to $\Omega_V$. Let us summarize this discussion in the following Theorem.

\begin{thm}
\label{selfdual-charge-duality}
Let $V=\bigoplus_{n=0}^{\infty}V_n$ be a regular $\mathbb{N}$-graded vertex operator algebra such that
$V_1$ is a solvable Leibniz algebra and $V_0$ is a graded-Gorenstein
local algebra. Assume $\Ker ~L(-1)|_{V_0}=\mathbb C\bf 1$ and $V_0\neq \mathfrak{a}$.
Let $\mathcal{M}\in\mathfrak{M}$ such that $L(1)\mathcal{M}=0$, and let $\langle~,~\rangle$ be a
nondegenerate invariant bilinear form on $V$. Then for every
$u\in\mathcal{M}$ the $u_0$-generalized eigenspaces satisfy $$ \langle V_u[\lambda],V_u[\mu]\rangle=0
 \text{ whenever }\lambda+\mu\neq0.$$ The restriction of the invariant bilinear form induces a nondegenerate
pairing $$ V_u[\lambda]\times V_u[-\lambda]\longrightarrow\mathbb{C},$$
and the corresponding nilpotent Jordan parts are skew-adjoint: $\langle N_ua,b\rangle=-\langle a,N_ub\rangle$. In particular, the same conclusions hold on the Heisenberg vacuum spaces $\Omega_{V,u}[\lambda]\times\Omega_{V,u}[-\lambda].$

If, in addition, $V$ satisfies $(\Omega\text{-SS})$, then the generalized
eigenspaces are ordinary eigenspaces and $$\langle V_u(\lambda),V_u(\mu)\rangle=0
 \text{ unless }\lambda+\mu=0.$$ Thus opposite ordinary $\mathcal{M}$-eigenspaces occur in nondegenerately paired
spaces.
\end{thm}

\begin{rem}
The nondegeneracy assumption is important in the preceding argument:
it allows the adjoint operators $S_u^*$ and $N_u^*$ to be defined
uniquely and permits the use of uniqueness of the additive Jordan
decomposition.  If the invariant bilinear form is allowed to be
degenerate, the generalized eigenspace orthogonality
$\langle V_u[\lambda],V_u[\mu]\rangle=0$ when $\lambda+\mu\neq0$ can still be proved directly from
$\langle u_0a,b\rangle=-\langle a,u_0b\rangle$ by a polynomial argument, without introducing adjoint operators.
\end{rem}
\begin{rem}
The preceding results separate two distinct phenomena.  The existence of
the Heisenberg vertex operator algebra $M(1)$ follows from the local
$V_0$--$V_1$ structure developed above, whereas semisimplicity of the
Heisenberg zero modes is controlled entirely by the multiplicity spaces
$\Omega_V$.  In the degenerate case, $(\Omega\text{-SS})$ is
therefore a natural structural hypothesis.  In the nondegenerate case,
one gains the additional opposite-eigenvalue duality of
Theorem \ref{selfdual-charge-duality}, but not semisimplicity automatically.
\end{rem}

\subsection{Integral $ \mathcal{M}$-weights and a lattice subalgebra}

The approach in this subsection is inspired in substantial part by Mason's lattice-subalgebra theory for strongly regular vertex operator algebras \cite{M}. In the strongly regular setting, Mason studies the simultaneous weights of the zero modes associated with a nondegenerate abelian subspace of \(V_1\), the resulting weight group, finite-index subgroups arising from module deformations, and the reconstruction of lattice vertex operator subalgebras. We adapt this circle of ideas to the present $\mathbb{N}$-graded, non-CFT-type setting. Several ingredients that are automatic or available from strong regularity in Mason's setting need to be isolated here as additional structural conditions, most notably the Heisenberg vacuum semisimplicity condition $(\Omega\text{-SS})$, the full-rank integrality condition $(\mathrm{IL})$, and the lattice cocycle compatibility condition $(\mathrm{LC})$.

For the remainder of this subsection, we let $V=\bigoplus_{n=0}^{\infty}V_n$ be a regular $\mathbb{N}$-graded vertex operator algebra such that
$V_1$ is a solvable Leibniz algebra and $V_0$ is a graded-Gorenstein
local algebra, $\Ker ~L(-1)|_{V_0}=\mathbb C\bf 1$ and $V_0\neq \mathfrak{a}$.
Let $\mathcal{M}\in\mathfrak{M}$ such that $L(1)\mathcal{M}=0$, and assume $(\Omega\text{-SS})$.  Thus the
commuting zero modes from $\mathcal M$ act semisimply on $V$ and on every
simple ordinary $V$-module.  We write
\[
 \langle u,v\rangle_{\mathcal M}:=((u,v))
 \qquad (u,v\in\mathcal M)
\]
for the nondegenerate Heisenberg form.

Let $P_V=\{\beta\in\mathcal M\mid V(\beta)\neq0\}$, where $V(\beta)=
 \{a\in V\mid
 u_0a=\langle u,\beta\rangle_{\mathcal M}a
 \text{ for every }u\in\mathcal M\}$. We call $\beta$ a $\mathcal{M}$-weight if $V(\beta)\neq 0$, and call $V(\beta)$ the corresponding $\mathcal{M}$-weight space. Thus $P_V$ is the set of $ \mathcal {M} $-weights occurring in $V$. This terminology reflects the fact that for $u\in\mathcal{M}$, the commuting zero modes $u_0$ act on $V(\beta)$ through the linear functional $ u\longmapsto\langle u,\beta\rangle_{\mathcal M}$. Since $\langle~,~\rangle_{\mathcal M}$ is nondegenerate, this functional is uniquely represented by the vector \(\beta\in\mathcal M\).

By Proposition \ref{PV}, Proposition \ref{PW}, simplicity of $V$, semisimplicity of the
$\mathcal M$-action, and $L(n)\mathcal M=0$ for $n\geq1$ imply that
$P_V$ is an additive subgroup of $\mathcal M$ and
\[
 V=\bigoplus_{\beta\in P_V}
 M(1,\beta)\otimes\Omega_V(\beta),
\]
where
\[
 \Omega_V(\beta)=
 \{w\in V(\beta)\mid u_nw=0
 \text{ for all }u\in\mathcal M,\ n>0\}.
\]
Moreover,
\[
 \Omega_V(0)=C_V(M(1)),
 \qquad
 V(0)=M(1)\otimes\Omega_V(0),
\]
and each $V(\beta)$ is an irreducible $V(0)$-module.

Define
\[
 P_V^\circ=
 \{u\in\mathcal M\mid
 \langle u,\beta\rangle_{\mathcal M}\in\mathbb Z
 \text{ for all }\beta\in P_V\}.
\]
Also set
\[
 \mathcal L_0=
 \{u\in\mathcal M\mid
 u_0\text{ has integral eigenvalues on }V\}.
\]
Since the $\mathcal M$-action on $V$ is semisimple and the eigenvalues of
$u_0$ on $V(\beta)$ are exactly
$\langle u,\beta\rangle_{\mathcal M}$, we have
\begin{equation}\label{PVdualL0}
 P_V^\circ=\mathcal L_0.
\end{equation}

The following full-rank condition isolates the arithmetic input needed for
the lattice construction.

\begin{dfn}[Integral $\mathcal{M}$-weight lattice condition]
\label{IL-definition}
We say that $(V,\mathcal{M})$ satisfies \emph{(IL)} if
$\mathcal{L}_0=P_{V}^{\circ}$ is a full-rank lattice such that $\mathcal{M}_{\mathbb{R}}:=
 \mathbb{R}\otimes_{\mathbb{Z}}\mathcal{L}_0
 \subseteq\mathcal{M}$; equivalently,
$\operatorname{rank}_{\mathbb{Z}}\mathcal{L}_0=\dim_{\mathbb{C}}\mathcal{M}$ and $\mathbb{C}\otimes_{\mathbb{Z}}\mathcal{L}_0=\mathcal{M}$.
\end{dfn}

We set $\mathcal E=
 \{u\in\mathcal M\mid
 u_0\text{ has rational eigenvalues on }V\}.$

\begin{lem}\label{rational-norm-E}
Let $V=\bigoplus_{n=0}^{\infty}V_n$ be a regular $\mathbb{N}$-graded vertex operator algebra such that
$V_1$ is a solvable Leibniz algebra and $V_0$ is a graded-Gorenstein
local algebra, $\Ker ~L(-1)|_{V_0}=\mathbb C\bf 1$ and $V_0\neq \mathfrak{a}$.
Let $\mathcal{M}\in\mathfrak{M}$ such that $L(1)\mathcal{M}=0$, and assume $(\Omega\text{-SS})$ and (IL). If $u\in\mathcal E$, then $\langle u,u\rangle_{\mathcal M}\in\mathbb Q$.
\end{lem}

\begin{proof}
First, we let $u\in\mathcal L_0=P_V^\circ$.  Li's $\Delta$-operator
construction gives, for every weak $V$-module $W$, the weak module
$W^{(u)}$; see Proposition \ref{Keven}.  Since twisting by $u$ is inverted
by twisting by $-u$, irreducibility is preserved.  In particular,
$V^{(u)}$ is an irreducible ordinary $V$-module.

For $L(1)u=0$ and $u_1u=\langle u,u\rangle_{\mathcal M}{\bf1}$, the standard
$\Delta$-operator computation gives $$\Delta_u(x)\omega
 =
 \omega+u\,x^{-1}
 +\frac12\langle u,u\rangle_{\mathcal M}x^{-2}{\bf1}.$$ Consequently, the vacuum vector of the underlying vector space $V$, viewed
inside the module $V^{(u)}$, has conformal weight $\frac12\langle u,u\rangle_{\mathcal M}$. All conformal weights of simple $V$-modules are rational because $V$ is
regular.  Hence for $u\in\mathcal L_0$, $\langle u,u\rangle_{\mathcal M}\in\mathbb Q$.

Now let $u\in\mathcal E$.  Since $V$ is finitely generated, there exists
$N>0$ such that $Nu_0$ has integral eigenvalues on a finite homogeneous
generating space, and hence on all of $V$.  Thus $Nu\in\mathcal L_0$.
Therefore $N^2\langle u,u\rangle_{\mathcal M}
 =
 \langle Nu,Nu\rangle_{\mathcal M}\in\mathbb Q$, which proves the claim.
\end{proof}

Let $\operatorname{Irr}(V)$ denote the finite set of isomorphism classes of simple ordinary \(V\)-modules. For a simple ordinary \(V\)-module \(W\), we write $[W]\in \operatorname{Irr}(V)$ for the isomorphism class of $W$. For $\alpha\in P_V^{\circ}$, we let $W^{(\alpha)}$ denote the $\Delta_{\alpha}$-twist of $W$. Then
$$
\alpha\cdot[W]:=[W^{(\alpha)}]
$$
defines an action of \(P_V^\circ\) on \(\operatorname{Irr}(V)\).

We define
$$K=
 \{\alpha\in P_V^\circ\mid
 V^{(\alpha)}\cong V\text{ as a $V$-module}\}.$$

\begin{prop}\label{K-lattice-repaired}
Let $V=\bigoplus_{n=0}^{\infty}V_n$ be a regular $\mathbb{N}$-graded vertex operator algebra such that
$V_1$ is a solvable Leibniz algebra and $V_0$ is a graded-Gorenstein
local algebra, $\Ker ~L(-1)|_{V_0}=\mathbb C\bf 1$ and $V_0\neq \mathfrak{a}$.
Let $\mathcal{M}\in\mathfrak{M}$ such that $L(1)\mathcal{M}=0$, and assume $(\Omega\text{-SS})$ and (IL). Then $K$ has finite index in $P_V^\circ$ and is a full-rank
positive-definite even integral lattice.  Moreover,
\[
 K\subseteq P_V\subseteq K^\circ,
\]
and hence $P_V/K$ is finite.
\end{prop}

\begin{proof}
Recall that $P_V^{\circ}$ acts on  $\operatorname{Irr}(V)$ and $\operatorname{Irr}(V)$ is finite. In addition, $K=Stab_{P_V^{\circ}}([V])$ (i.e., the stabilizer of $[V]$ for an action on
a finite set). By the orbit-stabilizer correspondence for group actions, we have $P_V^{\circ}/K \longleftrightarrow P_V^{\circ}\cdot[V],\quad\alpha+K\mapsto [V^{(\alpha)}]$. Thus $$[P_V^\circ:K]=|P^{\circ}_V\cdot [V]|\leq |\operatorname{Irr}(V)|<\infty.$$

For simplicity, we set $m=[P_V^\circ:K]$. Then for every $\alpha\in P^{\circ}_V$, the coset $\alpha+K$ has finite order in $P^{\circ}_V/K$. Since $P^{\circ}_V/K$ has order $m$, we can conclude that $m\alpha\in K$. Thus $mP^{\circ}_V\subseteq K\subseteq P^{\circ}_V$. Hence, $K$ and $P^{\circ}_V$ span the same vector space. In addition, $Span_{\mathbb{R}}K=Span_{\mathbb{R}}P^{\circ}_V$. Since $P^{\circ}_V$ is a full-rank lattice, this implies that $K$ is also a full-rank lattice.

Now, we will show that $K\subseteq P_V\subseteq K^{\circ}$. Let $\alpha\in K$. Notice that for $u\in\mathcal{M}$, we have $\alpha_nu=0$ for $n\geq 2$, $\alpha_1u=\langle\alpha,u\rangle_{\mathcal{M}}{\bf 1}$. In addition, $$\Delta_{\alpha}(x)u=u+\langle \alpha,u\rangle_{\mathcal{M}}{\bf 1}x^{-1}.$$ Hence, $Y^{(\alpha)}(u,x)=Y(u,x)+\langle\alpha,u\rangle_{\mathcal{M}}Y({\bf 1},x)x^{-1}$. We set $Y^{(\alpha)}(u,x):=\sum_{n\in\mathbb{Z}}u^{(\alpha)}_nx^{-n-1}$. We have $$u_0^{(\alpha)}=u_0+\langle\alpha,u\rangle_{\mathcal{M}}{\bf 1}_{-1}.$$ Applying to ${\bf 1}$, we have $u_0^{(\alpha)}{\bf 1}=\langle\alpha,u\rangle_{\mathcal{M}}{\bf 1}.$ Since $V^{(\alpha)}\cong V$, it implies that there exists $\phi_{\alpha}:V^{(\alpha)}\rightarrow V$ such that $\phi_{\alpha}(Y_{V^{(\alpha)}}(u,x)v)=Y_V(u,x)\phi_{\alpha}(v)$ for $u,v\in V$. We set $$e^\alpha:=\phi_{\alpha}({\bf1}).$$ Observe that $$u_0e^{\alpha}=u_0\phi_{\alpha}({\bf 1})=\phi_{\alpha}(u^{(\alpha)}_0{\bf 1})=\langle\alpha,u\rangle \phi_{\alpha}({\bf 1})=\langle\alpha,u\rangle e^{\alpha}.$$ Hence for every $u\in\mathcal{M}$, $u_0e^{\alpha}=\langle\alpha,u\rangle e^{\alpha}$ and $e^{\alpha}\in V(\alpha)$. Because $e^{\alpha}\neq 0$, we then have $V(\alpha)\neq \{0\}$ and $\alpha\in P_V$. So, $K\subseteq P_V$. Since $K\subseteq P_V^\circ$, for $\alpha\in K$ and $\beta\in P_V$, we have $\langle\alpha,\beta\rangle_{\mathcal M}\in\mathbb{Z}$. Hence $P_V\subseteq K^\circ$. In summary, $$K\subseteq P_V\subseteq K^\circ.$$

Now, we will show that $K$ is a positive definite even lattice. The same $\Delta$-operator calculation as in
Lemma \ref{rational-norm-E} shows that the conformal weight of $e^{\alpha}$ is $$\wt(e^\alpha)=\wt \phi_{\alpha}({\bf 1})=\frac12\langle\alpha,\alpha\rangle_{\mathcal M}.$$ Because $e^\alpha\in V$ and $V$ is $\mathbb{N}$-graded, we can conclude that $$\frac{1}{2}\langle\alpha,\alpha\rangle_{\mathcal M}\in\mathbb{N}.$$ Thus $\langle\alpha,\alpha\rangle_{\mathcal M}\in2\mathbb{N}.$ This implies that $K$ is integral and $\langle\gamma,\beta\rangle_{\mathcal{M}}\in\mathbb{Z}$ for all $\gamma,\beta\in K$.

Assume that the rank of the lattice $K$
is $\ell$. We let $\alpha_1,\dots,\alpha_{\ell}$ be a $\mathbb{Z}$-basis of $K$  so that $K=\mathbb{Z}\alpha_1\oplus\dots\oplus\mathbb{Z}\alpha_{\ell}$. $K$ has full rank in $\mathcal{M}$, its $\mathbb{Z}$-basis is a $\mathbb{C}$-basis of $\mathcal{M}$. Since $\langle~,~\rangle_{\mathcal{M}}$ is nonsingular, the matrix $G_K=\left(\langle\alpha_i,\alpha_j\rangle_{\mathcal{M}}\right)_{i,j}$ is nonsingular. Suppose that there exists $\alpha\in K$ such that $\langle\alpha,\alpha\rangle_{\mathcal{M}}=0$. Then for any $\beta\in K$, $n\in\mathbb{Z}$, since $\beta+n\alpha\in K$, we have $$0\leq \langle\beta+n\alpha, \beta+n\alpha\rangle_{\mathcal{M}}=\langle\beta,\beta\rangle_{\mathcal{M}}+2n\langle\alpha,\beta\rangle_{\mathcal{M}}+\langle\alpha,\alpha\rangle_{\mathcal{M}}=\langle\beta,\beta\rangle_{\mathcal{M}}+2n\langle\alpha,\beta\rangle_{\mathcal{M}}.$$ Because the statement holds for all $n\in\mathbb{Z}$, we can conclude that $\langle\alpha,\beta\rangle_{\mathcal{M}}=0$ for every $\beta\in K$. However, since $G_K$ is nonsingular, we conclude that $\alpha=0$. Therefore, $$\langle\alpha,\alpha\rangle_{\mathcal{M}}>0\text{ for every nonzero }\alpha\in K.$$ Therefore, the bilinear form $\langle~,~\rangle_{\mathcal{M}}$ is positive definite on $\mathbb{R}\otimes_{\mathbb{Z}}K$. Since $K$ is even and integral, $K$ is a full-rank positive-definite even integral lattice as desired. \end{proof}

The remaining step is the standard reconstruction of the lattice vertex
operator algebra from the simple-current translations.  We state explicitly
the cocycle normalization used in the Heisenberg-coset results recalled
above.

\begin{dfn}[Lattice cocycle compatibility]
\label{LC-definition}
We say that $K$ satisfies \emph{(LC)} if the isomorphisms
$\phi_\alpha:V^{(\alpha)}\to V$ may be chosen so that
\[
 \phi_\alpha\phi_\beta
 =
 \varepsilon_K(\alpha,\beta)\phi_{\alpha+\beta}
 \qquad(\alpha,\beta\in K),
\]
where $\varepsilon_K$ is a normalized lattice $2$-cocycle with
\[
 \frac{\varepsilon_K(\alpha,\beta)}
 {\varepsilon_K(\beta,\alpha)}
 =
 (-1)^{\langle\alpha,\beta\rangle_{\mathcal M}}.
\]
\end{dfn}

\begin{thm}\label{lattice-subalgebra-repaired}
Let $V=\bigoplus_{n=0}^{\infty}V_n$ be a regular $\mathbb{N}$-graded vertex operator algebra such that
$V_1$ is a solvable Leibniz algebra and $V_0$ is a graded-Gorenstein
local algebra, $\Ker ~L(-1)|_{V_0}=\mathbb C\bf 1$ and $V_0\neq \mathfrak{a}$.
Let $\mathcal{M}\in\mathfrak{M}$ such that $L(1)\mathcal{M}=0$, and assume $(\Omega\text{-SS})$, (IL), and (LC).  Then $V$ contains a conformal vertex
operator subalgebra isomorphic to the lattice vertex operator algebra
$V_K$, where $K$ is a positive-definite even lattice of rank
$\dim\mathcal M$.  Moreover,
\[
 \min\{\langle\alpha,\alpha\rangle_{\mathcal M}
       \mid0\neq\alpha\in K\}\geq4.
\]
\end{thm}

\begin{proof}
By Theorem \ref{M-reductive}, $\langle\mathcal M\rangle\cong M(1)$. By $(\Omega\text{-SS})$ and the Heisenberg-coset decomposition,
$$V=
 \bigoplus_{\beta\in P_V}
 M(1,\beta)\otimes\Omega_V(\beta),
 ~\Omega_V(0)=C_V(M(1)).$$ Proposition \ref{K-lattice-repaired} shows that $K$ is a
positive-definite even lattice of full rank and that $P_V$ is a
nondegenerate rational lattice of finite rank.  Under (LC), applying Proposition \ref{lattice construction}, we have $$C_V(\Omega_V(0))\cong V_K.$$ In particular, $V_K$ is a conformal vertex operator subalgebra of $V$ with
Heisenberg conformal vector $\omega_{\mathcal M}
 =
 \frac12\sum_{i=1}^{\ell}u^i_{-1}u^i_{-1}{\bf1},$
where $\{u^1,\ldots,u^{\ell}\}$ is an orthonormal basis of $\mathcal M$.

It remains to prove the minimum-norm assertion.  If
$0\neq\alpha\in K$ had $\langle\alpha,\alpha\rangle_{\mathcal M}=2$, then the weight-one space of $V_K$ would contain $e^\alpha$, $e^{-\alpha}$, and $\alpha_{-1}{\bf1}$, which generate a copy of $\mathfrak{sl}_2$.  Since $V_K$ is conformally
embedded in $V$, this copy of $\mathfrak{sl}_2$ lies in $V_1$. But every
Leibniz subalgebra of the solvable Leibniz algebra $V_1$ is solvable,
whereas $\mathfrak{sl}_2$ is not. This contradiction proves that no
norm-two vector occurs. Since $K$ is positive definite and even, its
minimum norm is at least $4$.
\end{proof}

\begin{rem}
The hypotheses $(\Omega\text{-SS})$, (IL), and (LC) have distinct roles.
The first removes the Jordan obstruction for the Heisenberg zero modes;
the second supplies the full-rank integral lattice; and the third is
the cocycle compatibility needed to reconstruct the lattice VOA from the
simple-current translations.  None of these three points is concealed in
an algebraicity assumption on the exponential zero-mode group.
\end{rem}


\section{Conformal shifts, Heisenberg actions, and hidden Lie structures}\label{sec:hidden-sl2-shifted-lattice}

We illustrate the results of Sections 3 and 4 using conformally shifted lattice vertex operator algebras. The examples show that several structures which are closely related in the strongly regular setting can separate for regular $\mathbb{N}$-graded vertex operator algebras. 

A conformal shift changes the grading of a vertex operator algebra without changing its underlying vertex algebra. Consequently, vectors that form a familiar Lie-theoretic configuration in the original grading may occur in different homogeneous subspaces after the shift. In particular, a semisimple Lie algebra need not remain inside $V_1$, even though suitable modes of its constituent vectors may continue to generate the same Lie algebra. This gives rise to the hidden $\mathfrak{sl}_2$-structures studied below. 

At the same time, the quasi-primary subspace $\mathcal{M}$ of Section 4 may detect a positive-definite even lattice different from the lattice used in the original shifted construction. Thus, the examples distinguish three related but genuinely different structures: 

\boxed{$$
\text{\small{weight-one Lie structure}}
\neq
\text{\small{mode-theoretic Lie structure}}
\neq
\text{\small{lattice structure detected by }}\mathcal M.
$$}

A second theme concerns the semisimplicity of the zero-mode actions of $\mathcal{M}$. The shifted lattice models first provide examples in which the relevant zero modes act semisimply and the lattice construction of Section 4 can be carried out explicitly. We then consider regular examples in which a rank-one Heisenberg vertex operator subalgebra has a nonsemisimple zero-mode action. These examples show, with progressively stronger hypotheses, that regularity alone does not imply the condition $(\Omega\text{-SS})$ introduced in Section 4. 

Finally, we use these examples to clarify the roles of regularity, self-contragredience, and invariant bilinear forms. In particular, the examples indicate which conclusions of the strongly regular theory persist in the more general $\mathbb{N}$-graded setting and which require additional hypotheses.

\subsection{Shifted vertex operator algebras and the non-CFT-type setting}

Following Dong and Mason \cite{DM2}, we use conformal shifts, which show that regularity alone does not impose the usual CFT-type conditions or require a vertex operator algebra to be self-contragredient. Here, as throughout this paper, we say that $V$ is \emph{self-contragredient} if $$
V\cong V'
$$
as $V$-modules, where $V'$ denotes the contragredient module.  This property is referred to as \emph{self-duality} in \cite{DM2}.

Here conformal shifts serve as explicit models for the phenomena of Sections~3 and~4. A conformal shift changes the grading without changing the underlying vertex algebra, and consequently vectors belonging to a familiar Lie-theoretic configuration may move into different homogeneous subspaces.  In particular, a semisimple Lie algebra that is visible in degree one before the shift need not be contained in $V_1$ after the shift.  Its mode operators may nevertheless continue to generate the same Lie algebra.  This leads to the hidden $\mathfrak{sl}_2$-structures studied below.

At the same time, the Heisenberg subspace satisfying the quasi-primary condition of Section~4 can detect a positive-definite even lattice different from the lattice with which the shifted construction began.  Thus shifted lattice vertex operator algebras provide a useful setting in which the degree-one Lie algebra, the Lie algebra generated by suitable mode operators, and the conformally embedded lattice vertex operator algebra can be distinguished explicitly.

Throughout this subsection, $ V=\bigoplus_{n\geq 0}V_n$ is an $\mathbb{N}$-graded vertex operator algebra with finite-dimensional homogeneous subspaces.

\subsubsection{A general mode-theoretic $\mathfrak{sl}_2$ construction}
\label{subsec:general-hidden-sl2}

We begin with a construction that does not assume that the vectors playing
the roles of the raising and lowering operators belong to $V_1$.

Let $ e\in V_p$, $f\in V_s$, $h\in V_1$. We set $ N=p+s-2$. Now, we assume that
\begin{equation}
 h=e_Nf,\qquad h_0e=re,\qquad h_0f=-rf,
 \qquad r\neq0.
 \label{eq:general-sl2-data}
\end{equation}
We also assume
\begin{equation}
 h_1h=d{\bf1},~ h_0h=0.
 \label{eq:general-heisenberg-data}
\end{equation} Here, $d\in\mathbb{C}$. For $m\in\mathbb Z$, the commutator formula gives
\begin{align}
 [e_m,f_{N-m}]
 &=
 \sum_{i\geq0}\binom{m}{i}(e_if)_{N-i} \notag\\
 &=
 \binom{m}{N}h_0
 +\binom{m}{N+1}(e_{N+1}f)_{-1}
 +\sum_{i=0}^{N-1}\binom{m}{i}(e_if)_{N-i}.
 \label{eq:general-mode-commutator}
\end{align}
Here the terms with $i>N+1$ vanish because $V$ is $\mathbb{N}$-graded. We define the
obstruction operator
\begin{equation}
 \mathcal O_m(e,f)
 :=
 \sum_{i=0}^{N-1}\binom{m}{i}(e_if)_{N-i}.
 \label{eq:obstruction-operator}
\end{equation} 
The following proposition isolates precisely the condition under which the
three relevant modes give an $\mathfrak{sl}_2$-action.

\begin{prop}\label{prop:general-hidden-sl2} Let $ e\in V_p$, $f\in V_s$, $h\in V_1$ such that $h=e_Nf$, $h_0e=re$, $h_0f=-rf$, and $h_1h=d{\bf 1}$, $h_0h=0$. Here, $ N=p+s-2$, $r\neq 0$ and $d\in\mathbb{C}$. 

We fix $m\in\mathbb Z$ such that
$\binom{m}{N}\neq0$. Also, we suppose that $\mathcal O_m(e,f)=0$ and $ e_{N+1}f=\chi{\bf1}$ for some $\chi\in\mathbb C$. We set $$E^m=e_m,~ F^m=\frac{2}{r\binom{m}{N}}\,f_{N-m},
 ~H^m=\frac{2}{r}h_0+
 \frac{2}{r}\frac{\binom{m}{N+1}}{\binom{m}{N}}\chi{\bf 1}_{-1}.$$ Then $[E^m,F^m]=H^m$, $[H^m,E^m]=2E^m$, and $[H^m,F^m]=-2F^m$. Consequently, the operators $E^m,F^m,H^m$ define a representation of
$\mathfrak{sl}_2$ on $V$.
\end{prop}

\begin{proof}
By using the equation \eqref{eq:general-mode-commutator}, together with
$\mathcal O_m(e,f)=0$ and $e_{N+1}f=\chi{\bf1}$, we then have $[E^m,F^m]=H^m$. Because ${\bf 1}_{-1}$ commutes with every operator, the remaining
relations follow from $[h_0,e_m]=(h_0e)_m=re_m$, $[h_0,f_{N-m}]=(h_0f)_{N-m}=-rf_{N-m}$. Thus
$$[H^m,E^m]=2E^m,~ [H^m,F^m]=-2F^m$$ as desired.
\end{proof}

\begin{rem}
The operator $\mathcal O_m(e,f)$ measures the contribution of the lower
products
\[
 e_0f,\ldots,e_{N-1}f.
\]
Thus Proposition~\ref{prop:general-hidden-sl2} separates the
$\mathfrak{sl}_2$ mechanism from the additional products that may obstruct
it.  In the important case $p=0$ and $s=2$, one has $N=0$ and the
obstruction sum is empty.  This is one reason that the
$V_0\oplus V_1\oplus V_2$ case is particularly natural.
\end{rem}

We next record a useful consequence of the Heisenberg highest-weight
condition.

\begin{prop}
\label{prop:pairing-hidden-sl2}
Let $ e\in V_p$, $f\in V_s$, $h\in V_1$ such that $h=e_Nf$, $h_0e=re$, $h_0f=-rf$, and $h_1h=d{\bf 1}$, $h_0h=0$. Here, $ N=p+s-2$, $r\neq 0$ and $d\in\mathbb{C}$. We suppose that for $n>0$, $ h_ne=h_nf=0$. Then $ e_{N+1}f=\frac{d}{r}{\bf 1}$. \end{prop}

\begin{proof}
Because $h=e_Nf$, by commutator formula, we obtain $$d{\bf 1}
 =h_1h
 =h_1(e_Nf)=e_Nh_1f+\sum_{i=0}^{1}\binom{1}{i}(h_ie)_{N+1-i}f=(h_0e)_{N+1}f=r\,e_{N+1}f.$$
The assertion follows.
\end{proof}

We now specialize to the case that will later occur naturally in shifted
lattice vertex operator algebras. In this case, we assume that $e\in V_0$, $h\in V_1$, $f\in V_2$ satisfy
\begin{equation}
 e_0f=h,~ h_0e=2e,~ h_0f=-2f.
 \label{eq:012-sl2-data}
\end{equation}
We set $a=e_1f\in V_0$. The commutator formula gives
\begin{equation}
 [e_{-1},f_1]=\sum_{i=0}^{\infty}(-1)^i(e_if)_{-i}=h_0-a_{-1}.
 \label{eq:EF-012}
\end{equation}
Moreover, $h_0a=h_0e_1f=e_1h_0f+(h_0e)_1f=0$. We let $$V_0(0)=\{v\in V_0\mid h_0v=0\}.$$
\begin{prop}
\label{prop:012-sl2} Let $e\in V_0$, $h\in V_1$, $f\in V_2$ such that $e_0f=h$, $h_0e=2e$, $h_0f=-2f$. Assume $V_0(0)=\mathbb{C}{\bf1}$. Then there is $\kappa\in\mathbb{C}$ such that $e_1f=\kappa{\bf1}$. If
$E=e_{-1}$, $F=f_1$, $H=h_0-\kappa\,{\rm Id}_{V_0}$, then, as endomorphisms of $V_0$,
$$[E,F]=H,\qquad [H,E]=2E,\qquad [H,F]=-2F.$$ Hence $V_0$ is a finite-dimensional $\mathfrak{sl}_2$-module.
\end{prop}

\begin{proof}
Since $a=e_1f$ has $h_0$-weight zero, the hypothesis
$V_0(0)=\mathbb C{\bf1}$ gives $a=\kappa{\bf1}$.  Equation
\eqref{eq:EF-012} therefore gives $[E,F]=H$.  The remaining relations
follow from \eqref{eq:012-sl2-data}.
\end{proof}

This action has a strong consequence for the commutative algebra generated by $e$ in $V_0$.  We write
\[
 e^j=(e_{-1})^j{\bf1}.
\]

\begin{thm}
\label{thm:truncated-polynomial-sl2}
Let $e\in V_0$, $h\in V_1$, $f\in V_2$ such that $e_0f=h$, $h_0e=2e$, $h_0f=-2f$. Assume $V_0(0)=\mathbb{C}{\bf1}$, and $e\neq0$. Let $r$ be maximal such that $e^r\neq0$, and $e^{r+1}=0$. Then $\kappa=r\in\mathbb Z_{>0}$, and $\mathbb {C}[e]\cong\mathbb C[x]/(x^{r+1})$. Moreover, $\operatorname{Span}_{\mathbb{C}}
 \{{\bf1},e,e^2,\ldots,e^r\}$ is the irreducible $\mathfrak{sl}_2$-module of highest weight $r$.
In particular, $\mathbb C[e]$ is a local Artinian Gorenstein algebra with $\operatorname{Soc}(\mathbb{C}[e])=\mathbb{C}e^r$.
\end{thm}

\begin{proof}
Because $V_0$ is finite-dimensional and $h_0e=2e$, the vectors
$e^j$ have pairwise distinct $h_0$-weights $2j$.  Hence $e$ is nilpotent,
and there is a maximal $r$ as above.

The vacuum satisfies $E{\bf1}=e$, $F{\bf1}=0$, $H{\bf1}=-\kappa{\bf1}$. Thus ${\bf1}$ is a lowest-weight vector for the finite-dimensional
$\mathfrak{sl}_2$-module generated by ${\bf1}$. Since $$ E^j{\bf1}=e^j\neq0\quad(0\leq j\leq r),~
 E^{r+1}{\bf1}=0,$$ this module has dimension $r+1$. Finite-dimensional
$\mathfrak{sl}_2$ representation theory implies that its lowest weight is $-r$. Therefore $\kappa=r$. The remaining assertions follow immediately.
\end{proof}

\begin{rem}
Theorem~\ref{thm:truncated-polynomial-sl2} is one of the points where the
vertex operator algebra structure, finite-dimensional
$\mathfrak{sl}_2$ representation theory, and the Gorenstein structure of $V_0$ meet. The nilpotency index of a weight-zero vector is detected by
the highest weight of a hidden $\mathfrak{sl}_2$-module.
\end{rem}

\subsubsection{A rank-one lattice specialization}
\label{subsec:rank-one-specialization}

Let $ L=\mathbb Z\alpha$, $(\alpha,\alpha)=d>0$, and set $ h=\alpha(-1){\bf1}\in V_1$. Suppose that there exist $ e\in V_0$, $f\in V_2$ such that $e_0f=h$, $h_0e=2e$, $h_0f=-2f$, $h_1h=d{\bf1}$. Assume also that $V_0(0)=\mathbb C{\bf1}$.  By
Theorem~\ref{thm:truncated-polynomial-sl2}, we can and will write 
$e_1f=\kappa{\bf1}$, $\mathbb{C}[e]\cong\mathbb C[x]/(x^{\kappa+1})$.

\begin{prop}
\label{prop:lattice-norm-nilpotency} Let $ L=\mathbb Z\alpha$, $(\alpha,\alpha)=d>0$, and set $ h=\alpha(-1){\bf1}\in V_1$. Suppose that there exist $ e\in V_0$, $f\in V_2$ such that $ e_0f=h$, $h_0e=2e$, $h_0f=-2f$, $h_1h=d{\bf1}$. Assume that $V_0(0)=\mathbb C{\bf1}$ and $ e_0(h_1f)=0$. Then $ \kappa=\frac{d}{2}=\frac{(\alpha,\alpha)}{2}$. Consequently, $ \mathbb{C}[e]
 \cong
 \frac{\mathbb{C}[x]}
 { \left(x^{\frac{(\alpha,\alpha)}{2}+1}\right)}$. In particular, $(\alpha,\alpha)$ is even.
\end{prop}

\begin{proof}
Since $h=e_0f$, $ d{\bf1}
 =h_1h
 =h_1(e_0f)
=e_0(h_1f)-(e_0h)_1f$. The relation $h_0e=2e$ implies $e_0h=-2e$. Hence, $ d{\bf1}
 =
 e_0(h_1f)+2e_1f$. By hypothesis, the first term vanishes, so $ d{\bf1}=2\kappa{\bf1}$. Therefore $\kappa=d/2$, and the result follows from
Theorem~\ref{thm:truncated-polynomial-sl2}.
\end{proof}

\begin{rem}
For a norm-two lattice vector, Proposition
\ref{prop:lattice-norm-nilpotency} gives
\[
 \mathbb C[e]\cong\mathbb C[x]/(x^2).
\]
This is exactly the rank-one building block that appears below in shifted
$A_1^{\oplus m}$ examples.
\end{rem}

\subsubsection{Shifted lattice vertex operator algebras}
\label{subsec:shifted-lattice-model}

Let $L$ now be a positive-definite even lattice. We set $\mathfrak h=\mathbb C\otimes_{\mathbb Z}L$. Let $V_L$ be the corresponding lattice vertex operator algebra with
standard conformal vector $\omega_L$.  For $q\in L^\circ$, consider the
shifted conformal vector $$\omega_q=\omega_L+q(-2){\bf1}.$$ Then $$L_q(0)=L_L(0)-q(0),$$
and
\begin{equation}
 \wt_q(e^\beta)
 =
 \frac{(\beta,\beta)}2-(q,\beta)
 \qquad(\beta\in L).
 \label{eq:shifted-lattice-weight}
\end{equation}
The shift does not change the underlying vertex algebra. What changes is
the placement of lattice vectors in the conformal grading.

\begin{prop}
\label{prop:shifted-lattice-section4-hypotheses}
For the space $\mathcal M=\mathfrak h(-1){\bf1}$, the shifted lattice vertex operator algebra $(V_L,\omega_q)$ satisfies the following
properties.
\begin{enumerate}
\item Every $u_0$, $u\in\mathcal M$, acts semisimply.
\item $\Omega_{(V_L,\omega_q)}
 =
 \bigoplus_{\beta\in L}\mathbb{C}e^{\beta}$, and $(\Omega\text{-SS})$ holds.
\item We have that $ P_{(V_L,\omega_q)}=L$, $P_{(V_L,\omega_q)}^{\circ}=L^\circ$. Thus the full-rank integrality condition {\rm(IL)} holds.
\item The usual lattice $2$-cocycle realizes the cocycle compatibility
required in {\rm(LC)}.
\end{enumerate}
\end{prop}

\begin{proof}
For $u\in\mathcal{M}$ and $\beta\in L$, $ u_0e^{\beta}=(u,\beta)e^{\beta}$, and $u_0$ commutes with $v_n$ for all $v\in\mathcal{M}$ and $n<0$.  This proves
(1) and (2). Moreover, $(V_L,\omega_q)(\beta)\neq 0$ for all $\beta\in L$, proving (3).  Finally, the standard normalized lattice cocycle
$\varepsilon$ satisfies $\frac{\varepsilon(\alpha,\beta)}
      {\varepsilon(\beta,\alpha)}
 =(-1)^{(\alpha,\beta)}$, which gives (4).
\end{proof}

\begin{rem}
Thus $(\Omega\text{-SS})$, (IL), and (LC), which are structural assumptions
in Section~4, are automatic in the lattice model.  This is one reason that
shifted lattice vertex operator algebras are useful test examples for the
abstract results of that section.
\end{rem}

\begin{rem}
The quasi-primary hypothesis from Section~4 behaves differently under a
conformal shift. If $\omega_q=\omega_K+q(-2){\bf1}$, then $L_q(1)u(-1){\bf1}=-2(q,u){\bf 1}$. Thus the full Heisenberg space $\mathfrak{h}(-1){\bf 1}$ need not be quasi-primary.  The subspace to
which Section~4 theorem naturally applies is $q^\perp$.  This point is
essential in Example 1 below.
\end{rem}

Now, we will describe how a root of $\mathfrak{sl}_2$ is redistributed by a shift. Let $\alpha\in L$ be a root: $(\alpha,\alpha)=2$. In the standard lattice grading,
$e^\alpha$, $\alpha(-1){\bf1}$, $e^{-\alpha}$ all have weight one and generate the usual root
$\mathfrak{sl}_2\subseteq(V_L)_1$. After shifting by $q$, formula \eqref{eq:shifted-lattice-weight} gives $$\wt_q(e^\alpha)=1-(q,\alpha),~\wt_q(e^{-\alpha})=1+(q,\alpha), \text{ whereas }\wt_q(\alpha(-1){\bf1})=1.$$

\begin{prop}
\label{prop:root-redistribution}
Let $\alpha\in L$ satisfy $(\alpha,\alpha)=2$.
\begin{enumerate}
\item If $(q,\alpha)=0$, then the entire root triple remains in degree one
and $(V_{L,q})_1$ contains a copy of $\mathfrak{sl}_2$.
\item If $(q,\alpha)=1$, then $e^\alpha\in(V_{L,q})_0$, $\alpha(-1){\bf1}\in(V_{L,q})_1$, $e^{-\alpha}\in(V_{L,q})_2$.
\item If $(q,\alpha)=-1$, the roles of $e^\alpha$ and $e^{-\alpha}$ are
reversed.
\end{enumerate}
\end{prop}

\begin{proof}
This follows directly from \eqref{eq:shifted-lattice-weight}.
\end{proof}

\begin{rem} When $(q,\alpha)=1$, the vectors $e=e^{\alpha}$, $h=\alpha(-1){\bf 1}$, $f=e^{-\alpha}$, which form the ordinary root $\mathfrak{sl}_2$ in the original weight-one subspace, are redistributed across $V_0\oplus V_1\oplus V_2$. Nevertheless, the corresponding mode operators recover an $\mathfrak{sl}_2$-action on $V_0$.

This gives a useful comparison with Section~4.  When the lattice vertex operator subalgebra constructed there is conformally embedded, and $V_1$ is solvable, it cannot contain a norm-two lattice vector: such a vector would produce a root $\mathfrak{sl}_2$ contained in $V_1$. A shifted lattice vertex operator algebra behaves differently: the same underlying root vectors may be redistributed among degrees $0$, $1$, and
$2$.  Hence the semisimple Lie algebra disappears from the degree-one
Leibniz algebra while its mode-theoretic action survives.
\end{rem}

\subsubsection{\textbf{Example 1}: The model $A_1^{\oplus m}$ and a Gorenstein weight-zero algebra}
\label{subsec:A1m-model}

\ \

Let $m$ be a positive integer that is greater than or equal to 2. Let $$K=L_{A_1^{\oplus m}}
 =
 \mathbb Z\alpha_1\oplus\cdots\oplus\mathbb Z\alpha_m,~(\alpha_i,\alpha_j)=2\delta_{ij}$$ for $i,j\in\{1,2,...,m\}$. Here, $L_{A_1^{\oplus m}}$ denotes the root lattice of Lie algebra $A_1\oplus \dots\oplus A_1$ ($m$ terms). Also, we set $$\rho=\alpha_1+\cdots+\alpha_m,~
 q=\frac{\rho}{2}.$$ Then $(q,\alpha_i)=1$ for every $i$.  Hence, in the shifted lattice vertex operator algebra $(V_L,\omega_q)$ which we will denote by $V_{K,q}$ from now on, $$ e^{\alpha_i}\in(V_{K,q})_0,~
 \alpha_i(-1){\bf1}\in(V_{K,q})_1,~
 e^{-\alpha_i}\in(V_{K,q})_2.$$ For each $i$, the rank-one hidden $\mathfrak{sl}_2$ construction gives $(e^{\alpha_i})_{-1}^2{\bf1}=0$. The mutually orthogonal roots commute in the degree-zero algebra, and we
obtain the following.

\begin{prop}
\label{prop:A1m-Gorenstein}
The weight-zero subalgebra $A_K$ generated by $ e^{\alpha_1},\ldots,e^{\alpha_m}$ is isomorphic to
$$ A_K
 \cong
 \frac{\mathbb C[x_1,\ldots,x_m]}
 {(x_1^2,\ldots,x_m^2)},~ e^{\alpha_i}\leftrightarrow x_i.
$$ It is a local Artinian Gorenstein algebra with
$\operatorname{Soc}(A_K)=
 \mathbb{C}e^\rho$.
\end{prop}

\begin{proof}
Each $e^{\alpha_i}$ contributes $ \mathbb{C}[x_i]/(x_i^2)$. Orthogonality of the $\alpha_i$ implies that the corresponding weight-zero
products commute and factor independently. Hence, $$A_K\cong
 \bigotimes_{i=1}^m\mathbb C[x_i]/(x_i^2)$$ which is the stated truncated polynomial algebra. This implies that the unique maximal
ideal of $A_K$ is $(e^{\alpha_1},\ldots,e^{\alpha_m})$ and its 1-dimensional socle is generated by
$e^{\alpha_1}_{-1}\dots e^{\alpha_m}_{-1}{\bf 1}=e^{\rho}$.
\end{proof}

\begin{rem} This example gives a concrete connection between the lattice shift and the
Gorenstein structure of the weight-zero algebra: the socle vector
$e^\rho$ records the sum of the root directions used in the shift.\end{rem}

We now compare two different lattices associated with the same shifted
vertex operator algebra $(V_K,\omega_q)$. The underlying lattice is $K=L_{A_1^{\oplus m}} =
 \mathbb Z\alpha_1\oplus\cdots\oplus\mathbb Z\alpha_m,~(\alpha_i,\alpha_j)=2\delta_{ij}$ for $i,j\in\{1,2,...,m\}$, whereas the lattice obtained by applying the construction of Section~4 to the quasi-primary vectors will be denoted by $\mathcal{K}$.  

Recall that $ q=\frac{\rho}{2}$, $\rho=\alpha_1+\cdots+\alpha_m$, and that the shifted Virasoro operators satisfy
$L_q(n)=L_K(n)-(n+1)q(n)$. For $u\in\mathfrak h=\mathbb C\otimes_{\mathbb Z}K$, we have $$L_q(1)u(-1){\bf1}=-2(q,u){\bf1}.$$ Consequently, the quasi-primary Heisenberg space relevant to Section~4 is
\begin{equation}
 \mathcal M
 =q^\perp
 =\rho^\perp
 \subseteq \mathfrak h.
 \label{eq:M-rho-perp}
\end{equation}
In particular, $ \dim\mathcal M=m-1$.

Let $\pi:\mathfrak h\longrightarrow\rho^\perp$ be the orthogonal projection. If $\beta=\sum_{i=1}^m n_i\alpha_i\in K$, $s=\sum_{i=1}^m n_i$, then
\begin{equation}
 \pi(\beta)
 =\beta-\frac{s}{m}\rho
 =\sum_{i=1}^m\left(n_i-\frac{s}{m}\right)\alpha_i.
 \label{eq:projection-K}
\end{equation} Since $\mathcal{M}=\rho^{\perp}$, one may write $\beta\in K$ as $\beta=\beta-\frac{s}{m}\rho+\frac{s}{m}\rho$. Hence for every $u\in\mathcal{M}=\rho^{\perp}$, we have 
$$(u,\beta)=(u, \beta-\frac{s}{m}\rho+\frac{s}{m}\rho)=(u, \beta-\frac{s}{m}\rho)=(u, \pi(\beta)).$$ 
Therefore for $u\in \mathcal{M}$, $u_0e^{\beta}=(u,\pi(\beta))e^{\beta}$. This implies that 
\begin{equation}
 P_{(V_K,\omega_q)}=\pi(K).
 \label{eq:PV-projected-K}
\end{equation}

We set
$$R=K\cap\rho^\perp
 =\left\{
 \sum_{i=1}^m n_i\alpha_i\in K
 \ \middle|\
 \sum_{i=1}^m n_i=0
 \right\}.$$
The integral dual of $P_{(V_K,\omega_q)}$ inside $\rho^\perp$ is particularly simple.

\begin{prop}
\label{prop:dual-projected-charge}
With the notation above, $P_{(V_K,\omega_q)}^\circ=\frac{1}{2} R$. Equivalently,
$$ P_{(V_K,\omega_q)}^\circ
 =\left\{
 \frac{1}{2}\sum_{i=1}^m a_i\alpha_i
 \ \middle|\
 a_i\in\mathbb{Z},\ \sum_{i=1}^m a_i=0
 \right\}.$$
\end{prop}

\begin{proof}
Let $u\in\rho^\perp$. For $\beta\in K$, we have $(u,\beta)=(u,\pi(\beta))$. Thus $u\in P_{(V_K,\omega_q)}^\circ$ if and only if $(u,K)\subseteq\mathbb{Z}$. Because $K^\circ=\frac{1}{2}K$, we obtain
$$P_{(V_K,\omega_q)}^\circ=K^\circ\cap\rho^\perp
 =\frac{1}{2}K\cap\rho^\perp
 =\frac{1}{2}R.$$
\end{proof}

We next calculate the stabilizer lattice appearing in the
$\Delta$-operator construction of Section~4. 
\begin{thm}
\label{thm:Kcal-sqrt2-Am1}
Let $\mathcal{K}
 =\left\{
 u\in P_{(V_K,\omega_q)}^\circ\ \middle|\ V_K^{(u)}\cong V_K
 \text{ as a $(V_K,\omega_q)$-module}
 \right\}$. For the shifted lattice vertex operator algebra
$V=V_{K,\rho/2}$, $\mathcal K=K\cap\rho^\perp=R$. Moreover,$$\mathcal K\cong\sqrt{2}L_{A_{m-1}}.$$ In particular, $\min(\mathcal K)=4$. Here, $L_{A_{m-1}}$ is a root lattice of the Lie algebra $A_{m-1}$.
\end{thm}

\begin{proof}
Let $u\in \mathcal{K}$. Then there is a $V_{K,\omega_q}$-module isomorphism $\phi_u:V_K^{(u)}\rightarrow V_K$ such that $\phi_{u}(Y_{V_K^{(u)}}(v,x)w)=Y_{V_K}(v,x)\phi_{u}(w)$ for all $v\in (V_{K},\omega_q)$ and $w\in V_K^{(u)}$. For $h\in\mathfrak{h}=\mathbb{C}\otimes_{\mathbb{Z}}K$, we have 
\begin{eqnarray*}
\Delta_u(x)h(-1){\bf 1}&=&h(-1){\bf 1}+(u,h)x^{-1}{\bf 1}, \text{ and }\\
h_0^{(u)}&=&h(0)+(u,h){\bf 1}_{-1}. 
\end{eqnarray*} Here, we denote $Y^{(u)}(h(-1){\bf 1},x)$ by $\sum_{n\in\mathbb{Z}}h_n^{(u)}x^{-n-1}$. Hence, $h(0)\phi_u({\bf 1})=(u,h)\phi_u({\bf 1})$. 

We recall that 
$V_K=M(1)\otimes\mathbb{C}\{K\}=\bigoplus_{\lambda\in K}M(1)\otimes e^{\lambda}$ Here, $M(1)$ is a Heisenberg vertex operator algebra that is generated by $\mathfrak{h}=\mathbb{C}\otimes_{\mathbb{Z}}K$. For every $w\in M(1)\otimes e^{\lambda}$, $h(0)w=(h,\lambda)w$. Therefore, 
$V_K(\lambda)=\{w\in V_K~|~h(0)w=(h,\lambda)w\text{ for all }h\in\mathfrak{h}\}$ and $V_K(\lambda)\neq 0$ if and only if $\lambda\in K$. Because $h(0)\phi_u({\bf 1})=(u,h)\phi_u({\bf 1})$, we have $\phi_u({\bf 1})\in V_K(u)$. Consequently, $V_K(u)\neq 0$ and $u\in K$.

Conversely, assume $u\in K\cap \rho^{\perp}$. Since $V_K^{(u)}$ is an irreducible $(V_K,\omega_q)$-module, we then have that $V_K^{(u)}\cong V_{K+\gamma}$ for some $\gamma\in K^{\circ}$. In addition, there is a $(V_K,\omega_q)$-module isomorphism $\psi_u: V_K^{(u)}\rightarrow V_{K+\gamma}$ such that $\psi_{u}(Y_{V_K^{(u)}}(v,x)w)=Y_{V_{K+\gamma}}(v,x)\psi_{u}(w)$ for all $v\in (V_{K},\omega_q)$ and $w\in V_K^{(u)}$. By a calculation similar to the one above, we have $\psi_u({\bf 1})\in V_{K+\gamma}(u)$. Consequently, $V_{K+\gamma}\neq 0$ and $u\in K+\gamma$. However, since $u\in K$, we then have that $\gamma\in K$ and $V^{(u)}_K\cong V_K$ as $(V_K,\omega_q)$-modules as desired and $u\in \mathcal{K}$. In summary, $\mathcal{K}=K\cap \rho^{\perp}.$ Because $P_V^\circ=\frac{1}{2}R$, we obtain
$\mathcal K=P_V^\circ\cap K
 =\frac{1}{2}R\cap K
 =R$.

For $1\leq i\leq m-1$, set $\beta_i=\alpha_i-\alpha_{i+1}$. Then $\{\beta_1,\ldots,\beta_{m-1}\}$ is a $\mathbb Z$-basis of $R$, and
$
 (\beta_i,\beta_i)=4$, $
 (\beta_i,\beta_{i+1})=-2$, while $(\beta_i,\beta_j)=0$ when $|i-j|>1$.
Thus the Gram matrix of $R$ is twice the Cartan matrix of the Lie algebra $A_{m-1}$.  Therefore $R\cong\sqrt{2}L_{A_{m-1}}$, and its minimal norm is $4$. Here, $L_{A_{m-1}}$ is a root lattice of the Lie algebra $A_{m-1}$.
\end{proof}

\begin{rem}
This gives a concrete realization of the lattice theorem of Section~4.
The shifted VOA is built from the lattice $ K=L_{A_1^{\oplus m}}$, which has norm-two vectors.  Those norm-two directions are responsible for
the degree-zero Gorenstein algebra because the shift moves
$e^{\alpha_i}$, $\alpha_i(-1){\bf1}$, $e^{-\alpha_i}$ into degrees $0$, $1$, and $2$, respectively.  Section~4, however, uses
the quasi-primary Heisenberg space $\mathcal M=\rho^\perp$ and therefore
recovers the different lattice $\mathcal K=\sqrt{2}L_{A_{m-1}}$, which has minimum norm $4$.
\end{rem}

\begin{rem}
For $m=2$, $ \mathcal K=\mathbb Z(\alpha_1-\alpha_2)$, $(\alpha_1-\alpha_2,\alpha_1-\alpha_2)=4$. For $m=3$,
$\mathcal K
 =\mathbb Z(\alpha_1-\alpha_2)
 \oplus\mathbb Z(\alpha_2-\alpha_3)
 \cong\sqrt{2} L_{A_2}$, with Gram matrix
$\begin{pmatrix}
 4&-2\\
 -2&4
 \end{pmatrix}$. These low-rank cases make visible the distinction between the original
shifted lattice $K$ and the conformally embedded lattice $\mathcal K$
detected by Section~4.
\end{rem}


\subsubsection{Comparison with Sections~3 and~4}
\label{subsec:section3-4-comparison}

The results of this section suggest the following structural picture.

In Section~3, the presence of a semisimple Lie algebra in $V_1$ leads to an 
affine vertex operator algebra structure.  In Section~4, when $V_1$ is
solvable, the semisimple weight-one alternative is absent and, under the
additional hypotheses considered there, Heisenberg and lattice structures emerge.  The present section shows that solvability of $V_1$ does not
eliminate every trace of semisimple Lie theory.  Instead, an
$\mathfrak{sl}_2$-structure can be distributed across different conformal
weights and recovered through mode operators.

Shifted lattice vertex operator algebras make this mechanism explicit:
\[
 \begin{array}{ccc}
 \text{standard lattice grading}
 &\longrightarrow&
 \text{shifted grading}\\[2mm]
 e^\alpha\in V_1
 &&
 e^\alpha\in V_0\\
 \alpha(-1){\bf1}\in V_1
 &&
 \alpha(-1){\bf1}\in V_1\\
 e^{-\alpha}\in V_1
 &&
 e^{-\alpha}\in V_2.
 \end{array}
\]
Thus an ordinary root $\mathfrak{sl}_2\subseteq(V_L)_1$ becomes a hidden
$\mathfrak{sl}_2$ acting by modes on the shifted theory.

The shifted $A_1^{\oplus m}$ example (i.e. Example 1) makes the relation to Section~4
especially explicit.  The underlying shifted lattice is $ K=L_{A_1^{\oplus m}}$, while the quasi-primary Heisenberg directions are
$\mathcal M=\rho^\perp$.  Relative to this Heisenberg space, the lattice
recovered by the $\Delta$-stabilizer construction of Section~4 is $ \mathcal K=K\cap\rho^\perp\cong\sqrt{2}L_{A_{m-1}}$, and hence $\min(\mathcal K)=4$. Thus the norm-two lattice directions that create the Gorenstein algebra
$V_0$ and the minimum-four lattice detected by Section~4 coexist in the
same shifted vertex operator algebra but play different structural roles.

At the same time, for the quasi-primary Heisenberg subspace to which Section~4 applies, the shifted lattice examples satisfy
$(\Omega\text{-SS})$, (IL), and (LC) automatically.  They therefore
provide concrete models for the abstract hypotheses of Section~4.  The
$A_1^{\oplus m}$ example further shows that the hidden
$\mathfrak{sl}_2$-module structure naturally produces a local Gorenstein
weight-zero algebra whose socle records the direction of the conformal
shift.


\subsection{Invariant bilinear forms in the regular local setting}

As a reminder, $ V=\bigoplus_{n\geq 0}V_n$ is an $\mathbb N$-graded vertex operator algebra. Before turning to the examples, we record a consequence of the standing assumptions of Section~4.  Let
$$
\operatorname{Rad}\langle\,\cdot,\cdot\,\rangle
=
\{v\in V\mid
\langle v,w\rangle=0
\text{ for every }w\in V\}
$$ denote the radical of an invariant bilinear form $\langle\cdot,\cdot\rangle$ on $V$.  The radical of an invariant bilinear form is an ideal of $V$.

Under the standing assumptions of Section~4, Lemma~4.8 shows that $V$ is simple: regularity together with the locality of $V_0$ forces simplicity.  Consequently, if
$\langle\,\cdot,\cdot\,\rangle$ is a nonzero invariant bilinear form on $V$, then $
\operatorname{Rad}\langle\,\cdot,\cdot\,\rangle
$ is an ideal different from $V$.  Simplicity therefore gives $
\operatorname{Rad}\langle\,\cdot,\cdot\,\rangle=0$. Hence every nonzero invariant bilinear form in the setting of Section~4 is automatically nondegenerate.

This observation is useful for interpreting the results of the preceding subsection.  In the regular local setting, degeneracy of a nonzero invariant bilinear form is not an additional phenomenon that can occur independently: it is excluded by simplicity.  Thus, once the invariant form under consideration is known to be nonzero, its nondegeneracy follows from the standing assumptions.

The situation is different outside the regular setting.  If $V$ is not simple, a nonzero invariant bilinear form may have a nonzero radical, and that radical is then a proper ideal of $V$.  It is therefore natural to contrast the regular examples considered below with nonregular vertex operator algebras which contain a positive-definite even lattice vertex operator algebra but admit a degenerate invariant bilinear form.  Such examples lie outside the hypotheses of Section~4, and help indicate the role played there by regularity and simplicity.

Let $V$ be a vertex operator algebra and $M$ a $V$-module.  Recall that the direct sum $V\oplus M$ carries a natural split square-zero extension structure; see \cite[Proposition~4.8.1]{LLi}.  Explicitly, for $u,v\in V$ and $m,n\in M$, one sets
$$
Y_{V\ltimes M}(u+m,z)(v+n)
=
Y_V(u,z)v
+
Y_M(u,z)n
+
e^{zL(-1)}Y_M(v,-z)m.
$$ In particular, $
Y_{V\ltimes M}(m,z)n=0$, $m,n\in M$, so that $M$ is a square-zero ideal.  We apply this construction to a lattice vertex operator algebra in order to obtain a nonregular \(\mathbb N\)-graded vertex operator algebra containing a rational positive-definite even lattice vertex operator algebra and carrying a nonzero degenerate invariant bilinear form.

Let $L$ be a positive definite even lattice. Let $ U=V_L$, $M=V_{L+\lambda}$. Here $\lambda\in L^{\circ}$. We form the Lepowsky–Li split null extension $$ \widetilde V=U\ltimes M=U\oplus M. $$ The vertex algebra structure is defined as above. We retain the conformal vector
$\widetilde\omega=(\omega_U,0)$. Its modes act as the ordinary Virasoro operators of $U$ on the first summand and those of the $U$-module $M$ on the second: $ L_{\widetilde V}(n)|_U=L_U(n)$, $L_{\widetilde V}(n)|_M=L_M(n)$. Therefore, the Virasoro relations and the $L(-1)$-derivative property hold on all of $\widetilde{V}$.

The only issue that is not automatic for an arbitrary $U$-module is the integral grading. If $ M=\bigoplus_{n\ge0}M(n)$ has integral conformal weights compatible with the vertex operator algebra grading, then $$ \widetilde V_n=U_n\oplus M_n $$ gives the required grading.

\subsubsection{\textbf{Example 2}: A nonregular lattice extension with degenerate invariant form}

We next give an example showing that, once regularity and simplicity are dropped, a vertex operator algebra may contain a rational positive-definite even lattice vertex operator algebra while admitting a nonzero degenerate invariant bilinear form.

We let $\varepsilon_1,\dots,\varepsilon_8$ be an orthonormal basis of $\mathbb{R}^8$. For $i\in \{1,\dots,7\}$, we define $\alpha_i=\varepsilon_i-\varepsilon_{i+1}$. We define $\alpha_8=\varepsilon_7+\varepsilon_8$. The lattice $$L_{D_8}=\sum_{i=1}^8\mathbb{Z}\alpha_i=\{\sum_{i=1}^8m_i\varepsilon_i~|~m_i\in\mathbb{Z},\sum_{i=1}^8m_i\in2\mathbb{Z}\}$$ is the root lattice of the simple Lie algebra $D_8$. Next, we set $$
\lambda=\frac{1}{2}(\varepsilon_1+\dots+\varepsilon_8)\in L_{D_8}^\circ\setminus L_{D_8}.$$ Since $
(\lambda,\lambda)=2$, the integrally graded irreducible \(V_{L_{D_8}}\)-module $
M:=V_{L_{D_8}+\lambda}
$ has lowest conformal weight $1$. Therefore, $$
V^{\mathrm{sq}}
=
V_{L_{D_8}}\ltimes V_{L_{D_8}+\lambda}
$$ is an $\mathbb N$-graded vertex operator algebra of CFT-type containing $V_{L_{D_8}}$ as a vertex operator subalgebra.

Consider the split square-zero extension $
V^{\mathrm{sq}}
=
V_{L_{D_8}}\ltimes M
=
V_{L_{D_8}}\oplus M$. For $u,v\in V_{L_{D_8}}$ and $m,n\in M$, we define 
$$
Y_{V^{\mathrm{sq}}}(u+m,z)(v+n)
=
Y_{L_{D_8}}(u,z)v
+
Y_M(u,z)n
+
e^{zL(-1)}Y_M(v,-z)m,
$$ where $Y_M$ denotes the $V_{L_{D_8}}$-module vertex operator on $M$. In particular, for $m,n\in M$, $
Y_{V^{\mathrm{sq}}}(m,z)n=0,
$ so $M$ is a square-zero ideal of $V^{\mathrm{sq}}$. The vacuum and conformal vector are: $
{\bf 1}_{V^{\mathrm{sq}}}=({\bf 1},0),~\omega_{V^{\mathrm{sq}}}=(\omega_{L_{D_8}},0)
$.

Since the lowest conformal weight of $M$ is $1$,
$
V^{\mathrm{sq}}
=
\bigoplus_{n\geq0}V^{\mathrm{sq}}_n$, $V^{\mathrm{sq}}_0=\mathbb{C}\mathbf{1}
$. Thus $V^{\mathrm{sq}}$ is an $\mathbb{N}$-graded vertex operator algebra of CFT-type.  Moreover, $$
V_{L_{D_8}}\hookrightarrow V^{\mathrm{sq}},~u\mapsto u+0,
$$ is a vertex operator algebra embedding.  Hence $V^{\mathrm{sq}}$ contains the rational positive-definite even lattice vertex operator algebra $V_{L_{D_8}}$.

Let $
(\,\cdot,\cdot\,)_{D_8}
$ denote the standard nondegenerate invariant bilinear form on $V_{L_{D_8}}$, and let $
\pi:V^{\mathrm{sq}}\longrightarrow V_{L_{D_8}}$, $\pi(u+m)=u,
$ be the natural quotient homomorphism.  Define a bilinear form on $V^{\mathrm{sq}}$ by $
\langle a,b\rangle_{\mathrm{sq}}
=
\bigl(\pi(a),\pi(b)\bigr)_{D_8}
$. Equivalently, $
\langle u+m,v+n\rangle_{\mathrm{sq}}
=
(u,v)_{D_8}$. 

Since $\pi$ is a vertex operator algebra homomorphism and $(\,\cdot,\cdot\,)_{D_8}$ is invariant, the form
$\langle\,\cdot,\cdot\,\rangle_{\mathrm{sq}}$ is invariant. Clearly, $
\langle M,V^{\mathrm{sq}}\rangle_{\mathrm{sq}}=0,
$ and therefore $
M\subseteq
\operatorname{Rad}
\langle\,\cdot,\cdot\,\rangle_{\mathrm{sq}}$. Conversely, suppose that $u+m$ belongs to the radical.  Then, for every
$v\in V_{L_{D_8}}$, $
0
=
\langle u+m,v\rangle_{\mathrm{sq}}
=
(u,v)_{D_8}$. The nondegeneracy of the standard invariant form on \(V_{D_8}\) implies
\(u=0\). Hence, $$
\operatorname{Rad}
\langle\,\cdot,\cdot\,\rangle_{\mathrm{sq}}
=
M
=
V_{D_8+\lambda}.$$ Thus $V^{\mathrm{sq}}$ admits a nonzero degenerate invariant bilinear form, with $$
V^{\mathrm{sq}}/
\operatorname{Rad}
\langle\,\cdot,\cdot\,\rangle_{\mathrm{sq}}
\cong V_{D_8}.
$$ The vertex operator algebra \(V^{\mathrm{sq}}\) is not regular. Indeed, its adjoint module contains the nonzero proper submodule $
M\subset V^{\mathrm{sq}}$. The resulting exact sequence
$$
0\longrightarrow M
\longrightarrow V^{\mathrm{sq}}
\overset{\pi}{\longrightarrow}V_{D_8}
\longrightarrow0
$$ does not split as a sequence of \(V^{\mathrm{sq}}\)-modules. To see this, suppose that $$
V^{\mathrm{sq}}=M\oplus W
$$ for a $V^{\mathrm{sq}}$-submodule $W$.  Since the image of $W$ in the quotient is $V_{D_8}$, there is an element $
w=\mathbf1+m_0\in W$ for some $m_0\in M$. For any $m\in M$, the square-zero property gives $
m_{-1}m_0=0$, whereas the vacuum property gives $
m_{-1}\mathbf1=m$. Therefore $
m_{-1}w=m\in W$. Since \(m\in M\) and \(M\cap W=0\), this forces \(m=0\), a contradiction for nonzero \(m\in M\).  Hence the adjoint \(V^{\mathrm{sq}}\)-module is not completely reducible, and \(V^{\mathrm{sq}}\) is not regular.

This example has a useful comparison with the ordinary lattice construction.  The coset \(L_{D_8}+\lambda\) is the spinor coset which extends \(L_{D_8}\) to the even unimodular lattice $$
L_{E_8}=L_{D_8}\cup(L_{D_8}+\lambda)
$$ which is the root lattice of the simple Lie algebra $E_8$. Consequently, the same $V_{L_{D_8}}$-module decomposition $
V_{L_{D_8}}\oplus V_{L_{D_8}+\lambda}
$ also underlies the lattice vertex operator algebra
$$
V_{L_{E_8}}
=
V_{L_{D_8}}\oplus V_{L_{D_8}+\lambda}.
$$ The two extensions, however, have fundamentally different multiplication.  In $V_{L_{E_8}}$, products of vectors in $V_{L_{D_8}+\lambda}$ are nonzero and return to $V_{L_{D_8}}$, producing a simple regular lattice vertex operator algebra.  In $V^{\mathrm{sq}}$, by contrast, $
Y(M,z)M=0$, so $M$ is a nonzero square-zero ideal and is precisely the radical of the invariant form.

Finally, degeneracy of the invariant bilinear form does not by itself force a nonsemisimple Heisenberg action.  For $h\in\mathbb{C}\otimes_{\mathbb{Z}}L_{D_8}$ and $\beta\in L_{D_8}+\lambda$, the usual lattice-module action gives $$
h_0e^{\beta}=(h,\beta)e^{\beta}.
$$ Thus the Heisenberg zero modes act semisimply on both $V_{L_{D_8}}$ and $V_{L_{D_8}+\lambda}$, and hence on $V^{\mathrm{sq}}$. Therefore, this example separates two phenomena which should not be conflated: 
\begin{center}
\fbox{%
\parbox{0.82\textwidth}{%
\centering
Degeneracy of an invariant bilinear form does not\\
imply nonsemisimplicity of the Heisenberg zero-mode action.
}}
\end{center}

\subsection{Semisimple and nonsemisimple Heisenberg actions}

A second purpose of Section 5 is independent of the preceding discussion.  For the subspace $\mathcal M$ constructed in Section~4, Proposition \ref{heisenberg-jordan-derivations} shows that, for $u\in\mathcal M$, the Jordan decomposition $
u_0=S_u+N_u
$ has the property that the nilpotent part commutes with all modes $u_n$ for all $u\in \mathcal{M},~n\in\mathbb{Z}$ and vanishes on the Heisenberg vertex operator algebra $M(1)$ that is generated by $\mathcal{M}$. Thus, any failure of semisimplicity of $u_0$ is carried entirely by the Heisenberg vacuum space $
\Omega_V
=
\{w\in V\mid v_nw=0
\text{ for all }v\in\mathcal M,\ n>0\}.
$ Equivalently, the obstruction to semisimplicity is the possible nonvanishing of $
N_u|_{\Omega_V}.
$ This motivates the condition $(\Omega\text{-SS})$ introduced in Section~4.

The shifted lattice examples considered earlier satisfy $(\Omega\text{-SS})$. In the family obtained from the conformally shifted lattice $
L_{A_1^{\oplus m}},
$ the quasi-primary space $\mathcal{M}$ is the orthogonal complement of the shift direction, and the lattice detected by the construction of Section~4 is $
L_{\sqrt{2}A_{m-1}}.
$ In particular, the resulting vertex operator algebra contains the rational positive-definite even lattice vertex operator algebra $
V_{\sqrt{2}L_{A_{m-1}}}.
$ The recovered lattice has minimum norm $4$, even though the original lattice $L_{A_1^{\oplus m}}$ contains norm-two vectors.

These same examples also exhibit a complementary phenomenon.  The norm-two directions of the original lattice are not simply lost under the conformal shift.  Rather, their associated vectors are redistributed among different homogeneous subspaces.  In the rank-one situation, this produces vectors $
e\in V_0$, $h\in V_1$, $f\in V_2
$ whose appropriate modes generate an $\mathfrak{sl}_2$-action, even though this $\mathfrak{sl}_2$ is not contained in \(V_1\). Thus, solvability of $V_1$ can exclude a semisimple Lie algebra of degree one without eliminating its mode-theoretic shadow.

The following examples focus on a different issue: whether regularity forces the zero modes of $\mathcal{M}$ to act semisimply. Example 3 shows that regularity alone does not suffice. Example 4 gives a regular shifted lattice vertex operator algebra of CFT-type with solvable $V_1$ for which the zero modes of generators of a Heisenberg vertex operator subalgebra do not act semisimply. Example 5 shows that the same obstruction persists in the non-CFT-type setting where $V_0$ is a nontrivial local graded Gorenstein algebra and $V_1$ is a solvable Leibniz algebra. 


\subsubsection{\textbf{Example 3}: A regular example with nonsemisimple Heisenberg zero-mode action}

For $ k\in\mathbb Z_{>0}$, we let $
V=L_{\widehat{\mathfrak{sl}_3}}(k,0)$, be the simple affine vertex operator algebra of positive integral level.  Then \(V\) is regular and $V_1\cong\mathfrak{sl}_3$. We identify $V_1$ with $\mathfrak{sl}_3$ and use the invariant bilinear form induced by the affine vertex operator algebra.

Consider $$
s=\operatorname{diag}(1,1,-2),
\qquad
n=E_{12},
\qquad
u=s+n.
$$ Since $[s,n]=0$, the element $u$ is not semisimple: its Jordan decomposition in $\mathfrak{sl}_3$ has nonzero semisimple and nilpotent parts $s$ and $n$, respectively.

On the other hand, $u$ is nonisotropic.  Indeed, invariance of the form and the root-space decomposition give $
(s,n)=0$, $(n,n)=0$, and hence $
(u,u)=(s,s)\neq0.
$ The affine commutation relations therefore imply
$$
[u_m,u_n]
=
m\,k(u,u)\delta_{m+n,0}\operatorname{Id}.
$$ After rescaling $u$, if necessary, the vertex operator subalgebra generated by $u$ is a rank-one Heisenberg vertex operator algebra: $
\langle u\rangle\cong M(1).
$ The zero mode $u_0$, however, does not act semisimply on $V$. Already on $
V_1\cong\mathfrak{sl}_3$, its action is $
u_0x=[u,x]=\operatorname{ad}(u)x.
$

To see the failure of semisimplicity explicitly, set $$
e=E_{12},\qquad
f=E_{21},\qquad
h=E_{11}-E_{22}.
$$ Since $s$ commutes with $e,f,h$, the restriction of $u_0$ to the $\mathfrak{sl}_2$-subspace $
\operatorname{Span}\{e,h,f\}
$ is simply $\operatorname{ad}(e)$. In particular, $
u_0f=h$, $u_0h=-2e$, $u_0e=0.
$

Thus $f\longmapsto h\longmapsto-2e\longmapsto0
$ is a nontrivial nilpotent Jordan chain. Consequently, $
u_0|_{V_1}
$ is not semisimple, and therefore $u_0$ is not semisimple on $V$.

Hence the regular vertex operator algebra $L_{\widehat{\mathfrak{sl}_3}}(k,0)$ contains a rank-one Heisenberg vertex operator algebra $M(1)$ whose zero mode does not act semisimply on the ambient vertex operator algebra.  In particular,
$$
\boxed{
\text{regularity of }V
\quad\not\Longrightarrow\quad
\text{semisimplicity of an arbitrary Heisenberg zero-mode action}.
}
$$ 
This example also clarifies the role of the condition $(\Omega\text{-SS})$ introduced in Section~4. The usual Heisenberg subalgebra associated with a Cartan subalgebra of $\mathfrak{sl}_3$ acts semisimply on the affine vertex operator algebra.  The failure above occurs because we instead choose the nonsemisimple but nonisotropic element $
u=s+E_{12}$. Thus semisimplicity of Heisenberg zero modes depends not only on regularity of the ambient vertex operator algebra, but also on the Heisenberg subspace under consideration.

We emphasize, however, that this example does not satisfy all of the standing hypotheses of Section~4.  Indeed, $
V_1\cong\mathfrak{sl}_3
$ is semisimple rather than solvable. Therefore, the example does not show that $(\Omega\text{-SS})$ is independent of the other hypotheses of Section~4.  It shows instead that regularity alone cannot imply $(\Omega\text{-SS})$. This leaves the more restrictive question of whether nonsemisimple Heisenberg zero-mode action can occur for a regular $\mathbb{N}$-graded vertex operator algebra with solvable $V_1$.

\subsubsection{\textbf{Example 4}: A regular shifted lattice example with nonsemisimple Heisenberg zero-mode action}

Let $
L=\mathbb Z\alpha\oplus\mathbb Z\gamma
$ be the positive-definite even lattice with $
(\alpha,\alpha)=(\gamma,\gamma)=4$, $(\alpha,\gamma)=0
$, and let $
q=\frac{\alpha}{4}\in L^\circ.
$

Let us consider the shifted lattice vertex operator algebra $
V=V_{L,q}$, with conformal vector $
\omega_q=\omega_L+q(-2){\bf 1}$.

By Dong and Mason \cite{DM2}, \(V_{L,q}\) is regular. For $
\lambda=m\alpha+n\gamma\in L$, the shifted conformal weight is $
\wt_q(e^{\lambda})=\frac{(\lambda,\lambda)}2-(q,\lambda)=
2m^2+2n^2-m
$. This quantity is nonnegative for all $m,n\in\mathbb {Z}$, and it vanishes only for $
m=n=0$. Consequently, $
V_0=\mathbb{C}{\bf 1}
$. Thus this particular example is of CFT-type. The weight-one subspace is $
V_1
=
\mathbb{C}\alpha(-1){\bf 1}
\oplus
\mathbb{C}\gamma(-1){\bf 1}
\oplus
\mathbb{C}e^{\alpha}$. Indeed, $
\wt_q(e^{\alpha})
=
\frac{(\alpha,\alpha)}{2}-(q,\alpha)
=
2-1
=
1
$,
and there are no other nonzero lattice vectors of shifted conformal weight at most $1$.

We set $
a=\alpha(-1){\bf 1}$, $
g=\gamma(-1){\bf 1}$, $
b=e^{\alpha}$. The zero-mode products satisfy $$
a_0g=0,~
g_0b=0,~
a_0b=(\alpha,\alpha)b=4b,
~
b_0b=0.
$$ Hence $
[V_1,V_1]=\mathbb{C}b$, $[[V_1,V_1],[V_1,V_1]]=0$. Therefore $
V_1
$ is a solvable Lie algebra, isomorphic to a direct sum of a two-dimensional nonabelian solvable Lie algebra and a one-dimensional center.

Now define $
u=g+b
=
\gamma(-1){\bf 1}+e^{\alpha}$. Since $(\gamma,\alpha)=0$, we have $
u_1u
=
g_1g
=
(\gamma,\gamma)\mathbf1
=
4{\bf 1}$. Moreover, $
u_0u=0.
$ Since $u\in V_1$, all products $u_nu$ with $n\geq 2$ vanish.  It follows that the modes of $u$ satisfy the rank-one Heisenberg relations
$$
[u_m,u_n]
=
4m\,\delta_{m+n,0}{\bf 1}_{-1}.
$$ After normalizing $u$, the vertex operator subalgebra generated by $u$ is therefore a rank-one Heisenberg vertex operator algebra: $
\langle u\rangle\cong M(1).
$

The vector $u$ is also quasi-primary with respect to the shifted conformal vector. Indeed, $$
L_q(1)\gamma(-1)\mathbf1
=
-2(q,\gamma)\mathbf1
=
0,
$$ while $e^{\alpha}$ is a shifted highest-weight lattice vector, so $
L_q(1)e^\alpha=0$. Thus $
L_q(1)u=0
$.

We now show that $u_0$ is not semisimple on $V$. We have $
u_0=\gamma_0+b_0
$ and $
[\gamma_0,b_0]=0$. Let us consider $\Ker~ \gamma_0$. On this subspace $
\gamma_0=0,
$ and hence $
u_0=b_0=(e^\alpha)_0$. The operator $b_0$ is nonzero.  For example, using the standard lattice vertex operator formula, $
Y(e^\alpha,z)e^{-\alpha}
=
\varepsilon(\alpha,-\alpha)z^{-4}
\exp\left(
\sum_{r\ge1}\frac{\alpha(-r)}r z^r
\right){\bf 1}.
$ Taking the coefficient of $z^{-1}$ gives $
(e^{\alpha})_0e^{-{\alpha}}
=
\varepsilon(\alpha,-\alpha)
\left(
\frac{\alpha(-3)}{3}
+
\frac{\alpha(-2)\alpha(-1)}{2}
+
\frac{\alpha(-1)^3}{6}
\right){\bf 1},
$ which is nonzero.

On the other hand, $b_0$ preserves shifted conformal weight, because $b=e^{\alpha}\in V_1$. Because $V=(V_L,\omega_q)=\bigoplus_{\beta\in L}M(1)\otimes e^{\beta}$ and $b_0(M(1)\otimes e^{\beta})\subseteq M(1)\otimes e^{\beta+\alpha}$, we then have that for $v\in V_n\cap (M(1)\otimes e^{\beta})$, $(b_0)^rv\in V_n\cap(M(1)\otimes e^{\beta+r\alpha})$. However, $\dim_{\mathbb{C}}V_n<\infty$. Hence, there are only finitely many $M(1)\otimes e^{\beta+s\alpha}$ such that $V_n\cap (M(1)\otimes e^{\beta+s\alpha})\neq \{0\}$. Consequently, there is a positive integer $r$ such that $b_0^rv=0$ and $b_0$ is locally nilpotent. On the $\Ker~ \gamma_0$, $
u_0=b_0
$ is a nonzero locally nilpotent operator.  In particular, $
u_0
$ is not semisimple on $V$.  We have therefore obtained
$$
\boxed{
\begin{gathered}
V_{L,q}\text{ is regular of CFT-type}; V_1\text{ is solvable};
\langle u\rangle\cong M(1); u_0\text{ does not act semisimply on }V.
\end{gathered}}
$$
This example lies outside the strongly regular framework. Indeed $2q=\frac{\alpha}{2}$ is not in $L$, so Dong and Mason's criterion in \cite{DM2} shows that $(V_L,\omega_q)$ is not self-contragredient. Therefore, it demonstrates that regularity, even together with solvability of $V_1$, does not guarantee semisimplicity of the zero-mode of a chosen Heisenberg subspace. The next example places the same obstruction in the non-CFT-type setting of Section 4.


\subsubsection{\textbf{Example 5}: A regular non-CFT-type example with solvable $V_1$ and nonsemisimple Heisenberg action}

We now give an example showing that the condition $(\Omega\text{-SS})$ of Section~4 is a genuine additional hypothesis even in the regular non-CFT-type setting with $V_1$ solvable and $V_0$ a local graded-Gorenstein algebra of dimension greater than one.

Let $
L=
\mathbb{Z}\beta\oplus
\mathbb{Z}\alpha\oplus
\mathbb{Z}\gamma
$ be the positive-definite even lattice with orthogonal basis $\{\beta,\alpha,\gamma\}$ satisfying $
(\beta,\beta)=2$, $(\alpha,\alpha)=(\gamma,\gamma)=4$.

We set $
q=\frac{\beta}{2}+\frac{\alpha}{4}\in L^\circ
$ and consider the shifted lattice vertex operator algebra $
V=V_{L,q}$, $\omega_q=\omega_L+q(-2){\bf 1}$. By the shifted lattice theory of Dong and Mason \cite{DM2}, $V_{L,q}$ is regular.

For $
\lambda=m\beta+n\alpha+r\gamma\in L$, the shifted conformal weight is
$$
\begin{aligned}
\wt_q(e^\lambda)
&=
\frac{(\lambda,\lambda)}2-(q,\lambda)\\
&=
m(m-1)+2n^2-n+2r^2.
\end{aligned}
$$ This is a nonnegative integer for every $m,n,r\in\mathbb{Z}$. The weight-zero equation $
m(m-1)+2n^2-n+2r^2=0
$ has exactly two solutions $
(m,n,r)=(0,0,0)$, $(1,0,0)$. Consequently, $
V_0=\mathbb{C}{\bf 1}\oplus\mathbb{C}e^{\beta}$.

If $
b=e^{\beta}$, then the product on $V_0$ satisfies
$
b_{-1}b=0$, and hence $V_0\cong\mathbb C[x]/(x^2)$. In particular, $
\dim V_0=2$, $\mathfrak{m}=\mathbb{C}b$, $\operatorname{Soc}(V_0)=\mathbb{C}b$, so $V_0$ is a local graded-Gorenstein algebra. Moreover, $
L(-1)b=\beta(-1)e^{\beta}\neq0,
$ and therefore $
\Ker L(-1)|_{V_0}=\mathbb{C}{\bf 1}$.

It is useful to factor the shifted lattice vertex operator algebra as $$
V
\cong
W\otimes U,
$$ where $
W=V_{\mathbb{Z}\beta,\beta/2}$, $U=
V_{\mathbb{Z}\alpha\oplus\mathbb{Z}\gamma,\frac{\alpha}{4}}
$. Then $
W_0=\mathbb{C}{\bf 1}\oplus\mathbb{C}b
$ and $
W_1
=
\mathbb{C}h\oplus\mathbb{C}d$, where $
h=\beta(-1){\bf 1}$, $d=L(-1)x=\beta(-1)e^{\beta}$.

For the second factor, $
U_0=\mathbb{C}{\bf 1}
$ and $
U_1
=
\mathbb{C}a\oplus\mathbb{C}g\oplus\mathbb{C}e$, where
$
a=\alpha(-1){\bf 1}$, $g=\gamma(-1){\bf 1}$, $e=e^{\alpha}$. The only nonzero bracket in $U_1$, up to skew-symmetry, is $
a_0e=4e$. Thus $U_1$ is solvable.

Since $
V_1
=
(W_1\otimes U_0)
\oplus
(W_0\otimes U_1)$, we obtain $
V_1
=
\operatorname{Span}
\{h,d,a,g,e,ba,bg,be\}$, where tensor signs are suppressed. We define $
I=\mathbb{C}d\oplus bU_1$. Then $I$ is an ideal of $V_1$. Indeed, $
h_0b=2b$, $h_0d=2d$, while $
d_0=0
$ because $d=L(-1)b$ and $
(L(-1)b)_0=0$.

Furthermore, for $u,v\in U_1$, the zero-mode products involving $bu$ remain in $
\mathbb{C}d+bU_1$. In particular, $
(bu)_0(bv)=0$. This follows from the factor $z^{(\beta,\beta)}=z^2$ in the lattice operator product $
Y(e^\beta,z)e^\beta$, which removes the possible singular terms arising from the product of the weight-one vectors $u$ and $v$. Hence, $
[I,I]=0$, so $I$ is abelian.

When we quotient $V_1$ by $I$, we have that $
V_1/I
\cong
\mathbb{C}h\oplus U_1$, and $h$ commutes with $U_1$.  Since $U_1$ is solvable, $V_1/I$ is solvable.  Therefore $V_1$ is solvable. Thus $V_1$ is a solvable Leibniz algebra.

We next construct a Heisenberg vertex operator subalgebra whose zero mode fails to act semisimply.  We set $
u=g+e
=
\gamma(-1){\bf 1}+e^{\alpha}
$ and $
\mathcal{M}=\mathbb{C}u$. Since $
(\gamma,\alpha)=0$, we have that $u_0u=0
$ and $
u_1u
=
g_1g
=
(\gamma,\gamma){\bf 1}
=
4{\bf 1}$. Moreover, for $n\geq 2$, we have $
u_nu=0$. It follows that the modes of $u$ satisfy
$
[u_m,u_n]
=
4m\delta_{m+n,0}{\bf 1}_{-1}$.
 Consequently, $\langle u\rangle\cong M(1)$ is a rank-one Heisenberg vertex operator algebra.

In addition, $u$ is quasi-primary with respect to the shifted conformal vector. Indeed, $$
L_q(1)\gamma(-1){\bf 1}
=
-2(q,\gamma){\bf 1}
=
0,
$$
and $e^{\alpha}$ is a shifted lattice highest-weight vector of conformal weight one. Thus $
L_q(1)u=0$.

Let $B(~,~)$ be the Frobenius form on $
V_0\cong\mathbb C[x]/(x^2)
$ normalized by $
B(\mathbf1,b)=1$. Since $b$ spans the socle of $V_0$, the bilinear form of Section~4 satisfies
$$
((u,u))
=
B(u_1u,b)
=
B(4\mathbf1,b)
=
4.
$$ Therefore, $
\mathcal M\in\mathfrak M$.

We now show that $u_0$ is not semisimple. We have $
u_0=g_0+e_0$, $e=e^{\alpha}$, and $
[g_0,e_0]=0$. On $\Ker~\gamma_0$-subspace, $
g_0=0$, so that $
u_0=e_0$. The operator $e_0$ is nonzero. For example, $
Y(e^\alpha,z)e^{-\alpha}
=
\varepsilon(\alpha,-\alpha)z^{-4}
\exp\left(
\sum_{j\ge1}\frac{\alpha(-j)}{j}z^j
\right){\bf 1}$, and therefore
$$
(e^{\alpha})_0e^{-\alpha}
=
\varepsilon(\alpha,-\alpha)
\left(
\frac{\alpha(-3)}{3}
+
\frac{\alpha(-2)\alpha(-1)}{2}
+
\frac{\alpha(-1)^3}{6}
\right){\bf 1}
\neq 0.$$

By an argument similar to that in Example 4, one can show that $e_0$ is locally nilpotent.  Thus $u_0$ has a nonzero nilpotent part and is not semisimple on $V$.

We can locate this obstruction more precisely on the Heisenberg vacuum space $
\Omega_V
=
\{w\in V\mid u_nw=0\text{ for all }n>0\}$. Let $
u_0=S_u+N_u
$ be the Jordan decomposition. By Proposition~4.10, for $n\in\mathbb{Z}$, $
[N_u,u_n]=0$. Moreover, $
V=U(\widehat{\mathcal M}_{-})\Omega_V$.

Since $N_u\neq 0$ on $V$, it cannot vanish identically on $\Omega_V$. Indeed, if $
N_u|_{\Omega_V}=0$, then for every $
w\in\Omega_V
$ and every monomial $a\in U(\widehat{\mathcal M}_{-})$, $
N_u(aw)
=
aN_uw
=
0$, which would imply $N_u=0$ on all of $V$, a contradiction. Therefore $N_u|_{\Omega_V}\neq 0$. Equivalently, $V$ does not satisfy $(\Omega\text{-SS})$ for the Heisenberg subspace
$\mathcal M=\mathbb{C}u$.

Finally, let $
\mathfrak{a}
=
\operatorname{Span}
\{v_0a\mid v\in V_1,\ a\in\mathfrak m\}$. Since $
\mathfrak{m}=\mathbb{C}b
$ and $
h_0b=2b$, we obtain $
\mathfrak{a}=\mathbb{C}b=\mathfrak{m}$. Consequently, $
V_0\neq\mathfrak a$. We have therefore constructed a regular $\mathbb{N}$-graded vertex operator algebra satisfying $
\dim V_0>1$, $V_0$ local graded-Gorenstein, $V_1$ solvable, together with a nondegenerate quasi-primary Heisenberg subspace $
\mathcal M\in\mathfrak M
$ such that $
\langle\mathcal M\rangle\cong M(1)
$ but $
\mathcal M
$ does not act semisimply on $V$. In particular, 
\begin{center}
\fbox{%
\parbox{0.82\textwidth}{%
\centering
$(\Omega\text{-SS})$ is not a consequence of regularity,
solvability of $V_1$,\\
or the graded-Gorenstein local structure of $V_0$.
}}
\end{center}
Thus the condition $(\Omega\text{-SS})$ in Section~4 is an independent structural assumption.


\end{document}